\documentclass{amsart}
\usepackage{amsmath,amsthm,amsfonts,amssymb,xcolor}
\usepackage[all]{xy}
\usepackage{graphicx}
\usepackage{mathrsfs}
\usepackage{enumitem}
\usepackage{etoolbox}
\apptocmd{\sloppy}{\hbadness 10000\relax}{}{}
\apptocmd{\sloppy}{\vbadness 10000\relax}{}{}
\usepackage{hyperref}
\usepackage[letterpaper,margin=1.1in]{geometry}

\numberwithin{equation}{section}
\theoremstyle{plain}

\newtheorem{thm}{Theorem}[section]
\newtheorem{prop}[thm]{Proposition}

\newtheorem{lem}[thm]{Lemma}
\newtheorem{claim}[thm]{Claim}

\theoremstyle{definition}
\newtheorem{rem}[thm]{Remark}
\newtheorem{definition}[thm]{Definition}
\newtheorem{ex}[thm]{Example}
\newtheorem{question}[thm]{Question}

\def\LL{\mathcal{L}}
\def\R{\mathbb{R}}

\def\N{\mathbb{N}}

\def\e{\epsilon}
\def\g{\gamma}
\def\d{\delta}
\def\D{\Delta}

\def\G{\Gamma}

\newcommand{\diam}{\mathop\mathrm{diam}\nolimits}
\newcommand{\dist}{\mathop\mathrm{dist}\nolimits}

\newcommand{\cA}{\mathcal{A}}

\newcommand{\card}{\operatorname{card}}

\begin{document}

\title{Bi-Lipschitz geometry of quasiconformal trees} 

\author{Guy C. David \and Vyron Vellis}
\thanks{G.~C.~David was partially supported by NSF DMS grants 1758709 and 2054004. V.~Vellis was partially supported by NSF DMS grants 1800731 and 1952510.}
\date{\today}

\address{Department of Mathematical Sciences\\ Ball State University\\ Muncie, IN 47306}
\email{gcdavid@bsu.edu}
\address{Department of Mathematics\\ The University of Tennessee\\ Knoxville, TN 37966}
\email{vvellis@utk.edu}

\begin{abstract}
A quasiconformal tree is a doubling metric tree in which the diameter of each arc is bounded above by a fixed multiple of the distance between its endpoints. We study the geometry of these trees in two directions. First, we construct a catalog of metric trees in a purely combinatorial way, and show that every quasiconformal tree is bi-Lipschitz equivalent to one of the trees in our catalog. This is inspired by results of Herron-Meyer and Rohde for quasi-arcs. Second, we show that a quasiconformal tree bi-Lipschitz embeds in a Euclidean space if and only if its set of leaves admits such an embedding. In particular, all quasi-arcs bi-Lipschitz embed into some Euclidean space.
\end{abstract}

\maketitle

\section{Introduction}\label{sec:intro}

In this paper, a \textit{(metric) tree} is a compact, connected, locally connected metric space with the property that each pair of distinct points forms the endpoints of a unique arc. In some sense, trees make up the simplest class of one-dimensional continua, and are ubiquitous in analysis and geometry.

Within the class of all trees, an important role has been played by the class of \textit{quasiconformal trees} studied in \cite{BM, BM2, Kinneberg}. By definition, these are trees $T$ that satisfy two simple geometric properties:
\begin{itemize}
\item $T$ is \textit{doubling}: there is a constant $N$ such that each ball in $T$ can be covered by $N$ balls of half the radius.
\item $T$ is \textit{bounded turning}: there is a constant $C$ such that each pair of points $x,y \in T$ can be joined by a continuum whose diameter is at most $Cd(x,y)$.
\end{itemize}
These conditions are both invariant under quasisymmetric mappings, making the class of quasiconformal trees a natural quasisymmetrically invariant class. We do not recall the definition of quasisymmetric mappings here (see \cite{Heinonen} or \cite{BM}), but merely note that they are an important generalization of conformal mappings to arbitrary metric spaces.

Quasiconformal trees appear in several fields of analysis. For instance, Julia sets of semihyperbolic polynomials (e.g. $z^2+i$) are quasiconformal trees (see \cite{CJY} and \cite[p. 95]{CG-dynamics}), and quasiconformal trees $T$ in $\R^2$ (often called \emph{Gehring trees}), were recently characterized by Rohde and Lin \cite{RohdeLin} in terms of the laminations of the conformal map $f: \mathbb{C}\setminus\mathbb{D} \to \mathbb{C}\setminus T$.

Quasiconformal trees generalize two more well-known types of spaces. For one, quasiconformal trees that are simply topological arcs (i.e., have no branching) are called \textit{quasi-arcs}, and have been studied in complex analysis and analysis on metric spaces for decades \cite{GH}. For example, the famous von Koch snowflake is a quasi-arc. A well-known result of Tukia and V\"ais\"al\"a \cite{TuVa} shows that quasi-arcs are exactly those spaces that are quasisymmetrically equivalent to the unit interval $[0,1]$.

Quasiconformal trees also generalize (doubling) \textit{geodesic trees}. Geodesic trees are trees in which, for each pair of points $x,y$, the unique arc joining them has (finite) length equal to $d(x,y)$. Thus, in geodesic trees all paths are ``straight'' (isometric to intervals in the real line), whereas paths in quasiconformal trees may be fractal, like the von Koch snowflake. Geodesic trees are of course standard objects of study in many parts of mathematics and computer science. Recently, Bonk and Meyer \cite{BM} generalized the result of Tukia and V\"ais\"al\"a mentioned above by showing that each quasiconformal tree is quasisymmetric to a geodesic tree.

Rather than studying the quasisymmetric geometry of quasiconformal trees, this paper is concerned with the finer notion of bi-Lipschitz geometry. Recall that a mapping $f$ between two metric spaces is called \textit{bi-Lipschitz} (or $L$-bi-Lipschitz to emphasize the constant) if there is a constant $L\geq 1$ such that
$$ L^{-1} d(x,y) \leq d(f(x),f(y)) \leq L d(x,y), \quad \text{ for all } x,y\in X.$$
Thus, bi-Lipschitz mappings preserve distances up to constant factors. All bi-Lipschitz mappings are quasisymmetric, but the converse is false. For example, one may parametrize the von Koch snowflake $K$ by a quasisymmetric map $[0,1]\rightarrow K$, but not by a bi-Lipschitz map.

Given a metric space $X$, natural questions in the bi-Lipschitz world are: 
\begin{itemize}
\item \emph{Uniformization}: Which metric spaces are bi-Lipschitz equivalent to $X$, i.e., admit a surjective bi-Lipschitz mapping onto $X$?
\item \emph{Embeddability}: Does $X$ admit a bi-Lipschitz embedding into some Euclidean space $\R^n$, i.e., a bi-Lipschitz mapping from $X$ into $\R^n$?
\end{itemize}
The first of these questions is about recognizing or providing models for spaces up to bi-Lipschitz equivalence, i.e., up to bounded distortion of their metrics. The second is about understanding which spaces can be viewed as subsets of Euclidean space up to bounded distortion, and in complete generality is a major problem in analysis, geometry, and computer science \cite{Heinonen_embed, Ostrovskii}.

We study both of these questions for quasiconformal trees. Concerning the first, we give a ``combinatorial model'' for generating quasiconformal trees based on a purely discrete construction, and then show that every quasiconformal tree is bi-Lipschitz equivalent to one of our combinatorial constructions. This is in the vein of the combinatorial models for quasi-arcs up to bi-Lipschitz equivalence given by Rohde \cite{Rohde} and Herron-Meyer \cite{HM}, although the construction for trees is more elaborate. Our main theorem on this topic is Theorem \ref{thm:maincombthm}.

Concerning the second question, we build on ideas from \cite{RV} to show that every \textbf{quasi-arc} admits a bi-Lipschitz embedding into some Euclidean space, and use this to show that the bi-Lipschitz embedding properties of quasiconformal trees are completely controlled by their sets of \textit{leaves} (Theorem \ref{thm:mainembed}). We leave open the main question of whether all quasiconformal trees admit bi-Lipschitz embeddings into Euclidean space; see below for additional background and discussion.

We now discuss these ideas in more detail.

\subsection{Combinatorial models for quasiconformal trees up to bi-Lipschitz equivalence}\label{sec:introcomb}
We first give a way to define metric spaces using certain sequences of combinatorial graphs, that is, $G=(V,E)$ where $V$ is the vertex set and $E$ is the edge set. This is inspired by the ideas of \cite{Rohde, HM} concerning quasi-arcs, with a number of new wrinkles in the case of trees. To simplify the presentation as much as possible, a number of definitions are postponed until Section \ref{sec:prelim}.

Let $A$ be an ``alphabet'': a set of the form $\{1,\dots,n\}$ or $A=\N$. Denote by $\varepsilon$ the empty word and by $|w|$ the length of a word, i.e., the number of letters. Let $A^0 = \{\varepsilon\}$, and for each $k\in\N$ denote by $A^k$ the set of all words made from the alphabet $A$ of length exactly $k$. Define the set of finite words
\[ A^* = \bigcup_{k=0}^{\infty}A^k.\]
Denote also by $A^{\N}$ the set of infinite words formed by the alphabet $A$, and $A^\N_u\subseteq A^\N$ the set of all infinite words that begin with a given finite word $u\in A^*$.

\begin{definition}\label{def:combdata}
We consider the following \textbf{combinatorial data} $\mathscr{C} = (A,(G_k)_{k\in\N})$ where:
\begin{enumerate}
\item\label{eq:combdata1} $A$ is a finite or infinite alphabet: $A=\{1,\dots, M\}$ for some integer $M\geq 2$, or $A=\mathbb{N}$.
\item\label{eq:combdata2} For each $k\in\mathbb{N}$, $G_k=(A^k, E_k)$ is a connected combinatorial graph on the vertex set $A^k$ with the following properties:
\begin{enumerate}
\item\label{eq:combdata2a} For each $w\in A^k$ the subgraph of $G_{k+1}$ induced by the vertex set $\{wi: i\in A\}$ is connected.
\item\label{eq:combdata2b} If $\{w,u\}\in E_k$, then there is a pair $(i,j)\in A\times A$ such that $\{wi ,uj\}\in E_{k+1}$. 
\end{enumerate}
\end{enumerate}
\end{definition}

We next define a way to ``move between'' different infinite word sets $A^\N_u$ using the structure of the combinatorial data. Moves between $A^\N_u$ and $A^\N_v$ are always permitted if $u$ and $v$ are adjacent words of equal length, but in general we take into account the full scope of the combinatorial data.

Thus, given combinatorial data $\mathscr{C} = (A,(G_k)_{k\in\N})$, we say that two infinite word sets $A^\N_{u_1}$ and $A^\N_{u_2}$ \textit{combinatorially intersect}, and write $A^\N_{u_1} \wedge_{\mathscr{C}} A^\N_{u_2} \neq \emptyset$, if the following holds:
\begin{align}
&\text{For each } n > \max\{|u_1|,|u_2|\}, \text{ there exist words } w_1, w_2\in A^n\text{, beginning with } u_1 \text{ and } u_2\text{,}\label{eq:combintersection}\\&\text{respectively, that are adjacent in } G_n.\nonumber
\end{align}

In other words, two word sets $A^{\N}_{u_1}$ and $A^{\N}_{u_2}$ combinatorially intersect if their restrictions to every sufficiently large finite level are adjacent. Below, in Definition \ref{def:combintersection}, we will give a precise definition of the set $A^\N_{u_1} \wedge_{\mathscr{C}} A^\N_{u_2}$ and show that its non-emptiness is equivalent to \eqref{eq:combintersection}.

Given this notion of combinatorial intersection, we can describe how to move between two infinite words:
\begin{definition}\label{def:chain}
Given two words $w,w' \in A^{\N}$ we say that $\{A^{\N}_{w_1},\dots, A^{\N}_{w_n}\}$ is a \emph{chain joining $w$ with $w'$} if $w\in A^{\N}_{w_1}$, $w'\in A^{\N}_{w_n}$ and for every $i=1,\dots,n-1$, we have $A^{\N}_{w_i}\wedge_{\mathscr{C}}A^{\N}_{w_{i+1}} \neq \emptyset$.
\end{definition}

Now that we have a way to move between two infinite words, we can define a distance on $A^\N$ by assigning costs to each chain with a ``diameter function'':

\begin{definition}\label{def:diam}
Given an alphabet $A$, a \textit{diameter function} is a function $\D:A^* \to [0,1]$ that satisfies:
\begin{enumerate}
\item\label{eq:diam1} $\D(\varepsilon)=1$;
\item\label{eq:diam3} for each $w\in A^k$ and $i\in A$, $\Delta(wi)=0$ for all but finitely many $i\in A$;
\item\label{eq:diam4} $\lim_{n\to \infty}\max\{\D(w): w\in A^n\} = 0.$
\end{enumerate} 
The class of all diameter functions on $A$ is defined by $\mathscr{D}(A)$. Given $0< \delta_1 \leq \delta_2 \leq 1$ and finite $A$, we denote by $\mathscr{D}(A,\d_1, \d_2)$ the collection of all diameter functions on the alphabet $A$ such that,
\[\text{for each $w\in A^*$ and $i,j \in A$}, \qquad \D(wi)=\D(wj)\quad \text{and}\quad\frac{\Delta(wi)}{\Delta(w)}\in\{\delta_1, \delta_2\}.\]
\end{definition}

Note that, in Definition \ref{def:diam}, \eqref{eq:diam3} is automatic if $A$ is finite, and \eqref{eq:diam4} is automatic if $\D \in \mathscr{D}(A,\d_1,\d_2)$ and $\delta_2<1$. In \eqref{eq:diam4}, condition \eqref{eq:diam3} implies that the maximum is actually achieved, even if $A$ is infinite. 

Given combinatorial data $\mathscr{C} = (A,(G_k)_{k\in\N})$ and $\D\in \mathscr{D}(A)$, we define a pseudometric $D_{\mathscr{C},\D}$ on $A^\mathbb{N}$ by:
\begin{equation}\label{eq:pseudometric}
 D_{\mathscr{C},\D}(w, u) = \inf \sum_{i=0}^N \Delta(v_i)
\end{equation}
where the infimum is taken over all chains $\{A^\N_{v_i}\}$ joining $w$ with $u$.  

We prove in Lemma \ref{lem:pseudometric} that $D_{\mathscr{C},\D}$ is indeed always a pseudometric on $A^\N$. Taking the quotient space $\mathcal{A} :=A^\N/\sim$ under the equivalence relation $w \sim w'$ whenever$D_{\mathscr{C},\D}(w,w')=0$, we obtain a metric space
$$ (\mathcal{A}, d_{\mathscr{C},\D}), $$
where $d_{\mathscr{C},\D}([w],[v]) = D_{\mathscr{C},\D}(w,v)$.

To help digest the definition, we provide a number of examples illustrating this combinatorial construction in Section \ref{sec:examples} below.

Our main theorem on these combinatorial models is as follows:

\begin{thm}\label{thm:maincombthm}
\hfill
\begin{enumerate}
\item\label{eq:maincomb1} If $\mathscr{C}$ defines combinatorial data and $\Delta\in \mathscr{D}(A)$, then the space $(\mathcal{A}, d_{\mathscr{C},\D})$ is compact, connected, and bounded turning with constant $C$=1.
\item\label{eq:maincomb2} If in addition each graph $G_k$ in the combinatorial data is a combinatorial tree, then the space $(\mathcal{A}, d_{\mathscr{C},\D})$ is a metric tree.
\item\label{eq:maincomb3} Conversely, if $X$ is an arbitrary quasiconformal tree, then there exist combinatorial data $\mathscr{C}=(A,(G_k)_{k\in\N})$ and a diameter function $\D\in\mathscr{D}(A,K_1, K_2)$ such that each $G_k$ is a combinatorial tree and $X$ is bi-Lipschitz equivalent to the space $(\mathcal{A}, d_{\mathscr{C},\D})$. The choice of alphabet, the constants $K_1$ and $K_2$, and the bi-Lipschitz constant depend only on the doubling and bounded turning constants of $X$, and $\diam(X)$.
\end{enumerate}
\end{thm}
Parts \eqref{eq:maincomb1} and \eqref{eq:maincomb2} of Theorem \ref{thm:maincombthm} are proven in Proposition \ref{prop:comb-tree}, and part \eqref{eq:maincomb3} is proven (with a more detailed statement) in Theorem \ref{thm:main}.

We emphasize that an important feature of Theorem \ref{thm:maincombthm} is that all quasiconformal trees are built (up to bi-Lipschitz equivalence) not only from combinatorial objects but from the simple \emph{homogeneous} word sets $A^\mathbb{N}$ and the additional data provided by $\{G_k\}$ and the diameter function. In some sense, one can view the construction in \cite{HM}, which combinatorially builds bi-Lipschitz models of all quasi-\emph{arcs}, as being a special case of the above construction in the case where $A$ has two elements and the graphs $G_k$ are combinatorial arcs, and so we show that the above re-interpretation and expansion of their construction yields all quasiconformal \emph{trees} up to bi-Lipschitz equivalence. Later, in Section \ref{sec:examples}, we provide some concrete examples and pictures of the combinatorial construction described above, including describing in more detail how quasi-arcs fit into our picture.

\begin{rem}
The metric space $(\mathcal{A}, d_{\mathscr{C},\D})$ constructed from given combinatorial data and diameter function need not be doubling in general, even if the alphabet $A$ is finite, the graphs $G_k$ are all combinatorial trees, and the diameter function $\D$ lies in $\mathscr{D}(A,\delta_1,\delta_2)$ for $0<\delta_1<\delta_2<1$.

However, in Proposition \ref{prop:doubling} we give some sufficient conditions that imply that the space $(\mathcal{A}, d_{\mathscr{C},\D})$ is doubling. In Theorem \ref{thm:maincombthm}\eqref{eq:maincomb3}, the space $(\mathcal{A}, d_{\mathscr{C},\D})$ that we construct always satisfies these conditions. This is stated explicitly in Theorem \ref{thm:main}.
\end{rem}

\subsection{Combinatorial descriptions of metric spaces with good tilings} 
Some techniques in the proof of Theorem \ref{thm:maincombthm}(3) can be used for a more general class of metric spaces that can be tiled in a uniform fashion. Roughly speaking, we say that a metric space has a ``good tiling'' if there exists an alphabet $A$, a constant $r\in(0,1)$ and a tiling decomposition $\{X_w : w\in A^*\}$ of $X$ such that each tile $X_w$ has diameter comparable to $r^{|w|}$ and any two non-intersecting tiles $X_w$, $X_u$ have distance at least a fixed multiple of $\max\{r^{|w|},r^{|u|}\}$. See Section \ref{sec:gencombdata} for a precise definition.

In Proposition \ref{prop:goodtiles} we show that any such space is bi-Lipschitz equivalent to a space $(\mathcal{A},d_{\mathscr{C},\D})$ for some combinatorial data $\mathscr{C}$ and $\D(w)=r^{|w|}$. Spaces with good tilings include many attractors of iterated function systems such as the square, the cube, the Sierpi\'nski carpet, the Sierpi\'nski gasket and others; see Example \ref{ex:gasket} for further discussion.

We note that Proposition \ref{prop:goodtiles} is not a generalization of Theorem \ref{thm:maincombthm}: if $X$ is a quasiconformal tree, the combinatorial data that Proposition \ref{prop:goodtiles} will provide may not consist of combinatorial trees, as required by Theorem \ref{thm:maincombthm}. The proof also proceeds differently, and in fact we do not know if every quasiconformal tree possesses a good tiling in the sense given in Section \ref{sec:gencombdata}.

\subsection{Bi-Lipschitz embeddings of quasi-arcs and quasiconformal trees}\label{sec:introembed}
We now turn our attention to the problem of finding bi-Lipschitz embeddings of quasiconformal trees into Euclidean space. The most natural question is:
\begin{question}\label{q:embed}
Does every quasiconformal tree have a bi-Lipschitz embedding into some Euclidean space $\R^n$?
\end{question}
We do not answer this question here and, indeed, it may be rather difficult. In the special case of doubling, \textit{geodesic} trees, the answer is known to be positive, by a theorem of Gupta-Krauthgammer-Lee \cite{GKL}; see also \cite[Corollary 8]{GT11}. Lee-Naor-Peres also give an alternative proof of the result for geodesic trees in \cite[Theorem 2.12]{LNP}.
% I reconsidered what I wrote slightly. I think for our purposes the ``unweighted'' result in GKL is already sufficient, though I will keep the reference to GT11.
% In the special case of doubling, \textit{geodesic} trees, the answer is known to be positive, by combining a construction for unweighted trees due to Gupta-Krauthgammer-Lee \cite{GKL} with a theorem of Gupta-Talwar \cite{GT11}. (See \cite[Corollary 8]{GT11}.) Lee-Naor-Peres also give an alternative proof of the result for geodesic trees in \cite[Theorem 2.12]{LNP}.

By adapting techniques of Romney and the second named author, we make progress on Question \ref{q:embed} in the case where the quasiconformal tree has no branching:
\begin{prop}\label{prop:introqcircles}
Every quasi-arc admits a bi-Lipschitz embedding into some Euclidean space $\R^n$.
\end{prop}

Proposition \ref{prop:introqcircles} is a simplified version of Proposition \ref{prop:qcircles} below, where we identify the sharp dimension $n$ for the embedding. We note that Herron and Meyer proved Proposition \ref{prop:introqcircles} in the special case of quasi-arcs with Assouad dimension less than 2; see \cite[Theorem C]{HM}.

Using Proposition \ref{prop:qcircles}, and results of Lang-Plaut \cite{LP} and Seo \cite{Seo}, we end by giving a criterion that can answer Question \ref{q:embed} in certain examples. If $X$ is a metric tree, we denote by $\LL(X)$ be the set of \textit{leaves} of $X$, i.e.,
$$ \LL(X) := \{x\in X: X\setminus \{x\} \text{ is connected}\}.$$

\begin{thm}\label{thm:mainembed}
A quasiconformal tree $X$ admits a bi-Lipschitz embedding into some Euclidean space if and only if $\LL(X)$ admits a bi-Lipschitz embedding into some Euclidean space.
\end{thm}

Theorem \ref{thm:mainembed} is a simplified version of the quantitative statement of Theorem \ref{thm:mainembed2}.

\begin{rem}
If $X$ is a quasiconformal tree, the set $\LL(X)$ need not be closed and may even be dense in $X$. Thus, Theorem \ref{thm:mainembed} does not necessarily always reduce Question \ref{q:embed} to a simpler problem.

In many particular cases, however, it may be significantly easier to check embeddability of $\LL(X)$ than $X$ itself. For example, in many concrete settings, the leaf set $\LL(X)$ is an \textit{ultrametric space}, and every doubling ultrametric space bi-Lipschitz embeds into some Euclidean space \cite{LML}.
\end{rem}

\begin{rem}
An equivalent reformulation of Theorem \ref{thm:mainembed} is that a subset $E$ of a quasiconformal tree $X$ admits a bi-Lipschitz embedding into some Euclidean space if and only if the minimal sub-tree of $X$ containing $E$ does.
\end{rem}

\subsection{Outline of the paper}
In Section \ref{sec:prelim}, we review some elementary notions from graph theory concerning combinatorial graphs and trees. In Section \ref{sec:combtrees}, we provide more details on our combinatorial models and prove parts \eqref{eq:maincomb1} and \eqref{eq:maincomb2} of Theorem \ref{thm:maincombthm}. In Section \ref{sec:doubling}, we work in the case of combinatorial trees and identify conditions on $A$, $\mathscr{C}$, and $\D$ that guarantee that the metric tree $(\mathcal{A}, d_{\mathscr{C},\D})$ is doubling. 

In Section \ref{sec:characterization} we prove a more detailed version of part \eqref{eq:maincomb3} of Theorem \ref{thm:maincombthm}. The basic idea is to construct an $n$-adic decomposition $(X_w)_{w\in \{1,\dots,n\}^*}$ of a given quasiconformal tree $X$, for some $n\geq 2$ that satisfies the following properties:
\begin{enumerate}
\item Each $X_w$ is the union of its children $X_{w1},\dots,X_{wn}$, which are themselves trees. Each two of the children intersect in at most one point, which has valency 2 in $X$.
\item Each child $X_{wi}$ of $X_w$ has diameter comparable to that of $X_w$.
\item Any two points $x,y$ on $X_w\cap \overline{X\setminus X_w}$ have distance comparable to the diameter of $X_w$.
\end{enumerate}
This is accomplished by performing certain subdivisions and gluings on top of a construction of Bonk and Meyer \cite{BM}. Once we have such a decomposition, we can build combinatorial data $\mathscr{C}$ and a diameter function $\D$ such that $(\mathcal{A}, d_{\mathscr{C},\D})$ is bi-Lipschitz equivalent to $X$.

Section \ref{sec:examples} contains some examples and pictures that illustrate how our combinatorial data yields metric spaces in a few concrete cases.

Section \ref{sec:gencombdata} considers more general metric spaces, not necessarily trees, that admit a notion of ``good tiling''. We show that such spaces can also be viewed from our combinatorial data, in a slightly different way than Theorem \ref{thm:maincombthm}. In particular, we describe how some self-similar spaces like the unit square and the Sierpi\'nski gasket can be constructed in our framework.

Finally, in Section \ref{sec:embed}, we prove a quantitative version of Proposition \ref{prop:introqcircles} and then apply a bi-Lipschitz welding result of Lang and Plaut \cite{LP} and a bi-Lipschitz embedding characterization of Seo \cite{Seo} to complete the proof of Theorem \ref{thm:mainembed}. 

\section{Preliminaries}\label{sec:prelim}
In this section, we introduce some further preliminary definitions and results related to the combinatorial models defined in Section \ref{sec:introcomb}.

\subsection{Words}

Recall from Section \ref{sec:introcomb} that we start with an alphabet $A = \{1, \dots, M\}$, for some integer $M\geq 2$, or $A=\N$. In addition to the sets $A^*$, $A^\N$, $A^\N_u$ defined above, we also set a few other pieces of notation. For $w\in A^*$ and $k\geq |w|$ define
\[ A^k_w = \{wu : u\in A^{k-|w|}\}, \qquad A^*_w = \{wu : u\in A^{*}\}.\]

Given $n\in\N$ and $w\in A^{\N}$ denote by $w(n)$ the unique word $u\in A^n$ such that $w=uw'$ for some $w'\in A^{\N}$. Similarly, if $n\in\mathbb{N}$ and $w\in A^{*}$, $w(n)$ denotes the initial subword of $w$ with length $n$, and we set $w(n)=w$ if $n\geq |w|$.

Finally, given $k\in\N$ and $u \in A^k$ denote by $u^{\uparrow}$ the unique element of $A^{k-1}$ such that $u \in A^k_{u^{\uparrow}}$.

\subsection{Combinatorial graphs and trees}
Definition \ref{def:combdata} above uses some graph theory terminology. A \emph{combinatorial graph} is a pair $G=(V,E)$ of a finite or countable vertex set $V$ and an edge set 
\[E \subset \{\{v,v'\} : v,v' \in V\text{ and }v\neq v'\}.\]
If $\{v,v'\}\in E$, we say that the vertices $v$ and $v'$ are \textit{adjacent} in $G$.

A combinatorial graph $G' = (V',E')$ is a \emph{subgraph} of $G=(V,E)$ (and we write $G\subset G'$) if $V'\subset V$ and $E'\subset E$. We commonly generate subgraphs of $G=(V,E)$ by starting with a vertex set $V'\subset V$ and considering the \textit{subgraph of $G$ induced by $V'$}: the graph $G'=(V',E')$ where $E'$ is the set of all edges between two vertices of $V'$.

A \emph{path} in $G$ is a set $ \g = \{\{v_1,v_2\}, \dots, \{v_{n-1},v_n\}\} \subset E$; in this case we say that $\g$ joins $v_1$, $v_n$. The path $ \g = \{\{v_1,v_2\}, \dots, \{v_{n-1},v_n\}\}$ is a \emph{combinatorial arc} or \emph{simple path} if for all $i, j \in \{1,\dots,n\}$, $v_i = v_j$ if and only if $i=j$; in this case we say that the endpoints of the arc $\g$ are the points $v_1,v_n$. A combinatorial graph $G = (V,E)$ is connected, if for any distinct $v,v' \in V$ there exists a path $\g$ in $G$ that joins $v$ with $v'$. A \emph{component} of a combinatorial graph $G$ is a maximal connected subgraph of $G$.

A graph $T = (V,E)$ is a \emph{combinatorial tree} if for any distinct $v,v'$ there exists unique combinatorial arc $\g$ whose endpoints are $v$ and $v'$. Given a combinatorial tree $T = (V,E)$ and a point $v\in V$, define the valencies 
\[ \text{Val}(T,v) := \card\{e \in E : v\in e\} \qquad\text{and}\qquad  \text{Val}(T) := \max_{v\in V} \text{Val}(T,v)\]
and the set of leaves $\text{Leaves}(T) := \{v\in V : \text{Val}(T,v) = 1\}$. Here $\card$ denotes the cardinality of a finite or countable set, taking values in $\mathbb{N} \cup \{\infty\}$.

Given a combinatorial graph $G=(V,E)$ and a vertex $v\in V$, we write $G\setminus\{v\}$  to be the subgraph of $G$ induced by $V\setminus\{v\}$. Note that, if $T$ is a tree, then every component of $T\setminus\{v\}$ is a tree.

\section{A model for bounded turning metric spaces and trees}\label{sec:combtrees}

\subsection{Combinatorial data}
Recall the notion of combinatorial data $\mathscr{C} = (A, (G_k)_{k\in\N})$ from Definition \ref{def:combdata}, where $A$ is an alphabet and $G_k=(A^k,E_k)$ are combinatorial graphs on the vertex sets $A^k$, satisfying certain axioms. \textbf{For the remainder of Section \ref{sec:combtrees}, we fix combinatorial data $\mathscr{C}= (A, (G_k)_{k\in\N})$.}

Our first lemma gives some basic structural properties of these graphs. In particular, if each $G_k$ is a combinatorial tree,  then the pair $(i,j) \in A\times A$ of Definition \ref{def:combdata}\eqref{eq:combdata2b} is unique.

\begin{lem}\label{lem:adjacent}
Let $k\geq j$ and $v\neq w\in A^j$.
\begin{enumerate}
\item If $v$ and $w$ are adjacent in $G_j$, then there are words $v'$ and $w'$ in $A^{k-j}$ such that $vv'$ and $ww'$ are adjacent in $G_k$.
\item If $G_k$ is a combinatorial tree and there are words $v'$ and $w'$ in $A^{k-j}$ such that $vv'$ and $ww'$ are adjacent in $G_k$, then $v$ and $w$ are adjacent in $G_j$.
\item If $G_k$ is a combinatorial tree and $v$ and $w$ are adjacent in $G_j$, then there is a unique pair of words $(v', w')$ in $A^{k-j} \times A^{k-j}$ such that $vv'$ and $ww'$ are adjacent in $G_k$.
\end{enumerate}
\end{lem}
\begin{proof}
The first statement is an immediate consequence of \eqref{eq:combdata2b} in Definition \ref{def:combdata}, and induction on $k-j$. 

For the second, suppose that $v$ and $w$ were not adjacent in $G_j$, under these assumptions.

Let $v=u_0, u_1, \dots, u_{n-1}, u_n=w$ be a path from $v$ to $w$ in $G_j$. Note that $n\geq 2$. Then, by the first statement in the lemma and part (2) of Definition \ref{def:combdata}, there is a simple path from $vv'\in A^k_v$ to $ww'\in A^k_w$ in $G_k$ of the form
$$ \text{ elements of } A^k_{u_0}, \text{ elements of } A^k_{u_1}, \dots, \text{ elements of } A^k_{u_n}.$$
On the other hand, there is also an adjacency between  $vv'$ and $ww'$ in $G_k$. This contradicts the assumption that $G_k$ is a tree.

For the third claim, the existence of $v'$ and $w'$ follows from (1). Suppose that the uniqueness failed. We consider the following two possible cases.

Suppose first that there are two distinct $v',v'' \in A$ and $w'\in A$ such that both $vv'$ and $vv''$ are adjacent to $ww'$. Then there exists two combinatorial arcs in $G_{k}$ that join $vv'$ with $vv''$; one through the vertices of $G_{k}$ restricted on $A^{k}_v$ (by \eqref{eq:combdata2a} in Definition \ref{def:combdata}) and another is $\{\{vv',ww'\},\{ww', vv''\}\}$. This contradicts the fact that $G_{k}$ is a tree.

The other possibility is that there are two distinct $v',v'' \in A$ and two distinct $w',w''\in A$ such that $vv'$ is adjacent to $ww'$, and $vv''$ are adjacent to $ww''$. Then there exist two combinatorial arcs in $G_{k}$ that join $vv'$ with $vv''$; one through the vertices of $G_{k}$ restricted on $A^{k}_v$ (by (2a) in Definition \ref{def:combdata}) and another through the vertices of $G_{k}$ restricted on $A^{k}_w$ along with edges $\{vv',ww''\}$ and $\{vv'',ww''\}$. This again contradicts the fact that $G_{k}$ is a tree.
\end{proof}

\subsection{Combinatorial intersection and chains}

Recall the notion of combinatorial intersection $A^\N_u \wedge_\mathscr{C} A^\N_v$ defined in \eqref{eq:combintersection} in Section \ref{sec:introcomb}. There, we only defined what it means for this set to be non-empty, but here we actually give a meaning to the set itself.

\begin{definition}\label{def:combintersection}
Given $u_1,u_2 \in A^*$, define
\begin{align} 
A^{\N}_{u_1} \wedge_{\mathscr{C}} A^{\N}_{u_2} := &\{w \in A^{\N}_{u_1} : \forall n > \max\{|u_1|,|u_2|\} \text{ there exists  $u\in A^n_{u_2}$ with $\{w(n),u\}\in E_n\}$}\\
\cup &\{w \in A^{\N}_{u_2} : \forall n >\max\{|u_1|,|u_2|\} \text{ there exists  $u\in A^n_{u_1}$ with $\{w(n),u\}\in E_n\}$}.\nonumber
\end{align}
The set $A^{\N}_{u_1} \wedge_{\mathscr{C}} A^{\N}_{u_2}$ is called the \textit{combinatorial intersection} of $A^{\N}_{u_1}$ and $A^{\N}_{u_2}$.
\end{definition}

We now show that this definition agrees with that in \eqref{eq:combintersection}, and give an equivalent reformulation in the case of trees.
\begin{lem}\label{lem:intersection1}
Let $u_1,u_2 \in A^*$. The following are equivalent.
\begin{enumerate}
\item The set $A^{\N}_{u_1}\wedge_{\mathscr{C}}A^{\N}_{u_2}$ is non-empty.
\item For every $k>\max\{|u_1|,|u_2|\}$ there exists $v_1\in A^k_{u_1}$ and $v_2 \in A^k_{u_2}$ such that $\{v_1,v_2\} \in E_k$.
\end{enumerate}
If each graph $G_k$ is a combinatorial tree, then (1) and (2) are also equivalent to the following.
\begin{enumerate}
\item[(3)] There exists $k> \max\{|u_1|,|u_2|\}$ and $v_1\in A^k_{u_1}$, $v_2 \in A^k_{u_2}$ such that $\{v_1,v_2\} \in E_k$.
\end{enumerate}
\end{lem}

\begin{proof}
We start by showing the equivalence of (1) and (2). That (1) implies (2) follows immediately from the definition of $A^{\N}_{u_1}\wedge_{\mathscr{C}}A^{\N}_{u_2}$. 

To show that (2) implies (1), we will inductively construct elements of $A^{\N}_{u_1}\wedge_{\mathscr{C}}A^{\N}_{u_2}$. Let $k_0=\max\{|u_1|,|u_2|\}$ and choose $u_1i_1 \in A^{k_0+1}_{u_1},u_2 j_1 \in A^{k_0+1}_{u_2}$ such that $\{u_1i_1,u_2j_1\} \in E_{k_0+1}$. By \eqref{eq:combdata2b} in Definition \ref{def:combdata}, given that $\{u_1i_1\cdots i_{n-k}, u_2j_1\cdots j_{n-k}\} \in E_n$ for some $n\geq k+1$, there exist $i_{n-k+1},j_{n-k+1}\in A$ such that $\{u_1i_1\cdots i_{n-k+1}, u_2j_1\cdots j_{n-k+1}\} \in E_{n+1}$. Set now 
\[ w_1 = u_1i_1i_2\cdots  \quad\text{and}\quad w_2=u_2j_1j_2\cdots\]
and note that both $w_1$ and $w_2$ are in $A^{\N}_{u_1}\wedge_{\mathscr{C}}A^{\N}_{u_2}$.

Assume now that each graph $G_k$ is a combinatorial tree. Clearly (2) implies (3) so it suffices to show that (3) implies (2). Assume there is an integer $k_0\geq  \max\{|u_1|,|u_2|\}$ and words $w_1 \in A_{u_1}^{k_0}$ and $w_2\in A^{k_0}_{u_2}$ such that $\{w_1,w_2\} \in E_{k_0}$. If $k\geq k_0$, then by Lemma \ref{lem:adjacent}(1), there exist $v_1 \in A^{k}_{w_1}$ and $v_2\in A^{k}_{w_2}$ (hence $v_1 \in A^{k}_{u_1}$ and $v_2\in A^k_{u_2}$) such that $\{v_1,v_2\}\in E_k$. If $k$ is an integer with $ \max\{|u_1|,|u_2|\} \leq k \leq k_0$, then by Lemma \ref{lem:adjacent}(2), there exist $v_1 \in A^{k}_{u_1}$ and $v_2\in A^{k}_{u_2}$ such that $w_1\in A^k_{v_1}$, $w_2\in A^k_{v_2}$ and $\{v_1,v_2\}\in E_k$. Therefore, (2) holds. 
\end{proof}

The next lemma gives a description of the set $A^{\N}_{u_1}\wedge_{\mathscr{C}}A^{\N}_{u_2}$ in the case that each $G_k$ is a combinatorial tree. 

\begin{lem}\label{lem:intersection2}
Let $u_1,u_2 \in A^*$ with $|u_1|\leq |u_2|$, let $k_1=|u_1|$ and let $u'_2 = u_2(k_1)$. 
\begin{enumerate}
\item If $u_2'=u_1$ (that is, $u_2 \in A^{*}_{u_1}$), then $A^{\N}_{u_2} \subset A^{\N}_{u_1}\wedge_{\mathscr{C}}A^{\N}_{u_2}$.
\end{enumerate}
Suppose additionally that each $G_k$ is a combinatorial tree. Then:
\begin{enumerate}
\item[(2)] If $A^{\N}_{u_1}\wedge_{\mathscr{C}}A^{\N}_{u_2} \neq \emptyset$, then either $\{u_1,u_2'\} \in E_{k_1}$ or $u_1=u_2'$. 
\item[(3)] If $\{u_1,u_2'\} \in E_{k_1}$, then $A^{\N}_{u_1}\wedge_{\mathscr{C}}A^{\N}_{u_2}$ contains exactly two elements; one in $A^{\N}_{u_1}$ and one in $A^{\N}_{u_2}$. The converse is also true.
\end{enumerate}
\end{lem}

\begin{proof}
Let $u_1,u_2,v \in A^*$ and $k_1\in\N$ be as in the statement and let $k_2=|u_2|$. 

To prove (1), assume that $u_2'=u_1$, that is, $u_2 \in A^{k_2}_{u_1}$. Let $w\in A^{\N}_{u_2}$. By Definition \ref{def:combdata}(2a), the subgraph of $G_{k_2+1}$ induced by $A^{k_2+1}_{u_1}$ is connected. Fix $v\in A^{k_2+1}_{u_1}$ adjacent to $w(k_2+1)$. Applying Definition \ref{def:combdata}(2b) we find a sequence $\{i_1,i_2\dots\} \subset A$ such that for each $n\in\N$, $vi_1\cdots i_n$ is adjacent to $w(k_2+n+1)$. Since $vi_1\cdots i_n \in A^{k_2+n+1}_{u_1}$, by definition, $w \in A^{\N}_{u_1}\wedge_{\mathscr{C}}A^{\N}_{u_2}$.

Assume now for the rest of the proof that each $G_k$ is a combinatorial tree. To prove (2), assume that $A^{\N}_{u_1}\wedge_{\mathscr{C}}A^{\N}_{u_2} \neq \emptyset$. By Lemma \ref{lem:intersection1}(2), we have that there exists $v_1\in A^{k_2+1}_{u_1}$ and $v_2 \in A^{k_2+1}_{u_2}$ such that $\{v_1,v_2\} \in E_{k_2+1}$. Applying Lemma \ref{lem:adjacent}(2), we have that either $u_1=u_2'$ or $\{u_1,u_2'\} \in E_{k_1}$. 

To prove (3), assume that $\{u_1,u_2'\} \in E_{k_1}$ and let $v_1$ and $v_2$ be as in the proof of (2). That is, $v_1 \in A^{k_2+1}_{u_1}$, $v_2\in A^{k_2+1}_{u_2}$, and $\{v_1,v_2\}\in E_{k_2+1}$. By Definition (2b) of \ref{def:combdata}, there exist $i_1,i_2,\dots \in A$ and $j_1,j_2,\dots,\in A$ such that for all $m\in\N$, $\{v_1i_1\cdots i_m, v_2j_1\cdots j_m\} \in E_{k_2+1+m}$. It follows that the words $w_1= v_1i_1i_2\cdots \in A^{\N}_{u_1}$ and $w_2= v_2j_1j_2\cdots \in A^{\N}_{u_2}$ are in $A^{\N}_{u_1}\wedge_{\mathscr{C}}A^{\N}_{u_2}$.

Suppose now that there exist two distinct $w_1',w_1 \in A^{\N}_{u_1}$ such that $w_1',w_1 \in A^{\N}_{u_1}\wedge_{\mathscr{C}}A^{\N}_{u_2}$. Let $l>k_2$ be an integer such that $w_1(l)\neq w_1'(l)$. By Definition \ref{def:combintersection}, there exist $v, v'\in A^\N_{u_2} \subseteq A^\N_{u'_2}$ such that $\{w_1(l),v\}$ and $\{w'_1(l),v'\}$ are in $E_l$. This contradicts the uniqueness statement of Lemma \ref{lem:adjacent}(3).
%Lemma \ref{lem:intersection1}(2), there exist 

Finally, for the converse of (3) simply note that if $A^{\N}_{u_1}\wedge_{\mathscr{C}}A^{\N}_{u_2}$ contains exactly two elements, then by (2), either $u_1=u_2'$, or $u_1$ is adjacent to $u_2'$. However, the former is false since in that case, by (1), $A^{\N}_{u_1}\wedge_{\mathscr{C}}A^{\N}_{u_2}$ would be an infinite set. 
\end{proof}

We now study chains, as defined in Definition \ref{def:chain} of Section \ref{sec:introcomb}. The following lemma shows that, if each $G_k$ in the combinatorial data is a combinatorial tree, that chains must respect the ``between-ness'' relation in each $G_k$.

\begin{lem}\label{lem:chain}
Suppose that each graph $G_k$ is a combinatorial tree. Let $w_1,w_2,w_3 \in A^k$, and let $w_2$ be on the unique combinatorial arc in $G_k$ that joins $w_1$ and $w_3$. If $u_1\in A^{\N}_{w_1}$ and $u_3 \in A^{\N}_{w_3}$, then for every chain $\{A^{\N}_{v_1},\dots, A^{\N}_{v_n}\}$ joining $u_1$ with $u_3$, there exists $v\in A^*$ and $i\in\{1,\dots,n\}$ such that $A^{\N}_{w_2v}\subset A^{\N}_{v_i}$.
\end{lem}
\begin{proof}
We may assume that the three words $w_1, w_2, w_3$ are distinct, otherwise the lemma is trivial.

As a start, we note that $u_1$ has an initial $w_1$ substring and an initial $v_1$ substring, so either $v_1$ is an initial substring of $w_1$ or vice versa. A similar consideration applies to $u_3$, $v_n$, and $w_3$.

For each $i\in 1,\dots, n$, we define a subset $P_i \subseteq A^k = V(G_k)$ as follows: If $|v_i|<k$, then let $P_i = A^k_{v_i}$. If $|v_i|\geq k$, then let $P_i = \{v_i(k)\}$. In either case, $P_i$ induces a connected subgraph of $G_k$.

\begin{claim}
$P_1$ contains $w_1$ and $P_n$ contains $w_3$.
\end{claim}
\begin{proof} If $|v_1|<k$, then $v_1$ is an initial substring of $w_1$, and so $P_1 = A^k_{v_1} \ni w_1$. If $|v_1|\geq k$, then $w_1=v_1(k)\in P_1$.

By the same argument, $P_n$ contains $w_3$.
\end{proof}

\begin{claim}
For each $i\in \{1,\dots, n-1\}$, either $P_i \cap P_{i+1}\neq \emptyset$ or there is an edge $\{a,b\}\in E_k$ with $a\in P_i$ and $b\in P_{i+1}$. 
\end{claim}
\begin{proof}
Assume without loss of generality that $|v_i|\geq |v_{i+1}|$. Since $\{v_i\}$ is a chain, $A^\N_{v_i} \wedge_{\mathscr{C}} A^\N_{v_{i+1}}\neq \emptyset$.

Case 1: If $|v_{i+1}| \leq |v_i| < k$, then $P_i=A^k_{v_i}$ and $P_{i+1}=A^k_{v_{i+1}}$. These contain adjacent elements by Lemma \ref{lem:intersection1}(2).

Case 2: If $k \leq |v_{i+1}| \leq |v_i|$, then $P_i=\{v_i(k)\}$ and $P_{i+1}=\{v_{i+1}(k)\}$. By Lemma \ref{lem:intersection1}(3) and Lemma \ref{lem:adjacent}(2), the elements $v_i(k)$ and $v_{i+1}(k)$ are either equal or adjacent in $G_k$.

Case 3: If $|v_{i+1}| < k \leq |v_i|$,  then $P_i=\{v_i(k)\}$ and $P_{i+1}=A^k_{v_{i+1}}$. 
%We have $|v_i|=|v_{i+1}v'|$ for some $v'\in A^*$. 
If $v_i(k) \in A^{k}_{v_{i+1}}$, then clearly $P_i \subseteq P_{i+1}$. Otherwise, since $A^{\N}_{v_1} \wedge_{\mathscr{C}}A^{\N}_{v_{i+1}} \neq \emptyset$, by Lemma \ref{lem:adjacent}(2), there exist $j,l \in A$ and $v' \in A^{|v_i|-|v_{i+1}|}$ such that $v_ij$ is adjacent to $v_iv'l$, and since $v_i(k) \not\in A^{k}_{v_{i+1}}$, we have by Lemma \ref{lem:adjacent}(2) that $v_i(k)$ (which is in $P_i$) is adjacent to $(v_{i+1}v')(k)$ (which is in $P_{i+1}$). 
%Lemma \ref{lem:intersection1}(2) and Lemma \ref{lem:adjacent}(2) say that $v_i(k)$ is adjacent to $(v_{i+1}v')(k)$. Since $|v_{i+1}|<k$, the vertex $(v_{i+1}v')(k)\in A^k_{v_{i+1}} = P_{i+1}$.
This completes the proof of the claim.
\end{proof}

Thus, the union of the sets $P_1, P_2, \dots, P_n$ induces a connected subgraph of $G_k$ that contains $w_1$ and $w_3$. It therefore must contain $w_2$, so $w_2\in P_i$ for some $i$.

If $|v_i|< k$, then this means that $w_2\in P_i = A^k_{v_i}$. Thus $A^\N_{w_2}\subseteq A^\N_{v_i}$, which proves the lemma in this case.

If $|v_i|\geq k$, then  $w_2\in P_i = \{v_i(k)\}$. Thus, $w_2v=v_i$ for some word $v$, which proves the lemma in this case.
\end{proof}

\subsection{Diameter functions and metrics}\label{sec:diameter}
Recall the notion of a diameter function $\D$ on an alphabet $A$ (and the class $\mathscr{D}(A)$ of all diameter functions on $A$) from Definition \ref{def:diam}. \textbf{For the remainder of Section \ref{sec:combtrees}, we fix a diameter function $\D\in \mathscr{D}(A)$.}

Given $\mathscr{C}$ and $\D$, we defined the distance $D_{\mathscr{C},\D}$ on $A^\N$ in \eqref{eq:pseudometric} by taking an infimum over chains. We first prove that $D_{\mathscr{C},\D}$ is indeed a pseudometric as claimed.

\begin{lem}\label{lem:pseudometric}
The function $D_{\mathscr{C},\D}$ is a pseudometric on $A^\mathbb{N}$.
\end{lem}

\begin{proof}
First, notice that for any $w\in A^{\N}$ and any $n\in\N$, $\{A_{w(n)}^{\N}\}$ is a chain that joins $w$ with $w$. Thus,
$$D_{\mathscr{C},\D}(w,w) \leq \D(w(n)) \leq \max_{v\in A^n}\D(v),$$
which vanishes as $n\to \infty$ by property \eqref{eq:diam4} in Definition \ref{def:diam}. Hence, $D_{\mathscr{C},\D}(w,w)=0$.

The symmetry of $D_{\mathscr{C},\D}$ is trivial, as any chain joining $w$ with $u$ is also a chain joining $u$ with $w$.

For the triangle inequality, fix $\e>0$. Let $\{A_{w_1}^{\N},\dots,A_{w_n}^{\N}\}$ be a chain joining $w$ with $u$ and let $\{A_{u_1}^{\N},\dots,A_{u_m}^{\N}\}$ be a chain joining $u$ with $v$ such that
\[ \sum_{i=1}^n \D(w_i) < \frac{\e}{2} + D_{\mathscr{C},\D}(w,u) \qquad\text{and}\qquad  \sum_{j=1}^n \D(u_j) < \frac{\e}{2} + D_{\mathscr{C},\D}(u,v).\]
By Lemma \ref{lem:intersection2}(1), we have that $A^{\N}_{w_n}\wedge_{\mathscr{C}}A^{\N}_{u_1} \neq \emptyset$, and so $\{A_{w_1}^{\N},\dots,A_{w_n}^{\N},A_{u_1}^{\N},\dots,A_{u_m}^{\N}\}$ is a chain joining $w$ with $v$. Thus, $D_{\mathscr{C},\D}(w,v) \leq D_{\mathscr{C},\D}(w,u) + D_{\mathscr{C},\D}(u,v) + \e$. As $\e$ was chosen arbitrarily, the lemma follows.
\end{proof}

We now describe more precisely the associated metric space $(\mathcal{A}, d_{\mathscr{C},\D})$ associated to a given combinatorial data $\mathscr{C}$ and diameter function $\D$ on $A$, introduced briefly in Section \ref{sec:introcomb}.

To turn $D_{\mathscr{C},\D}$ into a metric, we define a relation on $A^{\N}$. In particular, we write $w\sim u$ (for convenience we drop the dependence on $\mathscr{C},\D$) if and only if $D_{\mathscr{C},\D}(w,u)=0$. Since $D_{\mathscr{C},\D}$ is a pseudometric, it follows that $\sim$ is an equivalence relation. Using this identification, we define
\[ \mathcal{A}= A^{\N}/\sim \qquad\text{and}\qquad \mathcal{A}_w= A^{\N}_w/\sim \quad\text{for each $w\in A^*$}.\]
Based on $D_{\mathscr{C},\D}$, we define a function $d_{\mathscr{C},\D}$ on $\mathcal{A}\times\mathcal{A}$ in the usual way: if $[w],[u] \in \mathcal{A}$, then set
\[ d_{\mathscr{C},\D}([w],[u]) := D_{\mathscr{C},\D}(w,u).\]

The function $d_{\mathscr{C},\D}$ is well-defined. To see why this is true, let $w,w',u, \in A^{\N}$ such that $[w]=[w']$. By Lemma \ref{lem:pseudometric} we have that $D_{\mathscr{C},\D}(w,u) \leq D_{\mathscr{C},\D}(w,w') + D_{\mathscr{C},\D}(w',u) = D_{\mathscr{C},\D}(w',u)$. Similarly, $D_{\mathscr{C},\D}(w',u) \leq D_{\mathscr{C},\D}(w,u)$ and, thus, $D_{\mathscr{C},\D}(w',u) = D_{\mathscr{C},\D}(w,u)$.

\begin{lem}\label{lem:dist}
The function $d_{\mathscr{C},\D}$ is a metric on $\mathcal{A}$ and for each $w \in A^*$, $\diam{\mathcal{A}_{w}} \leq \D(w)$.
\end{lem}

\begin{proof}
We first show that $d_{\mathscr{C},\D}$ is a metric. It is clear that $d_{\mathscr{C},\D}$ is non-negative, symmetric and $d_{\mathscr{C},\D}([w],[u]) = 0$ if and only if $[w]= [u]$ in $\mathcal{A}$. The triangle inequality follows from Lemma \ref{lem:pseudometric}. 

Let $w\in A^*$ and $[u_1],[u_2]\in \mathcal{A}_w$. We may choose $u_1$ and $u_2$ in $A_w$. The set $\{A^{\N}_w\}$ is then a chain joining $u_1$ with $u_2$ and $d_{\mathscr{C},\D}([u_1], [u_2]) \leq \D(w)$. Therefore, $\diam{\mathcal{A}_{w}^{\N}} \leq \D(w)$.
\end{proof}

We use standard metric space terminology when discussing $(\mathcal{A}, d_{\mathscr{C},\D})$. In particular, if $[w]\in\mathcal{A}$ and $r>0$, we write $B([w],r)$ for the open ball centered at $[w]$ of radius $r$ in this space.

\subsection{Bounded turning spaces}\label{sec:BTtrees}

We now work towards the following proposition, which proves parts \eqref{eq:maincomb1} and \eqref{eq:maincomb2} of Theorem \ref{thm:maincombthm}.

\begin{prop}\label{prop:comb-tree}
The metric space $(\mathcal{A},d_{\mathscr{C},\D})$ is compact, path-connected, and $1$-bounded turning. Moreover, if each combinatorial graph $G_k$ is a combinatorial tree, then the metric space is a tree.
\end{prop}
(Here we are using the shorthand ``$C$-bounded turning'' for ``bounded turning with constant $C$''.)

The separate statements of Proposition \ref{prop:comb-tree} are proven in Lemmas \ref{lem:compact}, \ref{lem:connected}, \ref{lem:bt}, and \ref{lem:tree}.

\begin{lem}\label{lem:cover}
Fix $w\in A^*$. Let
$$ I = \{i\in A: \Delta(wi)>0\}.$$
If $\diam(\mathcal{A}_w)>0$, then
$$ \mathcal{A}_w \subseteq \bigcup_{i\in I} \mathcal{A}_{wi}.$$
\end{lem}
\begin{proof}
The assumption that $\diam(\mathcal{A}_w)>0$ implies that $I$ is non-empty. Let $k=|w|$.

Consider any $[v]\in\cA_w$; without loss of generality, $v(k)=w$. We will show that $[v]=[u]$ for some $u\in \cup_{i\in I} A^\N_{wi}$. If $v(k+1)\in \{wi:i\in I\}$, then we are done, so suppose it is not. Then there is a simple path
$$ u_1, u_2, \dots, u_n$$
in the combinatorial tree $G_{k+1}$ such that $u_1=v(k+1)$, $u_n=wi$ for some $i\in I$, and $u_j\notin \{wi:i\in I\}$ for $1\leq j \leq n-1$.

By Lemma \ref{lem:adjacent}, there is $u\in A^{\N}_{u_n}$ such that for each $m$, either $u(k+m) \in A_{u_{n-1}}^{k+m}$, or $u(k+m)$ is adjacent to some element of $A_{u_{n-1}}^{k+m}$. In either case,
%sequence $i_1, i_2, \dots$ in $A$ such that $u_n i_1 \dots i_m$ is adjacent to some element of $A_{u_{n-1}}^{k+m}$ for each $m$. Let 
%$$ u = u_n i_1 i_2 \dots \in A^\mathbb{N}.$$
for each $m\geq 1$ we have that $A^{\N}_{u_{n-1}}\wedge_{\mathscr{C}}A^{\N}_{u(k+m)} \neq \emptyset$. Therefore, the set
$$ \{A^\mathbb{N}_{u_1}, \dots, A^\mathbb{N}_{u_{n-1}}, A^\mathbb{N}_{u(k+m)}\}$$
is a chain that joins $v$ to $u \in \cA_{wi}$. Note that $\D(u_j)=0$ for $1\leq j \leq n-1$. Therefore,
$$ D_{\mathscr{C},\D}(v, u) \leq \Delta(u(k+m)) \leq \max\{\Delta(r): r\in A^{k+m}\} \rightarrow 0 \text{ as } m\rightarrow \infty.$$
It follows that $[v]=[u]\in \cA_{wi}$. This completes the proof.
\end{proof}
We can now prove a slightly stronger version of the first statement in Proposition \ref{prop:comb-tree}.
\begin{lem}\label{lem:compact}
For each $w\in A^*$, the metric space $(\mathcal{A}_w,d_{\mathscr{C},\D})$ is compact.
\end{lem}
In particular, taking $w=\varepsilon$ we see that $(\mathcal{A},d_{\mathscr{C},\D})$ is compact, as required in Proposition \ref{prop:comb-tree}.
\begin{proof}
We show that $(\mathcal{A}_w,d_{\mathscr{C},\D})$ is sequentially compact. Let $([w_n])$ be a sequence in $\mathcal{A}_w$. Suppose that this sequence has no convergent subsequence. This implies that $\diam(\cA_w)>0$, otherwise $([w_n])$ would be constant.

Let 
$$ I_1 = \{i\in A: \Delta(wi)>0\}.$$
Note that $I_1$ is finite by Definition \ref{def:diam}. Thus, by Lemma \ref{lem:cover} there exists $i_1\in I_1$ and a subsequence $([w^1_n])$ of $([w_n])$ in $\mathcal{A}_{wi_1}$.

We proceed by induction to construct sets $I_m\subseteq A$, indices $i_m\in I_m$, and subsequences $([w^m_n])$ of $([w_n])$ contained in $\cA_{wi_1i_2\dots i_m}$.

Assuming that there is a subsequence $([w_n^m]) \subseteq \mathcal{A}_{wi_1\cdots i_m}$, let
$$ I_{m+1} = \{ i\in A: \Delta(wi_1\cdots i_m i)>0\},$$
which is finite as above. As above,  $\diam(\cA_{wi_1\cdots i_m})>0$, otherwise $([w_n^m])$ would be constant, hence convergent. Thus, by Lemma \ref{lem:cover}, there is $i_{m+1}\in I_{m+1}\subseteq A$ and a subsequence $([w^{m+1}_n])$ of $([w_n^m])$ in $\mathcal{A}_{wi_1\cdots i_{m+1}}$.

Set $u = wi_1i_2\cdots \in A^{\N}$ and consider the subsequence $([w^n_n])$ of $([w_n])$. Then, $d_{\mathscr{C},\D}([w_{n}^n],[u]) \leq \D(u(n)) \to 0$ as $n\to\infty$, contradicting our assumption. Thus, $(\mathcal{A}_w,d_{\mathscr{C},\D})$ is compact.
\end{proof}

We now work towards the connectedness properties. The following definition is convenient: An \textit{$\e$-path} in a metric space $(X,d)$ is a finite sequence $(x_1, \dots, x_n)$ such that $d(x_i, x_{i+1})\leq\e$ for each $i\in\{1,\dots, n-1\}$. We say that the $\e$-path \textit{joins} $a$ and $b$ if $a=x_1$ and $b=x_n$.

\begin{lem}\label{lem:discreteBT}
Let $[w_1],[w_2]\in\cA$ with $d_{\mathscr{C},\D}([w_1],[w_2])<r$, and let $\epsilon>0$. Then there is a $\epsilon$-path joining $[w_1]$ and $[w_2]$ of diameter less than $r$.
\end{lem}
\begin{proof}
Fix $[w_1]$, $[w_2]$, $r>0$, and $\e>0$ as in the statement of the lemma. Let $\{A^\mathbb{N}_{u_1},\dots,A^\mathbb{N}_{u_k}\}$ be a chain joining $w_1$ with $w_2$ such that 
\[ \sum_{i=1}^{k}\D(u_i) \leq d_{\mathscr{C},\D}([w_1],[w_2]) + \frac{r - d_{\mathscr{C},\D}([w_1],[w_2])}{2} < r.\]
Note that for any $i,j\in \{1,\dots,k\}$ and any $w_i \in A^{\N}_{u_i}$ and $w_j\in A^{\N}_{u_j}$, we may use a subset of this same chain to join them, and so obtain
\begin{equation}\label{eq:diambound}
 d_{\mathscr{C},\D}([w_i],[w_j]) <r.
\end{equation}

By property \eqref{eq:diam4} in Definition \ref{def:diam}, there exists $m\in\N$ such that $\D(u) \leq \e/2$ for all $u\in A^m$. By the properties of $G_m$ and Lemma \ref{lem:intersection1}, there exists a path 
\[ \g = \{\{u_1',u_2'\}, \dots \{u_{k-1}',u_k'\}\} \subset \bigcup_{i=1}^k A^m_{u_i}\] 
such that $w_1\in A^{\N}_{u_1'}$ and $w_2 \in A^{\N}_{u_n'}$. For each $i\in\{1,\dots,n\}$ let $v_i = u_i'1^{\infty}$ and let $v_0=w_1$ and $v_{n+1}=w_2$. 
Then for each $i=1,\dots,n-1$,
\[ d_{\mathscr{C},\D}([v_i],[v_{i+1}]) \leq \D(u_i') + \D(u_{i+1}') \leq \e/2 + \e/2 = \e\]
and similarly $d_{\mathscr{C},\D}([w_1],[v_{1}]) \leq \D(u_1') \leq \e$ and $d_{\mathscr{C},\D}([w_2],[v_{n}]) \leq \D(u_n') \leq \e$.

Thus, $([v_0], [v_1], \dots, [v_{n+1}])$ is an $\e$-path joining $[w_1]$ to $[w_2]$. Its diameter is less than $r$ by \eqref{eq:diambound}.
\end{proof}

The following lemma completes the proof of the topological properties in Proposition \ref{prop:comb-tree}.

\begin{lem}\label{lem:connected}
The metric space $(\mathcal{A},d_{\mathscr{C},\D})$  has the property that $\overline{B([w_0],r)}$ is connected for each $[w_0]\in\cA$ and $r>0$.

In particular, the space is connected, locally connected, and path-connected.
\end{lem}

\begin{proof}
The second sentence follows from the first: connectedness by taking $r=1\geq\diam(\cA)$, local connectedness by, e.g., \cite[(15.1)]{Whyburn}, and path-connectedness by the Hahn-Mazurkiewicz Theorem and Lemma \ref{lem:compact}.

For the first sentence, fix $w_0 \in A^{\N}$ and $r>0$. To show that $\overline{B([w_0],r)}$ is connected, it suffices to show that for any $\epsilon>0$, each $[w]\in \overline{B([w_0],r)}$ can be joined to $[w_0]$ by an $\epsilon$-path contained in $ \overline{B([w_0],r)}$.

The point $[w]$ is less than $\epsilon$-distance away from an element $[w']$ of $B([w_0],r)$. There is an $\epsilon$-path joining $[w_0]$ to $[w']$ inside $B([w_0],r)$, by Lemma \ref{lem:discreteBT}. Since $d_{\mathscr{C},\D}([w'],[w])<\epsilon$, appending $[w]$ to this path yields an $\epsilon$-path joining $[w_0]$ to $[w]$ inside $\overline{B([w_0],r)}$.
\end{proof}

\begin{lem}\label{lem:bt}
The metric space $(\mathcal{A},d_{\mathscr{C},\D})$ is $1$-bounded turning.
\end{lem}
\begin{proof}
Let $[w_1], [w_2]\in \cA$, with $r=d_{\mathscr{C},\D}([w_1],[w_2])>0$. Let $\epsilon>0$. By Lemma \ref{lem:discreteBT}, there is an $\epsilon$-path $(v_0, v_1, \dots, v_n)$ joining $[w_1]$ to $[w_2]$ with diameter at most $r+\epsilon$.

Define a compact set $K_\epsilon \subseteq \cA$ by
$$ K_\epsilon = \cup_{j=0}^n \overline{B([v_j],2\epsilon)}.$$
Note that each ball in this union is connected, by Lemma \ref{lem:connected}. Since $\overline{B([v_j],2\epsilon)} \cap \overline{B([v_{j+1}],2\epsilon)}\neq \emptyset$ for each $j=0\dots n-1$, it follows that $K_\epsilon$ is also connected. Moreover,
\begin{equation}\label{eq:diamK}
 \diam(K_\epsilon) \leq r+5\epsilon.
\end{equation}
The sets $K_{1}, K_{1/2}, K_{1/3}, \dots$ are each compact, connected, and contain both $[w_1]$ and $[w_2]$. They therefore admit a subsequence that converges in the Hausdorff metric to a compact, connected set that contains $[w_1]$ and $[w_2]$. By \eqref{eq:diamK}, this set has diameter $r$. This completes the proof.
\end{proof}

\subsection{Metric Trees}
We now prove the second half of Proposition \ref{prop:comb-tree}, namely, that if each combinatorial graph in our data is in fact a combinatorial tree, then the resulting metric space is a metric tree. \textbf{Thus, for the remainder of Section \ref{sec:combtrees}, we assume that each combinatorial graph $G_k$ is a metric tree, and we rename the graphs $T_k$ to reflect this.}

\begin{lem}\label{lem:tree1}
Suppose that $w,w',w_0 \in A^k$ and $w_0$ is on the unique combinatorial arc in $T_k$ that joins $w$ with $w'$. If there exist $u\in A^{\N}_{w}$ and $u'\in  A^{\N}_{w'}$ such that $[u]=[u']$, then $[u]\in \mathcal{A}_{w_0}$
\end{lem}

\begin{proof}
Let $w,w',w_0$ be as in the statement of the lemma. We claim that for any $\e>0$ sufficiently small, there exists $v \in A^{\N}_{w_0}$ such that $D_{\mathscr{C},\D}(u,v) < \e$. Assuming this claim, by Lemma \ref{lem:compact}, it follows that there exists $u_0 \in \mathcal{A}_{w_0}$ such that $D_{\mathscr{C},\D}(u,u_0)=0$ and we obtain that $[u]\in \mathcal{A}_{w_0}$.

To prove the claim, fix $\e>0$. Since $D_{\mathscr{C},\D}(u,u') = 0$, there exists a chain $\{A^{\N}_{w_1},\dots,A^{\N}_{w_m}\}$ that joins $u$ with $u'$ such that
$ \sum_{l=1}^m\D(w_l) < \e.$
By Lemma \ref{lem:chain}, there exist $l_0 \in \{1,\dots,m\}$ and $v \in A^{\N}_{w_0}\cap A^{\N}_{w_{l_0}}$. In particular, $\{A^{\N}_{w_1},\dots,A^{\N}_{w_{l_0}}\}$ is a chain joining $u$ with $v$. It follows that 
$$ D_{\mathscr{C},\D}(u,v) \leq \sum_{l=1}^{l_0}\D(w_l) \leq  \sum_{l=1}^m\D(w_l) < \e.$$
%and so $D_{\mathscr{C},\D}(u,v) < \e$. 
As $\epsilon>0$ was arbitrary, this proves the initial claim and hence the lemma.
\end{proof}

\begin{lem}\label{lem:tree}
The metric space $(\mathcal{A},d_{\mathscr{C},\D})$ is a metric tree. 
\end{lem}

\begin{proof}
First of all, since $(\mathcal{A},d_{\mathscr{C},\D})$ is Hausdorff and path-connected, it is also arcwise connected, see e.g. \cite[Section 31]{Willard}. Let $[w_1], [w_2]$ be two distinct arbitrary points in $\cA$. We will show that there is a point of $\cA\setminus\{[w_1],[w_2]\}$ (in fact, a whole continuum) that every path $\g$ from $[w_1]$ to $[w_2]$ must contain. This clearly implies that there can be no simple closed path containing $[w_1]$ and $[w_2]$, and therefore that $\cA$ is a metric tree. (See \cite[Theorem 1.1]{Charatonik} for various characterizations of metric trees, called dendrites there, from which we are using characterization (20).)

For each $n\in\N$ let 
\[ \{ v_{n,1}, \dots, v_{n,m(n)} \} \subseteq A^n, \]
be all the vertices of $T_n$ lying on the unique combinatorial arc that joins $w_1(n)$ with $w_2(n)$, ordered so that $v_{n,1}=w_{1}(n)$, $v_{n,m(n)}=w_2(n)$, and $\{v_{n,i},v_{n,i+1}\} \in E_n$ for all $i=1,\dots,m(n)-1$. 

Note that, for each $n\in\N$ and $i\in\{1, \dots, m(n+1)\}$, the word $v_{n+1, i}(n)$ lies on the combinatorial arc from $w_1(n)$ to $w_2(n)$, i.e., is equal to $v_{n,j}$ for some $j\in\{1, \dots, m(n)\}$. Indeed, if not, then the combinatorial arc $\{ v_{n,1}, \dots, v_{n,m(n)} \} $ avoids $v_{n+1, i}(n)$, and so by Definition \ref{def:combdata}, properties \eqref{eq:combdata2a} and \eqref{eq:combdata2b}, we can form an arc from $w_1(n+1)$ to $w_2(n+1)$ that avoids $v_{n+1,i}$, contradicting the uniqueness of this arc in $T_{n+1}$. 

Conversely, if $n\in\N$ and $i\in\{1, \dots, m(n)\}$, then some $v_{n+1,j}$ has $v_{n+1,j}(n)=v_{n,i}$. If not, then using Definition \ref{def:combdata}, properties \eqref{eq:combdata2a} and \eqref{eq:combdata2b}, we could construct a separate combinatorial arc joining $w_1(n+1)$ and $w_2(n+1)$ that does contain some child of $v_{n,i}$, violating the tree conditoin.

The upshot of the previous two paragraphs is that each $\mathcal{A}_{v_{n+1,i}}$ is contained in some $\mathcal{A}_{v_{n,j}}$, and each $\mathcal{A}_{v_{n,i}}$ contains some $\mathcal{A}_{v_{n+1,j}}$

In particular, for each $n\in\N$,
\[ \bigcup_{i=1}^{m(n+1)}\mathcal{A}_{v_{n+1,i}} \subseteq \bigcup_{i=1}^{m(n)}\mathcal{A}_{v_{n,i}}.\]
Let
$$ K_n:= \bigcup_{i=1}^{m(n)}\mathcal{A}_{v_{n,i}}\subseteq \mathcal{A},\text{ and }$$
\[ K := \bigcap_{n=1}^{\infty} K_n \subseteq \cA \]
Note that the above sets are all compact by Lemma \ref{lem:compact}.

\begin{claim}
We have that $[w_1],[w_2] \in K$.
\end{claim} 
\begin{proof}
We have that $w_1=v_{n,1}$ for each $n$, so $w_1\in A^\N_{v_n,1}$ for each $n$. Hence $[w_1]\in \mathcal{A}_{v_{n,1}}\subseteq K_n$ for each $n$, and $[w_1]$ is therefore in $K$. Similarly, $[w_2]\in K$.

%Given $\e>0$, there exists $n\in\N$ such that $\sup_{w\in A^n}\D(w)<\e/2$. Then,
%\[ d_{\mathscr{C},\D}([w_1], K_n) \leq D_{\mathscr{C},\D}(w_1,v_{n,1}1^{\infty}) \leq \D(w_1(n)) + \D(v_{n,1}) <\e.\]
%Therefore, $[w_1]\in K$. Similarly, $[w_2]\in K$.
\end{proof}

\begin{claim}\label{claim:Kcontinuum}
The set $K$ contains a continuum that joins $[w_1]$ with $[w_2]$.
\end{claim}
\begin{proof}
For any $\d>0$, there exists $n\in \N$ such that $\sup_{w\in A^n}\D(w)<\d/2$. We first claim that, for any $i =1,\dots, m(n)$ there exists a point $[v_i] \in \mathcal{A}_{v_{n,i}} \cap K$. Indeed, by the discussion at the beginning of the proof of this lemma, there is a sequence
$$ \mathcal{A}_{v_{n,i}} \supseteq \mathcal{A}_{v_{n+1,i_1}} \supseteq \mathcal{A}_{v_{n+2,i_2}} \supseteq \dots$$
By compactness of $\mathcal{A}$ and the definition of $K$, there is an element of $K$ in the intersection of these.

It is then immediate that  $([w_1], [v_1], [v_2],\dots,[v_{m(n)}], [w_2])$ is a $\d$-path in $K$ joining $[w_1]$ with $[w_2]$. As the choice of $\delta>0$ was arbitrary, it follows from this that $[w_1]$ and $[w_2]$ must lie in the same connected component of $K$ (see \cite[(9.2), p. 15]{Whyburn}), which must also be closed as $K$ is compact.
% As $\delta$ was arbitrary, this shows that the connected component of $K$ containing $[w_1]$ also contains $[w_2]$. 
\end{proof}

\begin{claim}
The set $K$ is contained in every path $\g$ from $[w_1]$ to $[w_2]$ in $\cA$. 
\end{claim}
\begin{proof}
Fix such a path $\g$ and let $\e>0$ and $[v_0] \in K$. Choose $n\in \N$ such that $\sup_{w\in A^n}\D(w)<\e$. Let $i \in \{1,\dots,m(n)\}$ such that $[v_0] \in \mathcal{A}_{v_{n,i}}$. Let $\{T_{n,j} =(V_j, E_j)\}_j$ enumerate the components of $T_n \setminus\{v_{n,i}\}$. For each $j$, let $X_j = \bigcup_{w\in V_j}\mathcal{A}_w$. These are compact sets: each can be rewritten as $X_j = \bigcup_{w\in V_j, \D(w)=0}\mathcal{A}_w$, and this is a finite union of compact sets by Definition \ref{def:diam}(2) and Lemma \ref{lem:compact}.

Moreover, the union of these sets contains $\mathcal{A} \setminus \mathcal{A}_{v_{n,i}}$. Finally, the sets $\{X_j\}$ also have the property that $X_j \cap X_{j'} \subseteq \cA_{v_{n,i}}$ whenever $j\neq j'$. Indeed, if $[v]\in X_j \cap X_{j'}$, then $[v]=[u]=[u']$, where $u(n)\in T_{n,j}$ and $u(n)\in T_{n,j'}$. The unique combinatorial arc from $u(n)$ to $u(n')$ in $T_n$ contains $v_{n,i}$, so by Lemma \ref{lem:tree1} we have that $[v]=[u]\in \mathcal{A}_{v_{n,i}}$.

If neither of $[w_1]$ or $[w_2]$ is contained in $\mathcal{A}_{v_{n,i}}$, then $w_1(n)$ and $w_2(n)$ are contained in different subgraphs $T_{n,j}$ and hence $[w_1],[w_2]$ are contained in different sets $X_j$. In either case, the path $\g$ must intersect $\mathcal{A}_{v_{n,i}}$. Thus,
\[ d_{\mathscr{C},\D}(\g,[v_0]) \leq \D(w_{n,i}) < \e.\]
Since $\epsilon>0$ was arbitrary, we have $[v_0]\in \g$.
\end{proof}

Thus, every path in $\cA$ from $[w_1]$ to $[w_2]$ contains $K$, which contains a fixed continuum joining $[w_1]$ and $[w_2]$. In particular, any two such paths must intersect somewhere other than their endpoints. This shows that $\cA$ is a metric tree.
\end{proof}

\begin{rem}\label{rem:arc}
Given $w_1,w_2 \in A^{\N}$, let $K\subset \mathcal{A}$ be as in the proof of Lemma \ref{lem:tree}. We showed above that $K$ contains a continuum that joins $[w_1]$ with $[w_2]$ and, conversely, that every path in $\mathcal{A}$ that joins $[w_1]$ with $[w_2]$ contains $K$. Therefore, $K$ is the unique arc that joins $[w_1]$ with $[w_2]$ in $\mathcal{A}$.
\end{rem}

Together, Lemmas \ref{lem:connected}, \ref{lem:bt}, and \ref{lem:tree} prove Proposition \ref{prop:comb-tree}.

\section{Doubling metric trees}\label{sec:doubling}
Recall that a metric space is $C$-doubling if there exists a constant $C\geq 1$ such that for any $x\in X$ and $r>0$, the ball $B(x,r)$ can be covered by at most $C$ balls of radius $r/2$. Our goal here is to give some sufficient conditions for our combinatorial construction to yield a doubling metric tree.

\textbf{For the remainder of Section \ref{sec:doubling}, we assume that $A$ is an alphabet and $\mathscr{C} = (A,(T_k)_{k\in\N})$ is combinatorial data as in Definition \ref{def:combdata}, with the additional assumption that each graph $T_k$ is a combinatorial tree.}

%Given $u\in A^*$ we define the \emph{combinatorial boundary} of $A^{\N}_u$ by
%\[ \partial_{\mathscr{C}}A^{\N}_u := A^\N_u \cap \bigcup_{v\in A^{|u|}\setminus\{u\}} \left( A^{\N}_v \wedge_{\mathscr{C}}A^{\N}_u \right).\]
%In other words, $w\in \partial_{\mathscr{C}}A^{\N}_u$ if and only if $w\in A^{\N}_u$ and for every $n>|u|$, there exists $u' \in A^n \setminus A^n_u$ with $\{w(n),u'\}\in E_{n}$.  

\begin{prop}\label{prop:doubling}
Fix $N,n_0\in \N$, $c>1$, and $\d_1, \d_2\in (0,1)$.  There exists $C>1$, depending only on these constants, with the following property. Assume that:
\begin{enumerate}[label=(P\arabic*), ref=(P\arabic*)]
\item \label{P1} $\card{A} \leq N$.
\item \label{P2} $\text{Val}(T_k)\leq n_0$ for all $k\in\N$.
\item \label{P3} For all $w \in A^*$ and $i\in A$, $\d_1\D(w) \leq \D(wi) \leq \d_2\D(w)$.
\item \label{P4} Suppose that for some $k\in\N$ and some distinct $u,u_1,u_2\in A^n$ we have $A^{\N}_{u} \wedge_{\mathscr{C}}A^{\N}_{u_i} \neq \emptyset$ for $i=1,2$. If $w_i \in A^{\N}_{u} \wedge_{\mathscr{C}}A^{\N}_{u_i}$ for $i=1,2$, then $d_{\mathscr{C},\D}([w_1],[w_2]) \geq c^{-1}\D(u).$
%For all $u\in A^*$ and all distinct $w_1,w_2 \in \partial_{\mathscr{C}}A^{\N}_u$ we have $d_{\mathscr{C},\D}([w_1],[w_2]) \geq c^{-1}\D(u).$
\end{enumerate}
Then $(\mathcal{A},d_{\mathscr{C},\D})$ is $C$-doubling.
\end{prop}

\begin{rem}
Items \ref{P1}, \ref{P2}, and \ref{P3} of Proposition \ref{prop:doubling} are rather innocuous, while \ref{P4} requires some more thought. Essentially, \ref{P4} prevents the space from ``collapsing'' too many far away points close together, which may violate doubling. In Lemma \ref{lem:suffdoubling}, we provide a more easily checkable condition that implies \ref{P4}, and in Example \ref{ex:nondoubling} we show that \ref{P4} is necessary in Proposition \ref{prop:doubling}.

Note also that if $w_i,w_i' \in A^{\N}_{u} \wedge_{\mathscr{C}}A^{\N}_{u_i}$, then $d_{\mathscr{C},\D}([w_i],[w_i']) =0$. Therefore, in \ref{P4}, we may assume that $w_i \in (A^{\N}_{u} \wedge_{\mathscr{C}}A^{\N}_{u_i})\cap A^{\N}_{u}$.
\end{rem}

Recall the definition of a parent word $u^{\uparrow}$. For the proof of Proposition \ref{prop:doubling}, we make the following definition. Given $r>0$ define 
\[ A^*(r) := \left\{w\in A^* : \D(w)< r\text{ and }\D(w^{\uparrow}) \geq r\right\}. \]

\begin{rem}\label{rem:partition}
The set $A^*(r)$ induces a partition on $A^{\N}$. Namely, $A^{\N} = \bigcup_{u\in A^{*}(r)}A_u^{\N}$ and for distinct $w,u \in A^{*}(r)$ we have $A^{\N}_w \cap A^{\N}_u = \emptyset$.
\end{rem}

\begin{lem}\label{lem:partition}
Let $A$ and $\mathscr{C}$ satisfy \ref{P2}. Then, for each $r>0$ and for each $w\in A^*(r)$, there exist at most $n_0$ words $u \in A^*(r) \setminus \{w\}$ such that $A^{\N}_w\wedge_{\mathscr{C}}A^{\N}_u \neq \emptyset$.
\end{lem}

\begin{proof}
Let $r>0$ and $w\in A^*(r)$. To prove the claim, let $u_1,\dots,u_n$ be words in $A^*(r) \setminus \{w\}$ such that $A^{\N}_w\wedge_{\mathscr{C}}A^{\N}_{u_i} \neq \emptyset$ for each $i$.

Let $k_0 = |w|$. If $|u_i|< k_0$, then by Lemmas \ref{lem:intersection1} and \ref{lem:intersection2}, there exists a unique $u_i' \in A^{k_0}_{u_i}$ such that $\{w,u_i'\}\in E_{k_0}$. If $|u_i| \geq k_0$, then let $u_i' = u_i(k_0)$ and by Lemma \ref{lem:adjacent}, we have that $\{w,u_i'\} \in E_{k_0}$. We claim that if $i\neq j$, then $u_i'\neq u_j'$. Assuming the claim, by \ref{P2} we have that $n \leq n_0$ and so the proof is complete once we establish this claim. To do so, we fix distinct $i,j \in \{1,\dots,n\}$ and consider three possible cases.

\emph{Case 1.} Suppose that $|u_i|\geq k_0$ and $|u_j|\geq k_0$. For a contradiction, assume that $u_i' = u_j' = u'$. By Remark \ref{rem:partition} we have that $u' \neq w$. Therefore, by Lemma \ref{lem:intersection2}, $\{u',w\} \in E_{k_0}$. Let $k = \max\{|u_i|,|u_j|\}$.  By Lemma \ref{lem:adjacent}, there exist unique $w'' \in A^k$ and unique $u''\in A^k_{u'}$ such that $\{w'',u''\}\in E_k$. By Remark \ref{rem:partition}, either $u'' \not\in A_{u_i}^k$ or $u'' \not\in A_{u_j}^k$. Assuming the former (without loss of generality), by Lemma \ref{lem:intersection1}, we have $A^{\N}_{w}\wedge_{\mathscr{C}}A^{\N}_{u_i}= \emptyset$ which is a contradiction.

\emph{Case 2.} Suppose that $|u_i|\leq k_0$ and $|u_j|\leq k_0$. For a contradiction, assume that $u_i' = u_j' = u'$. Then $A^\N_{u_i}\cap A^\N_{u_j} \neq \emptyset$, which contradicts Remark \ref{rem:partition}.

\emph{Case 3.} Suppose that $|u_i|\leq k_0$ and $|u_j|\geq k_0$. By Remark \ref{rem:partition}, $u_i' \neq w$. Now apply the arguments of Case 1 to the triple $u_i'$, $w$, and $u_j$.
\end{proof}

\begin{proof}[{Proof of Proposition \ref{prop:doubling}}]
Let $[w]\in \mathcal{A}$ and $r>0$. To prove the proposition, it suffices to prove that the doubling property holds for the ball $B([w],r)$ if $r<c^{-1}\diam{\mathcal{A}}$. Let $u_0$ be the unique element of $A^*(c\delta_1^{-1}r)$ such that $w\in A^\N_{u_0}$.

\begin{claim}\label{claim:doubling1}
There exist at most $n_0$ words $u \in A^*(c\d_1^{-1}r) \setminus \{u_0\}$ such that $A^{\N}_{u_0}\wedge_{\mathscr{C}}A^{\N}_u \neq \emptyset$, and each such word $u$ satisfies 
$$ c\delta_1^{-1}r > \Delta(u)  \geq cr.$$
\end{claim}
\begin{proof}[Proof of Claim \ref{claim:doubling1}]
By Lemma \ref{lem:partition}, there exist at most $n_0$ such words $u \in A^*(c\d_1^{-1}r) \setminus \{u_0\}$. Moreover, by \ref{P3}, for each $u \in A^*(c\d_1^{-1}r)$,
\[ c\d_1^{-1}r > \D(u) \geq \d_1\D(u^{\uparrow}) \geq cr.\qedhere\]
\end{proof}

\begin{claim}\label{claim:doubling2}
If $u\in A^*(c\d_1^{-1}r)$ and $A^{\N}_{u_0}\wedge_{\mathscr{C}}A^{\N}_u = \emptyset$, then for any $w'\in A^{\N}_u$ we have $d_{\mathscr{C},\D}([w],[w']) \geq r$.
\end{claim}

\begin{proof}[Proof of Claim \ref{claim:doubling2}]
Let $\gamma \subset \mathcal{A}$ be the unique arc with endpoints $[w]$ and $[w']$. For each $k$, let $P_k$ be the simple path in $T_k$ from $w(k)$ to $w'(k)$.

Let $n=\max\{|u|,|u_0|\}$. Then $P_n$ must contain a vertex $v\in A^n\setminus (A^n_{u_0} \cup A^n_{u})$, otherwise $A^{\N}_{u_0}\wedge_{\mathscr{C}}A^{\N}_u \neq \emptyset$. Consider the following two possible cases.

\emph{Case 1.} Suppose that $v\in A^*(c\d_1^{-1}r)$ or $v$ has a descendent in $A^*(c\d_1^{-1}r)$. Then $v$ is adjacent to two distinct vertices $v_1$ and $v_2$ of $P_n$. For $i=1,2$, let $w_i\in A^\mathbb{N}_v$ be such that $w_i(k)\in P_k$ and is adjacent to an element of $A^k_{v_i}$ for each $k\geq n$. By Remark \ref{rem:arc}, both $[w_1]$ and $[w_2]$ are in $\gamma$. Therefore, by the 1-bounded turning property of $\mathcal{A}$, by \ref{P3}, and by \ref{P4},
\[ d_{\mathscr{C},\D}([w],[w']) = \diam{\gamma} \geq d_{\mathscr{C},\D}([w_1],[w_2]) \geq c^{-1}\D(v) \geq r.\]

\emph{Case 2.}
Suppose that $v$ is contained in $A^*_{v'}$ for some $v' \in A^*(c\d_1^{-1}r)$. Let $m =|v'|$. First, note that $P_m$ must contain $v'$; if not, then out of $P_m$ we could construct a combinatorial arc in $T_n$ that does not contain $v$ which implies that there are two distinct combinatorial arcs in $T_n$ with the same endpoints. The latter however contradicts the fact that $T_n$ is a tree. Second, by Remark \ref{rem:partition}, we have that $A^{\N}_{v'} \cap A^{\N}_{u_0} = \emptyset$. Since $A^{\N}_{u_0}\subset A^{\N}_{u_0(m)} = A^{N}_{w(m)}$, it follows that $A^{\N}_{v'} \cap A^{\N}_{w(m)} = \emptyset$. Similarly, $A^{\N}_{v'} \cap A^{\N}_{w(m)} = \emptyset$. Therefore, $v'$ is adjacent to two distinct vertices of $P_m$. Now working as in Case 1, we obtain that $d_{\mathscr{C},\D}([w],[w']) \geq c^{-1}\D(v') \geq r$.
\end{proof}

\begin{claim}\label{claim:doubling3}
Let $u \in A^*(c\d_1^{-1}r)$ and let $k$ be the smallest positive integer such that
\[ k\geq \frac{\log((2c)^{-1}\d_1)}{\log(\d_2)}.\]
Then 
$$\diam(\mathcal{A}_v) < r/2$$
for each $v\in A^{|u|+k}_u$.
\end{claim}
\begin{proof}[Proof of Claim \ref{claim:doubling3}]
By the upper bound in \ref{P3} we have that for every $v \in A^{|u|+k}_u$,
\[ \diam(\mathcal{A}_v) \leq \D(v) \leq \d_2^k\D(u) < \d_2^k\d_1^{-1}cr \leq r/2.\qedhere\]
\end{proof}

Let $\{u_1,\dots,u_p\}$ be all the words $u\in A^*(c\d_1^{-1}r)\setminus\{u_0\}$ such that $A^{\N}_{u_0}\wedge_{\mathscr{C}}A^{\N}_u \neq \emptyset.$ By Claim \ref{claim:doubling2},
$$ B([w],r) \subseteq \bigcup_{i=0}^p\mathcal{A}_{u_i}.$$
Claim \ref{claim:doubling1} implies that $p\leq n_0$. Claim \ref{claim:doubling3} implies that each of the sets $\mathcal{A}_{u_i}$ in this union can be covered by at most $N^k$ sets of diameter $< r/2$, hence $N^k$ balls of radius $r/2$. This completes the proof.
\end{proof}

We now give some sufficient conditions for \ref{P4} which are easier to verify.

For the next lemma we use the following notation. Consider combinatorial data $\mathscr{C} = (A,(T_k)_{k\in\N})$ as fixed at the beginning of this section. For each $k\in\N$ and $w\in A^k$, let $\partial_{\mathscr{C}} A^{k+1}_w$ be all words $u \in  A^{k+1}_w$ for which there exists $u' \in A^{k+1}\setminus  A^{k+1}_w$ with $\{u,u'\}\in E_{k+1}$.

\begin{lem}\label{lem:suffdoubling}
Let $\mathscr{C} = (A,(T_k)_{k\in\N})$ be combinatorial data as fixed at the beginning of this section, and let $\D \in \mathscr{D}(A)$. Assume that the following conditions hold for each $k\geq 0$. 
\begin{enumerate}
\item Suppose that $w,u,u' \in A^k$ are distinct with $\{w,u\}, \{w,u'\}\in E_k$. If $wi,wj,ul,u'l' \in A^{k+1}$ with $\{wi,ul\}, \{wj,u'l'\}\in E_{k+1}$, then $i\neq j$.
\item For any $w\in A^k$, and any distinct $u,u' \in \partial_{\mathscr{C}}A^{k+1}_w$, the arc $\{\{u,u_1\},\dots,\{u_n,u'\}\}$ joining $u$ with $u'$ in $T_{k+1}$ satisfies
\[ \D(u) + \D(u_1) + \cdots + \D(u_n) + \D(u') \geq \D(w).\]
\end{enumerate}
Then \ref{P4} of Proposition \ref{prop:doubling} holds with $c=1$.

In particular, $\diam(\mathcal{A}_u)=\Delta(u)$ for each $u\in A^*$ with at least two neighbors in $T_{|u|}$.
\end{lem}

For the proof of the lemma, given a chain $\mathcal{C} = \{A^{\N}_{u_1},\dots,A^{\N}_{u_n}\}$ joining two words in $A^\N$, we define the \emph{depth} of $\mathcal{C}$ to be the number $\text{Depth}(\mathcal{C}) := \max\{|u_1|,\dots,|u_n|\}$ and the \emph{$\D$-length of $\mathcal{C}$} to be
\[ \ell(\mathcal{C}) := \sum_{i=1}^n\D(u_i).\]

\begin{proof}
Fix $k\in\N$ and $w,u_1,u_2 \in A^k$ be distinct points such that $\{w,u_1\}$ and $\{w,u_2\}$ are in $E_k$. Let $w_1, w_2 \in A^{\N}_w$, $w_1'\in A^{\N}_{u_1}$, and $w_2' \in A^{\N}_{u_2}$ such that for any $n\geq k$ and any $i\in \{1,2\}$, $w_i(n)$ is adjacent to $w_i'(n)$. We will show that $d_{\mathscr{C},\D}([w_1],[w_2]) = \D(w)$.

On the one hand, $\{A^{\N}_w\}$ is a chain joining $w_1$ with $w_2$, so $d_{\mathscr{C},\D}([w_1],[w_2]) \leq \D(w)$. For the opposite inequality, fix $\mathcal{C} = \{A^{\N}_{v_1},\dots,A^{\N}_{v_n}\}$ to be a chain in $A^{\N}$ joining $w_1$ with $w_2$. We start by doing four reductions.

First, if $A^{\N}_w \subset A^{\N}_{v_i}$ for some $i$, then we can replace $\mathcal{C}$ with $\mathcal{C}' = \{A^{\N}_w\}$ which has smaller $\D$-length. Therefore we may assume that for all $i$, either $A^{\N}_w\cap A^{\N}_{v_i} = \emptyset$, or $A^{\N}_{v_i} \subset A^{\N}_w$.

Second, dropping some of the sets in the chain, if necessary, we may assume that $A^{\N}_{v_i}\subset A^{\N}_w$ for all $i$.

Third, if $A^{\N}_{v_i} \subset A^{\N}_{v_j}$, then we can drop $A^{\N}_{v_i}$.

Fourth, let $P_l$ be the combinatorial arc in $T_{l}$ that joins

$w'_1(l)$ with $w'_2(l)$. We first claim that $P_l$ contains $w_1(l),w_2(l)$. By Definition \ref{eq:combdata2a}, the subgraph of $T_l$ induced by the vertex set $A_{w}^l$ is connected so there exists a combinatorial arc $P'$ with vertices in $A_w^l$ that has endpoints $w_1(l),w_2(l)$. Adding the two points $w_1'(l),w_2'(l)$ along with edges $\{w_1(l),w_1'(l)\}, \{w_2(l),w_2'(l)\}$, we obtain a combinatorial arc in $T_l$ that has endpoints $w_1'(l),w_2'(l)$ and contains $w_1(l),w_2(l)$. By uniqueness of this arc, it must be $P_l$ and the proof of the claim is complete. Now, it follows from Lemma \ref{lem:chain} that for any $l\geq \text{Depth}(\mathcal{C})$,
\[ \bigcup_{v\in P_l}A^{\N}_v \subset A^{\N}_{v_1}\cup \cdots\cup A^{\N}_{v_n}.\]

\begin{claim}
The collection
$$ \mathcal{C}' = \{A^{\N}_{v_i} : A^{\N}_{v_i} \cap \bigcup_{v\in P_l}A^{\N}_v \neq \emptyset\}$$
forms a chain joining $w_1$ and $w_2$. 
\end{claim}
\begin{proof}
First, there exists $v\in P_l$ such that $w_1\in A^{\N}_v$ which implies that there exists $v_i$ such that $w_i \in A^{\N}_v \subset A^{\N}_{v_i}$. Therefore, $w_1$ is contained in some element of $\mathcal{C}'$ and similarly for $w_2$. 

Enumerate the arc $P_l = \{v_0', \dots, v_{p+1}'\}$ so that $v_0'=w_1'(l)$, $v_1'=w_1(l)$, $v_p' = w_2(l)$, $v_{p+1}'=w_2'(l)$, and for any $j$, $v_j'$ is adjacent to $v_{j+1}'$. Now, there exists a set $\{m_1, \dots, m_s\} \subset \{1,\dots,n\}$ such that 
\begin{enumerate}
\item $A^{\N}_{v_1'}\subset A^{\N}_{v_{m_1}}$, and $A^{\N}_{v_p'}\subset A^{\N}_{v_{m_s}}$,
\item for all $v_i'$, there exists $v_{m_j}$ such that $A^{\N}_{v_i'}\subset A^{\N}_{v_{m_j}}$,
\item if $A^{\N}_{v_i'}\subset A^{\N}_{v_{m_j}}$ and $A^{\N}_{v_{i+1}'}\subset A^{\N}_{v_{m_s}}$, then $m_j \leq m_s$.
\end{enumerate}

Now it is easy to see that $A^{\N}_{v_{m_i}}\wedge_{\mathscr{C}} A^{\N}_{v_{m_{i+1}}} \neq \emptyset$ so the set $\mathcal{C}' = \{A^{\N}_{v_{m_i}}: i=1,\dots,s\}$ forms a chain joining $w_1$ and $w_2$.
\end{proof}

The fourth reduction says, in other words, that we may drop all sets $A^\mathbb{N}_{v_i}$ from the chain such that $A^{\N}_{v_i} \cap \bigcup_{v\in P_l}A^{\N}_v = \emptyset$.

The four reductions imply that we may assume that for all $i$,
\begin{enumerate}
\item[(i)] $A^{\N}_{v_i}\subset A^{\N}_w$;
\item[(ii)] if $j\neq i$, then $A^{\N}_{v_i}\cap A^{\N}_{v_j} = \emptyset$;
\item[(iii)] for all $l\geq \text{Depth}(\mathcal{C})$, there exists $v\in P_l$ such that $A^{\N}_v \subset A^{\N}_{v_i}$; and
\item[(iv)] for all $l\geq \text{Depth}(\mathcal{C})$,  $\bigcup_{v\in P_l}A^{\N}_v \subset A^{\N}_{v_1}\cup \cdots\cup A^{\N}_{v_n}$.
\end{enumerate}

Let $k_0 = \text{Depth}(\mathcal{C})$ and $i_0\in \{1,\dots,n\}$ such that $|v_{i_0}|=k_0$. If $k_0 = |w|$, then $\mathcal{C}=\{A^{\N}_w\}$ and the $\D$-length of $\mathcal{C}$ is equal to $\D(w)$.

Assume now that $k_0 > |w|$. Then $v_{i_0}^{\uparrow}$ is contained in $P_{k_0-1}$. Moreover, $v_{i_0}^\uparrow$ has valency 2 in $P_{k_0-1}$, because the endpoints of $P_{k_0-1}$ are in $A^{k_0-1}_{u_i}$ and not in $A^{k_0-1}_w$. 

By (iii) and assumption (1) of the lemma, $A^{k_0}_{v_{i_0}^{\uparrow}} \cap P_{k_0}$ has at least two elements. By (ii), (iv) and the assumption that $|v_{i_0}| = \text{Depth}(\mathcal{C})$, each element of $A^{k_0}_{v_{i_0}^{\uparrow}} \cap P_{k_0}$ must be in $\{v_1, \dots, v_n\}$. Enumerate them as $\{v_{j_1}, v_{j_2}, \dots, v_{j_p}\}$. Since $v_{i_0}^\uparrow$ has valency 2 in $P_{k_0-1}$, the elements of $\{v_{j_1}, v_{j_2}, \dots, v_{j_p}\}$ contain the vertices of a simple path joining two distinct points of $\partial_{\mathscr{C}}A^{k_0}_{v_{i_0}^{\uparrow}}$.

But then, by assumption (2) of the lemma, 
\[ \D(v_{j_1}) + \cdots + \D(v_{j_p}) \geq \D(v_{i_0}^{\uparrow})\]
and we can replace $\mathcal{C}$ with the chain
\[ \mathcal{C} \cup \{ A^{\N}_{v_{i_0}^{\uparrow}}\} \setminus \{ A^{\N}_{v_{i}} : v_i \in A_{v_{i_0}^{\uparrow}}^{k_0} \}\]
which has at most the $\D$-length of $\mathcal{C}$.

Working in similar fashion, we can show that if $\text{Depth}(\mathcal{C}) > |w|$, then there exists a chain $\mathcal{C}'$ joining $w_1$ with $w_2$ such that $\text{Depth}(\mathcal{C}') = \text{Depth}(\mathcal{C}) - 1$ and has at most the $\D$-length of $\mathcal{C}$. Applying a backwards induction on the depth of $\mathcal{C}$, we obtain that 
\[ \ell(\mathcal{C}) \geq \ell(\{A^\mathbb{N}_w\}) = \D(w).\]
Therefore, $d_{\mathscr{C},\D}([w_1],[w_2]) \geq \D(w)$.

For the final statement in the lemma, any $u\in A^k$ with two distinct neighbors must have at least two distinct words in its combinatorial boundary, and so
$$ \diam(\mathcal{A}_u) \geq \D(w)$$
by the first part of the lemma. The reverse inequality follows from Lemma \ref{lem:dist}.
\end{proof}

For examples of combinatorial data and diameter functions satisfying the assumptions of Proposition \ref{prop:doubling} and Lemma \ref{lem:suffdoubling}, see Section \ref{sec:examples}.

\section{Characterization of quasiconformal trees}\label{sec:characterization}

We now claim that our combinatorial constructions above describe all quasiconformal trees up to bi-Lipschitz equivalence. The following result proves part \eqref{eq:maincomb3} of Theorem \ref{thm:maincombthm}, while providing additional details, and is the goal of this section.

\begin{thm}\label{thm:main}
Let $(X,d)$ be an $N$-doubling, $C$-bounded turning tree. Then for any $M\in\mathbb{N}$ sufficiently large, $K_1>0$ sufficiently small, and $K_2\in [\frac{1}{2},1)$, there exist:
\begin{enumerate}
\item an alphabet $A=\{1,\dots,M\}$,
\item combinatorial data $\mathscr{C}=(A,(T_k)_{k\in\N})$ with each $T_k$ a combinatorial tree,
\item a diameter function $\D \in \mathscr{D}(A,K_1, K_2)$
\end{enumerate} 
such that  $(\mathcal{A},d_{\mathscr{C},\D})$ is bi-Lipschitz equivalent to $X$.

The sufficient condition on $M$ depends only on $N$ and $C$. The sufficient condition on $K_1$ depends only on $M$, $N$, and $C$. The bi-Lipschitz constant depends only on $N$, $C$, $K_2/K_1$, and $\diam(X)$.

Moreover, $(\mathscr{C},\Delta)$ satisfies the conditions of Proposition \ref{prop:doubling}.
\end{thm}

We first make some small reductions. If $X$ is a single point, then Theorem \ref{thm:main} is easy. For example, one may take $M=2$, $\D\in\mathscr{D}\left(A,\frac{1}{3}, \frac{1}{3}\right)$, and each $T_k$ a combinatorial arc. Thus, we may assume that $\diam(X)>0$ and so, by rescaling, that $\diam(X)=1$. We may also assume that the bounded turning constant $C$ is equal to $1$, by replacing the metric $d$ on $X$ with a bi-Lipschitz equivalent $1$-bounded turning metric. (See \cite[Lemma 2.5]{BM}.) \textbf{All these assumptions are in force for the remainder of Section \ref{sec:characterization}. Thus, we fix an $N$-doubling, $1$-bounded turning metric tree $X$ of diameter $1$.}

\subsection{Subdividing into a uniform number of pieces}
To prove Theorem \ref{thm:main}, we use a construction of Bonk-Meyer \cite{BM} to decompose the tree $X$ into suitable pieces. We then modify this construction to decompose $X$ into an equal number of pieces at each scale. We first summarize the results we need from \cite[Section 5]{BM}.

\begin{prop}[Bonk-Meyer \cite{BM}]\label{prop:BM}
Let $\delta>0$ sufficiently small, depending on $N$. Then there is a constant $M(N,\delta)\in\mathbb{N}$, and for each $n\in\mathbb{N}$ there exists a $\delta^n$-separated set $V_n\subseteq X$ satisfying
$$ V_1 \subseteq V_2 \subseteq \dots $$
with the following properties.

Write $\mathcal{T}_n$ for the collection of closures of components of $X\setminus V_n$. Then:
\begin{enumerate}
\item Each $T\in\mathcal{T}_n$ is a connected subset (hence subtree) of $X$ with $\emptyset\neq T \cap\overline{X\setminus T} \subseteq V_n$.
\item Distinct elements $T,T'\in\mathcal{T}_n$ have at most one point in common, and such a common point is an element of $V_n$.
\item Each element of $V_n$ is in exactly two elements of $\mathcal{T}_n$.
\item Each element of $\mathcal{T}_{n+1}$ ($n\geq 1$) is in exactly one element of $\mathcal{T}_{n}$, and each element of $\mathcal{T}_{n}$ is the union of all elements of $\mathcal{T}_{n+1}$ inside it.
\item We have $\delta^n \leq \diam(T) \leq K\delta^n$ for each $T\in\mathcal{T}_n$, where $K$ is a constant depending only on $N$.
\item Each element of $\mathcal{T}_{n}$ contains at least two and at most $M(N,\delta)$ elements of $\mathcal{T}_{n+1}$.
\item Each element of $\mathcal{T}_{n}$ intersects at most $M(N,\delta)$ other elements of $\mathcal{T}_n$.
\end{enumerate}
\end{prop}

\begin{proof}
The first four items appear explicitly in \cite[Lemma 5.1]{BM}. The fifth appears in \cite[Equation (5.3)]{BM}. The existence of the upper bound $M(N,\delta)$ in (6) and (7) is an immediate consequence of (1)-(5) and the doubling property, as in \cite[Lemma 5.7]{BM}. The lower bound of two in (6) follows from (4) and (5) if $\delta<1/K$.
\end{proof}

Bonk and Meyer refer to the elements of $\mathcal{T}_n$ as ``$n$-tiles'', but we will reserve the word ``tiles'' for the modifications we construct below. Before that, we observe that adjacency graphs induced by these sets form combinatorial trees.

\begin{lem}\label{lem:tiletree}
Let $X$ be a metric tree. Let $\mathcal{S}$ be a finite collection of compact, connected subsets of $X$ such that $ \cup_{S\in\mathcal{S}} S = X$ and no point of $X$ is in more than two different sets of $\mathcal{S}$. 

Then the graph $G$ such that
$$ V(G) = \{ S\in\mathcal{S}\},$$
$$ E(G) = \{ \{S,S'\}\subseteq V(G) : S\neq S' \text{ and }  S \cap S' \neq \emptyset\}$$
is a combinatorial tree.
\end{lem}

\begin{proof}
The connectedness of $G$ follows easily from the facts that $X$ is connected, all $S\in\mathcal{S}$ are compact, and $\cup_{S\in\mathcal{S}} S = X$.

To see that $G$ is a combinatorial tree, we will use the following simple equivalent characterization of combinatorial trees: A connected finite graph is a combinatorial tree if and only if the removal of any edge disconnects it.

Thus, suppose that the removal of an edge $\{S,S'\}$ from $G$ left it connected. Let $S=S_0, S_1, \dots, S_n=S'$ be the ordered vertices along a simple path from $S$ to $S'$ in $G$ avoiding this edge; note that $n\geq 2$. Let $x\in S\cap S'$, $p\in S\cap S_1$, and $q\in S'\cap S_{n-1}$. The points $x$, $p$, and $q$ are distinct, by the assumption that no point is in more than two elements of $\mathcal{S}$. Similarly, $x$ is disjoint from $S_{i}$ for each $1\leq i \leq n-1$.

There is an arc from $p$ to $q$ in $S \cup S'$, which must pass through $x$. Since $X$ is a metric tree, $p$ and $q$ must be in distinct connected components of $X\setminus \{x\}$. On the other hand, $\cup_{i=1}^{n-1} S_i$ is a connected subset of $X\setminus \{x\}$ containing both and we reach a contradiction.
\end{proof}

We now modify the construction of Proposition \ref{prop:BM} so that each tile has an equal number of children. This requires us to give up some control on the diameters of the tiles. However, it is crucial to retain the property that the boundary points of a given tile are ``well-separated'', in the sense that the distance between two distinct boundary points of a tile is always comparable to the diameter of the tile. This is property (6) of Lemma \ref{lem:tiles} below.

Fix $\delta$ sufficiently small, depending on $N$, so that Proposition \ref{prop:BM} holds, and so that in addition $K\delta< 1/2$, where $K$ is the constant from Proposition \ref{prop:BM}(5). Thus, we have constants $K=K(N)$ and $M(N,\delta)$ from Proposition \ref{prop:BM}, items (5) and (6).

\begin{lem}\label{lem:tiles}
Let $M\geq M(N,\delta)$, $K_1\in (0,K^{-1}\delta^{\log_2(M)+1}]$, and $K_2\in [\frac{1}{2},1)$. Let $A=\{1,\dots,M\}$.

Then there is a collection of closed subsets $X_w\subset X$, for all $w\in A^*$, satisfying the following properties.
\begin{enumerate} 
\item For each $w\in A^*$, $X_w$ is a connected subset (hence subtree) of $X$, and $X_{\varepsilon} =X$.
\item For each $w\in A^*$ and $i\in A$, $X_{wi}\subseteq X_w$. Moreover, $X_w = \bigcup_{i\in A}X_{wi}$.
\item For each $w\in A^*$ and $i\in A$,
\[ K_1\diam{X_{w}} \leq \diam{X_{wi}} \leq K_2\diam{X_w}.\]
\item For each $w\in A^{*}\setminus\{\varepsilon\}$ and every $x\in X_{w}\cap \overline{X\setminus X_w}$, we have that $x$ is a leaf of $X_w$ and contained in $X_{w'}$ for exactly one $w'\in A^{|w|}\setminus \{w\}$. 
\item For every distinct $w,w'\in A^*$ with $|w|=|w'|$ we have that $X_{w}\cap X_{w'}$ is either a point or empty.
\item There exists $K_3 \in (0,1)$ such that for all $w\in A^*$ and for all distinct $x,y \in X_{w}\cap \overline{X\setminus X_w}$, we have
\[ d(x,y) \geq K_3 \diam{X_w}.\]
\end{enumerate}
\end{lem}

\begin{proof}
Fix $\delta, M,  K, K_1, K_2$ as above, and let $A=\{1,\dots,M\}$. We prove the lemma for $K_3 = \delta/K$.

We start the proof by noting that it suffices to prove the lemma with (5) replaced by the property:
\begin{enumerate} 
\item[(5')] For every $w\in A^*$ and distinct $i,j \in A$ we have that $X_{wi}\cap X_{wj}$ is either a point or empty.
\end{enumerate}
Indeed, assume that the lemma holds with (5) replaced by (5'). Given distinct $w,w'\in A^*$ with $|w|=|w'|$, there exists maximal (in word-length) $w_0 \in A^*$ such that $w,w'\in A_{w_0}^{|w|}$. There also exist distinct $i,j\in A$ such that $w \in A^{|w|}_{w_0i}$ and $w \in A^{|w|}_{w_0j}$. By (2) and (5'), we have that $X_w \cap X_{w'} \subset X_{w_0j} \cap X_{w_0j}$ which is either a point or empty.

We relabel the collections $\mathcal{T}_n$ constructed in Proposition \ref{prop:BM}. Set $T_{\varepsilon}=X$. We write $\mathcal{T}_1 = \{T_1,\dots,T_{m_{\varepsilon}}\}$. Assume now that for some $n\in\N$ and some $w\in \N^n$ we have defined $T_w$ to be an element of $\mathcal{T}_n$. Then we write $\{T_{w1},\dots, T_{wm_w}\}$ to be the elements of $\mathcal{T}_{n+1}$ contained in $T_w$. By Proposition \ref{prop:BM}(6), we have $2 \leq m_w \leq M$. Therefore, for every $T_w$ defined, we have $w\in A^*$. We set $\mathcal{W}$ to be the set of all words $w$ in $A^*$ for which $T_w$ has been defined. Given integer $n\geq 0$ and $w\in A^*$ we denote $\mathcal{W}^n = \mathcal{W}\cap A^n$, $\mathcal{W}_w = \mathcal{W} \cap A^*_w$, and $\mathcal{W}^n_w = \mathcal{W}\cap A^n_w$.

We now define the family $\{X_w\}_{w\in A^*}$ in an inductive manner.

\medskip

\textsc{Step 0.} Set $X_{\varepsilon} = T_{\varepsilon} = X$.

\medskip

\textsc{Inductive hypothesis.} Suppose that for some integer $k\geq 0$ we have defined closed sets $\{X_w\}_{w\in A^k}$ such that the properties of the lemma up to level $k$ hold, with $K_3=\delta/K$. That is, we assume that the following conditions hold.
\begin{enumerate}
\item For each $l\leq k$ and $w\in A^l$, $X_w$ is a connected subset of $X$.
\item For each $l\leq k-1$, $w\in A^{l}$, and $i\in A$, we have $X_{wi}\subseteq X_w$. Moreover, $X_w = \bigcup_{i\in A}X_{wi}$.
\item For each $l\leq k-1$, $w\in A^{l}$, and $i\in A$, we have
\[ K_1\diam{X_{w}} \leq \diam{X_{wi}} \leq K_2\diam{X_w}.\]
\item For each $l\leq k$, $w\in A^{l}\setminus\{\varepsilon\}$ and every $x\in X_{w}\cap \overline{X\setminus X_w}$, we have that $x$ is a leaf of $X_w$ and contained in $X_{w'}$ for exactly one $w'\in A^{l}\setminus \{w\}$. 
\item For each $l\leq k-1$, $w\in A^l$ and distinct $i,j \in A$ we have that $X_{wi}\cap X_{wj}$ is either a point or empty.
\item For each $l\leq k$, $w\in A^l$, and distinct $x,y \in X_{w}\cap \overline{X\setminus X_w}$, we have
\[ d(x,y) \geq (\delta/K) \diam{X_w}.\]
\end{enumerate}
In addition, we make the following inductive assumption:
\begin{enumerate}
\item[(7)] For each $w\in A^k$ there exists $u\in\mathcal{W}$ and distinct $ui_1,\dots, ui_q \in \mathcal{W}_u^{|u|+1}$ such that $X_w = \bigcup_{j=1}^q T_{ui_j}$.
%\item [(8)] If $\ell \leq k$, $w\in A^l$, and $i\in A$, then the assignments $w\mapsto u$ and $w\mapsto u'$ given by (7) have the property that $u'$ is a descendent of $u$, i.e., $u'(|u|)=u$. 
\end{enumerate}
Note that (7) holds when $k=0$.

\medskip

\textsc{Inductive step.} We now describe the construction of the sets $\{X_w\}_{w\in A^{k+1}}$.  Fix a word $w\in A^k$. By assumption (7), $X_w = T_{ui_1}\cup\cdots \cup T_{u i_q}$. For simplicity, we assume that $i_j = j$ for all $j$. By Proposition \ref{prop:BM}(6), $q \leq M$.

\emph{Case 1: $q=M$.} In this case we set $X_{wj}=T_{uj}$ for $j=1,\dots,M$.

\emph{Case 2: $q < M$.} Let $n$ be the smallest integer such that 
\begin{equation}\label{eq:tilesum}
\sum_{j=1}^q\card(\mathcal{W}^{n+|u|}_{uj}) \geq M.
\end{equation}
By Proposition \ref{prop:BM}(6), $2 \leq n \leq \log_2{M}+1$.

\emph{Case 2.1 : the sum in (\ref{eq:tilesum}) is equal to $M$.} In this case we set
\[ \{X_{wi} : i\in A\} := \left\{T_{v} : v \in \bigcup_{j=1}^q\mathcal{W}^{n+|u|}_{uj}\right\}.\]

\emph{Case 2.2: the sum in (\ref{eq:tilesum}) is strictly greater than $M$.} Enumerate the elements of $\bigcup_{j=1}^q\mathcal{W}^{n-1+|u|}_{uj} =\{u_1,\dots,u_r\}$ so that for each $i\in\{1,\dots,r\}$ the set
\[ T_{u_i}\cap \overline{X_w \setminus (T_{u_1}\cup\cdots \cup T_{u_{i}})}\]
contains only one point. In other words, the sets $X_w\setminus T_{u_1}$, $(X_w\setminus T_{u_1})\setminus T_{u_2}$, etc. are connected. That this is possible follows from Lemma \ref{lem:tiletree} and the fact that every finite combinatorial tree has a leaf.

By minimality of $n$, we have that $r<M$. Now let $m$ be the smallest integer in $\{1,\dots,r\}$ such that 
\begin{equation}\label{eq:tilesum2} 
\sum_{i=1}^m \card(\mathcal{W}_{u_i}^{n+|u|}) + (r-m) \geq M.
\end{equation}
Note that if $m=r$, then \eqref{eq:tilesum2} holds by \eqref{eq:tilesum}, so such a minimal $m$ exists.

\emph{Case 2.2.1: the sum in (\ref{eq:tilesum2}) is equal to $M$}. Then, by the assumption of Case 2.2, we have $m<r$ and we set
\[ \{X_{wi} : i\in A\} := \left\{T_{v} : v \in \bigcup_{j=1}^m\mathcal{W}^{n+|u|}_{u_i} \cup\{u_{m+1},\dots, u_r\}\right\}.\] 

\emph{Case 2.2.2: the sum in (\ref{eq:tilesum2}) is strictly greater than $M$.} As before, enumerate the elements of $\mathcal{W}_{u_m}^{n+|u|} = \{u_mi_1, \dots, u_mi_l\}$ so that for each $j\in\{1,\dots,l\}$ the set
\[ T_{u_mi_j}\cap \overline{T_{u_m} \setminus (T_{u_mi_1}\cup\cdots \cup T_{u_{m}i_j})}\]
contains only one point. 

By the minimality of $m$ (and the fact that $r<M$) we have
$$\sum_{i=1}^{m-1} \card(\mathcal{W}_{u_i}^{n+|u|}) + (r-(m-1)) \leq M-1$$
and so
\begin{equation}\label{eq:p}
\sum_{i=1}^{m-1} \card(\mathcal{W}_{u_i}^{n+|u|}) + (r-m) \leq M-2.
\end{equation}
Let 
$$p= M-1 - (r-m) - \sum_{i=1}^{m-1} \card(\mathcal{W}_{u_i}^{n+|u|}).$$
Note that $p\geq 1$ by \eqref{eq:p}. Moreover, $p\leq l-1=\card{\mathcal{W}_{u_m}^{n+|u|}}-1$, otherwise
$$ \sum_{i=1}^{m} \card(\mathcal{W}_{u_i}^{n+|u|}) + (r-m) \leq M-1,$$
contradicting \eqref{eq:tilesum2}.

Define now
\[ \mathcal{U} := \bigcup_{i=1}^{m-1}\mathcal{W}_{u_i}^{n+|u|} \cup \{u_mi_1,\dots, u_mi_p\} \cup \{u_{m+1},\dots,u_r\}. \]
%\[ \mathcal{U} := \bigcup_{i=1}^{m-1}\mathcal{W}_{u_i}^{n+|u|} \cup \{u_mi_1,\dots, w_mi_p\} \cup \{w_{m+1},\dots,w_r\}. \]
Note that $\card(\mathcal{U}) = M-1$ by choice of $p$. Set 
\[ \{X_{wi} : i\in A\} := \{T_v : v\in\mathcal{U}\} \cup \{\overline{T_{u_m} \setminus (T_{u_mi_1}\cup\cdots \cup T_{u_{m}i_p})}\}.\]

To complete the inductive step and the proof of Lemma \ref{lem:tiles}, it remains to check that the inductive properties (1)--(7) above are satisfied up to level $k+1$.

Property (1) holds: If $w\in A^k$, each $X_{wi}$ is either equal to some $T_v$ constructed in Proposition \ref{prop:BM}, and hence connected by Proposition \ref{prop:BM}(1), or (as is possible in Case 2.2.2) is a connected union of finitely many such $T_v$.

It also straightforward to check that property (7) holds. In cases 1, 2.1, and 2.2.1 of the construction, each $X_{wi}$ for $wi\in A^{k+1}$ is exactly equal to some set $T_u$ as constructed in Proposition \ref{prop:BM}, and therefore is a finite union of sets $T_{uj}$. In case 2.2.2, there is also the possibility that $X_{wi}$ is of the form $\overline{T_{u_m} \setminus (T_{u_mi_1} \cup \dots \cup T_{u_mi_p})}$, where $u_m\in\mathcal{W}$ and $i_k\in A$. In that case, $X_{wi}$ is also equal to a finite union of children of $T_{u_m}$, namely $\{T_{u_m k}: k\neq i_1, \dots, i_p\}$.

%Property (8) also holds. Let $w\in A^k$ and $i\in A$. If $u\in\mathcal{W}$ was assigned to $w\in A*$ by (7), then in each case the element of $\mathcal{W}$ that (7) assigns to $wi$ is a descendent of $u$.

To see that property (2) holds, set $w\in A^k$. In the construction of $\{X_{wi}:i\in A\}$, we write $X_w$ as a finite union $T_{u1} \dots T_{uq}$, where these sets come from Proposition \ref{prop:BM}. In each case, the sets $X_{wi}$ are constructed to be subsets of these $T_{uj}$ and exhaust each of them.

For property (4), set $wi\in A^{k+1}$ and $x\in X_{wi} \cap \overline{X\setminus X_{wi}}$. The construction of $X_{wi}$ and Proposition \ref{prop:BM}(2,3) ensures that $x$ is contained in at most one other $X_{wj}$ ($j\neq i$) and is a leaf of $X_{wi}$ in this case. 

If $x\in X_{wi} \cap X_{w'i'}$ for some $w\neq w'\in A^k$, then by induction $x$ is a leaf of $X_w$ and hence of $X_{wi}$. Moreover, in this case $x$ cannot be contained in any other $X_{w''}$ by induction, or in any other element $X_{wj}$ ($j\neq i$), since a leaf of $X_w$ can only be in one of the non-trivial connected subsets $X_{wj}$.

To see that property (5) holds, consider $w\in A^k$ and the set $X_{wi}\cap X_{wj}$ (for $i\neq j$). By (1), this intersection is either empty, a point, or a non-trivial continuum. By construction, each of the two sets $X_{wi} \cap X_{wj}$ is a finite union of distinct elements of some $\mathcal{T}_n$ constructed in Proposition \ref{prop:BM}, and so the intersection cannot be a continuum by Proposition \ref{prop:BM}(2).

%First, by design, properties (2) and (7) are immediately satisfied. Property (1) follows from Proposition \ref{prop:BM}(1) and the design of the enumerations $\{u_1,\dots,u_r\}$ and $\{u_mi_1, \dots, u_mi_l\}$. Properties (4), (5) for level $k+1$ follows from properties (4), (5) for level $k$, Proposition \ref{prop:BM}(3), and the design of the two enumerations $\{u_1,\dots,u_r\}$ and $\{u_mi_1, \dots, u_mi_l\}$.

For property (3), fix $w\in A^k$ and $i\in A$. By (7), there exists $u\in \mathcal{W}^l$ and $uj \in \mathcal{W}^{l+1}$ such that $T_{uj} \subset X_w \subset T_u$. By the design above, there exists $v \in \mathcal{W}^{l+n}_u$ and $vj' \in \mathcal{W}^{l+n+1}_u$ such that $T_{vj'} \subset X_{wi} \subset T_{v}$ and $2 \leq n \leq \log_2{M}+1$. Therefore, applying Proposition \ref{prop:BM}(5),
\[ K_1\leq K^{-1}\d^{\log_2{M}+1} \leq \frac{\diam{X_{wi}}}{\diam{X_w}} \leq K\delta \leq K_2.\]

Finally, for property (6), fix $w\in A^{k+1}$ and distinct $x,y \in X_w \cap \overline{X \setminus X_w}$. By (7), we know that $X_w = T_{u i_1}\cup \cdots \cup T_{u i_n}$ for some $u\in \mathcal{W}^l$ and $u i_1, \dots, u i_n \in \mathcal{W}^{l+1}$. By Proposition \ref{prop:BM}(1), $x,y$ have distance at least $\delta^{l+1}$ so
\[ \dist(x,y) \geq \d^{l+1} \geq (\d/K) \diam{T_u} \geq (\d/K)\diam{X_w}. \qedhere\]
\end{proof}

We call the sets $X_w$ constructed in Lemma \ref{lem:tiles} ``tiles''. We observe that these new tiles also maintain the property that they can only touch a controlled number of tiles of the same scale:

\begin{lem}\label{lem:tiledoubling}
There is a constant $n_0$, depending only on the doubling constant of $X$ and the constants from Lemma \ref{lem:tiles}, such that if $w\in A^*$, then
$$ \card\{ v\in A^{|w|}: v\neq w, X_v \cap X_w \neq \emptyset\} \leq n_0.$$ 
\end{lem}
\begin{proof}
Let
$$ W = \{ v\in A^{|w|}: v\neq w, X_v \cap X_w \neq \emptyset\}.$$
For each $v\in W$, Lemma \ref{lem:tiles}(5) implies that $X_w\cap X_v$ is a single point, which we call $x_{v}\in X_w \cap X_v$. Moreover, if $v,v'\in W$ and $v\neq v'$, then $x_v,x_{v'} \in X_w \cap \overline{X \setminus X_w}$. By property (4) of Lemma \ref{lem:tiles} we have that $x_v \neq x_{v'}$, and by property (6) we have that 
$$ d(x_v, x_{v'}) \geq K_3 \diam(X_w).$$
Since all the points $\{x_v: v\in W\}$ are contained in $X_w$, the doubling property of $X$ completes the proof. 
\end{proof}

\subsection{Definition of combinatorial data}\label{subsec:charcombdata}

Fix $\delta$ as above Lemma \ref{lem:tiles} and apply Lemma \ref{lem:tiles} with fixed parameters $M\in\N$ and $K_1, K_2\in (0,1)$ as in the statement of that lemma. Let $A=\{1, \dots, M\}$. We define combinatorial data $\mathscr{C}=(A,(T_k)_{k\in\N})$ by setting $T_k=(A^k, E_k)$, where two words $v,w$ of $A^k$ are adjacent if and only if $X_v \cap X_w \neq \emptyset$.

\begin{lem}
$\mathscr{C}$ satisfies the conditions of Definition \ref{def:combdata}, and each graph $T_k$ is a combinatorial tree.
\end{lem}
\begin{proof}
Property \eqref{eq:combdata1} of Definition \ref{def:combdata} is immediate. That $T_k$ is a (connected) combinatorial tree follows from Lemma \ref{lem:tiletree}.

Property \eqref{eq:combdata2a} of Definition \ref{def:combdata} holds similarly, taking $X=X_w$, which is connected, and again using Lemma \ref{lem:tiletree}.

For Property \eqref{eq:combdata2b}, consider $\{w,u\}\in E_k$. Then there is a point $x\in X_w\cap X_u$. By Lemma \ref{lem:tiles}(2), there are words $wi$ and $uj$ such that $x\in X_{wi} \cap X_{uj}$, and therefore $\{wi, uj\}\in E_{k+1}$.
\end{proof}

One basic consequence of this construction of combinatorial data is the following.
\begin{lem}\label{lem:tiletouch}
If $w,u\in A^*$ and $A^\N_w \wedge_\mathscr{C} A^N_u \neq \emptyset$, then $X_w \cap X_u \neq \emptyset$.
\end{lem}
\begin{proof}
Let $w,u\in A^*$ with $A^\N_w \wedge_\mathscr{C} A^N_u \neq \emptyset$. By Lemma \ref{lem:intersection1}, there are then $k\in\mathbb{N}$, $w'\in A^{k}_w$, and  $u'\in A^{k}_u$ with $\{w',u'\}\in E_k$.   It follows from the definition of $\mathscr{C}$ that $X_{w'} \cap X_{u'}\neq \emptyset$, $X_{w'}\subseteq X_w$, and $X_{u'}\subseteq X_u$. This proves the lemma.
\end{proof}

\subsection{Definition of diameter function}
We continue to use the quasiconformal tree $X$ fixed at the start of Section \ref{sec:characterization}, and the constants $M,K_1, K_2$ and combinatorial data $\mathscr{C}=(A, (T_k)_{k\in\N})$ fixed at the start of Section \ref{subsec:charcombdata}.

We now define a diameter function $\D\in \mathscr{D}(A,K_1,K_2)$ with the following two rules.
\begin{itemize}
\item $\D(\varepsilon) =1$.
\item Suppose that for some $w\in A^*$ we have defined $\D(w)$.
\begin{enumerate}
\item If $\D(w) \leq \diam{X_w}$, then we define $\D(wi)= K_2\D(w)$ for all $i\in A$.
\item If $\D(w) > \diam{X_w}$, then we define $\D(wi)=K_1 \D(w)$ for all $i\in A$. 
\end{enumerate}
\end{itemize}
This satisfies Definition \ref{def:diam}, with property (3) following from the fact that $K_1<K_2<1$.

We now show that $\Delta(w)$ is always comparable to $\diam(X_w)$. This argument is very similar to the proof of Theorem A in \cite[\textsection4.1]{HM}.

\begin{lem}
For all $w\in A^*$, 
\begin{equation}\label{eq:comparable}
(K_2/K_1)^{-1} \D(w) \leq \diam(X_w) \leq (K_2/K_1)\D(w).
\end{equation}
\end{lem}
\begin{proof}
By Lemma \ref{lem:tiles}(3) we have for all $w\in A^*$
$$ K_1\diam(X_w) \leq \diam(X_{wi}) \leq K_2\diam(X_w).$$

Note that \eqref{eq:comparable} holds for $w=\varepsilon$, since $\Delta(\varepsilon)=\diam(X_\varepsilon)=1$. Assume by induction that we have a word $w$ such that \eqref{eq:comparable} holds. Consider any $i\in A$. There are two possibilities.

\emph{Case 1}: $\D(w) \leq \diam(X_w)$. In this case, we have
$$ \D(wi) = K_2\D(w) \leq K_2\diam(X_w) \leq (K_2/K_1)\diam(X_{wi})$$
and
$$ \diam(X_{wi}) \leq K_2\diam(X_w) \leq K_2(K_2/K_1)\D(w) =(K_2/K_1)\D(wi),$$
which together prove \eqref{eq:comparable} for the word $wi$ in case 1.

\emph{Case 2}: $\D(w) > \diam(X_w)$. In this case, we have
$$ \D(wi) = K_1\D(w) \leq K_1(K_2/K_1)\diam(X_w) \leq (K_2/K_1)\diam(X_{wi})$$
and
$$ \diam(X_{wi}) \leq K_2\diam(X_w) < K_2\D(w) = (K_2/K_1)\D(wi),$$
which together prove \eqref{eq:comparable} for the word $wi$ in case 2.
\end{proof}

As in Section \ref{sec:diameter}, let $\sim$ be the equivalence relation on $A^{\N}$ induced by the diameter function $\D$ and let $\mathcal{A}=A^{\N}/\sim$ and $\mathcal{A}_w = A^{\N}_w/\sim$.

\subsection{Proof of Theorem \ref{thm:main}}\label{sec:thmproof}

A consequence of Lemma \ref{lem:tiles}(2) is that, for each $x\in X$, there exists an infinite word $w_x\in A^{\mathbb{N}}$ such that $x \in X_{w(n)}$ for all $n\in\N$. We therefore define a map $f\colon X \rightarrow \mathcal{A}$ by $f(x) = [w_x]$. 

\begin{lem}\label{lem:welldef}
The map $f: X \to \mathcal{A}$ defined above is well-defined and surjective. 
\end{lem}
\begin{proof}
Suppose that there exist two words $w,u \in A^\N$ such that for all $n\in\N$, $x\in X_{w(n)}\cap X_{u(n)}$. Then, by the construction of the combinatorial data $\mathscr{C}$, for each $n\in\N$ we have $\{w(n),u(n)\} \in E_n$.  (Recall that $E_n$ is the set of edges of $T_n$.) Thus, for each $n\in\N$, the set $\{A^{\N}_{w(n)}, A^{\N}_{u(n)}\}$ is a chain that joins $w$ with $u$, and so $d_{\mathscr{C},\D}([w],[u]) \leq \D(w(n)) + \D(u(n)) \rightarrow 0$ as $n\rightarrow\infty$. We therefore have that $d_{\mathscr{C},\D}([w],[u]) = 0$, which implies that $[w]=[u]$. This shows that $f$ is well-defined.

To show that $f$ is surjective, consider an arbitary $[u]\in\mathcal{A}$. We have nested compact tiles
$$ X_{u(1)} \supseteq X_{u(2)} \supseteq X_{u(3)} \dots$$
in $X$. Let $x\in \cap_{n\in\mathbb{N}} X_{u(n)}$. If $f(x)=w\in\mathcal{A}$, then by definition of $f$ we have
$$ x\in X_{w(n)} \cap X_{u(n)} \text{ for all } n\in\mathbb{N}.$$
As before, $u(n)$ and $w(n)$ are adjacent in $T_n$ for each $n$, and hence again
$$ d_{\mathscr{C},\Delta}([u],[w]) \leq \Delta(u(n))+\Delta(w(n))\rightarrow 0.$$
Thus, $[u]=[w]=f(x)$ and $f$ is surjective.
\end{proof}

The proof of Theorem \ref{thm:main} concludes with the next two results.

\begin{prop}\label{prop:bilipschitz}
The map $f: (X,d) \to (\mathcal{A},d_{\mathscr{C},\D})$ is bi-Lipschitz, with constant depending only on $K_1$, $K_2$, and $K_3$. 
\end{prop}

\begin{proof}
Fix $x,y \in X$. 

We first claim that $d_{\mathscr{C},\D}(f(x),f(y)) \geq \frac{K_1}{K_2} d(x,y)$. Suppose that $f(x)=[w]$ and $f(y) = [u]$. Let $\{ A^{\N}_{w_1},\dots, A^{\N}_{w_m}\}$ be a chain joining $w$ with $u$. Since $w\in A^\N_{w_1}$, we have $w_1=w(|w_1|)$ and therefore $x\in X_{w_1}$; similarly, $y\in X_{w_m}$.

We also have $X_{w_i}\cap X_{w_{i+1}} \neq \emptyset$ for each $i\in\{1,\dots,m-1\}$, by Lemma \ref{lem:tiletouch}.

Therefore, using the triangle inequality and \eqref{eq:comparable}, we have
\begin{equation}\label{eq:diameters}
\sum_{i=1}^m \D(w_i) \geq \frac{K_1}{K_2}\sum_{i=1}^m \diam{X_{w_i}} \geq \frac{K_1}{K_2}d(x,y).
\end{equation}
Taking the infimum over all possible chains, we obtain $d_{\mathscr{C},\D}(f(x),f(y)) \geq \frac{K_1}{K_2}d(x,y)$, as desired.

We now claim that 
\begin{equation}\label{eq:lipschitz}
d_{\mathscr{C},\D}(f(x),f(y)) \lesssim d(x,y),
\end{equation}
with implied constant depending only on $K_1, K_2, K_3$.

Let $w_0$ be a word in $\mathcal{W}$ of maximal length such that $x,y \in X_{w_0}$. Then, there exists distinct $i,j \in A$ such that $w_0i,w_0j \in \mathcal{W}$, $x\in X_{w_0i}$ and $y\in X_{w_0j}$. Set $k=|w_0|$. We consider the following two possible cases.

Suppose first that $X_{w_0i}\cap X_{w_0j} = \emptyset$. Let $\gamma$ be the unique arc in $X$ with endpoints $x,y$. Note that $\gamma\subseteq X_{w_0}$ as $X_{w_0}$ is connected. Assuming $X_{w_0i}\cap X_{w_0j} = \emptyset$, it follows that $\gamma \setminus (X_{w_0i}\cup X_{w_0j})$ is a non-empty relatively open subset of $\gamma$. There must therefore exist some $l \in A\setminus\{i,j\}$ such that $\gamma \cap \partial X_{w_0l}$ contains two distinct points $v,v'$ of $\partial X_{w_0l}$.

By the $1$-bounded turning property of $X$ and Lemma \ref{lem:tiles}(6),
\[ d(x,y) \geq \diam{\g} \geq d(v,v') \geq K_3\diam(X_{w_0l}).\]
On the other hand, $f(x),f(y) \in \mathcal{A}_{w_0}$ and so, by Lemma \ref{lem:dist} and \eqref{eq:comparable}, we have:
\[ d_{\mathscr{C},\D}(f(x),f(y)) \leq \diam{\mathcal{A}_{w_0}} \leq \D(w_0) \leq \frac{K_2}{K_1}\diam(X_{w_0}).\]
Therefore, using Lemma \ref{lem:tiles}(3),
$$ d(x,y) \geq K_3\diam(X_{w_0l}) \geq K_3 K_1 \diam(X_{w_0}) \geq \frac{K_1^2 K_3}{K_2} d_{\mathscr{C},\D}(f(x),f(y)).$$
This completes the proof of \eqref{eq:lipschitz} in the case where $X_{w_0i}\cap X_{w_0j} = \emptyset$.

Suppose now that  $X_{w_0i}\cap X_{w_0j} \neq \emptyset$. Find words $w,u \in A^*$ of maximal lengths such that $w_0w, w_0u \in \mathcal{W}^*$, $x\in X_{w_0w}$, $y\in X_{w_0u}$ and $X_{w_0w}\cap X_{w_0u} \neq \emptyset$. Then there exist $w_0wi, w_0uj \in A^*$ such that $X_{w_0wi}\cap X_{w_0u} = \emptyset$,  $X_{w_0uj}\cap X_{w_0w} = \emptyset$, $x\in X_{w_0wi}$ and $y\in X_{w_0uj}$.

Let $z$ be the unique point of $X_{w_0w}\cap X_{w_0u}$ and again set $\gamma$ to be the unique arc from $x$ to $y$ in $X$, which must pass through $z$. Choose $k\in A$ such that $z\in X_{w_0wk}$. Note that $k\neq i$ by the maximality of $w$, and that $z\in \partial X_{w_0 w k}$. The sub-arc of $\gamma$ from $x$ to $z$ must also contain a point $v\in \partial X_{w_0 w k}$ distinct from $z$, by Lemma \ref{lem:tiles}(4).

Hence, again by $1$-bounded turning and Lemma \ref{lem:tiles}(6), we have
$$ d(x,z) \geq d(v,z) \geq K_3 \diam(X_{w_0 w k}).$$
Similarly, 
$$d(y,z) \geq K_3 \diam(X_{w_0 u l}),$$
for some $l\in A$.

By the $1$-bounded turning property and Lemma \ref{lem:tiles}(3),
$$ d(x,y) \geq \frac{1}{2}(d(x,z)+d(y,z)) \geq \frac{1}{2}K_3(\diam(X_{w_0 w k}) + \diam(X_{w_0 u l})) \geq \frac{1}{2}K_3K_1(\diam(X_{w_0 w}) + \diam(X_{w_0 u})).$$

On the other hand, $f(x) \in \mathcal{A}_{w_0w}$, $f(y) \in \mathcal{A}_{w_0u}$ and $\{A^{\N}_{w_0w}, A^{\N}_{w_0u}\}$ is a chain joining $f(x)$ and $f(y)$. Therefore, by Lemma \ref{lem:dist} and by \eqref{eq:comparable}
\[ d_{\mathscr{C},\D}(f(x),f(y)) \leq \diam{\mathcal{A}_{w_0w}} + \diam{\mathcal{A}_{w_0w}} \leq \D(w_0w) + \D(w_0u) \leq \frac{K_2}{K_1}(\diam(X_{w_0 w}) + \diam(X_{w_0 u})).\]

Therefore, 
$$ d_{\mathscr{C},\D}(f(x),f(y)) \leq \frac{2K_2}{K_1^2 K_3} d(x,y)$$%\frac{2}{K_2 K_3} d(x,y).$$
This completes the proof of \eqref{eq:lipschitz} and hence of the proposition.
\end{proof}

Finally, to prove the ``moreover'' piece of Theorem \ref{thm:main}, we now show:
\begin{lem}
The combinatorial data $\mathscr{C}$ and diameter function $\Delta$ defined above satisfy the conditions of Proposition \ref{prop:doubling} for some choice of $N,n_0,c,\delta_1,\delta_2$
\end{lem}
\begin{proof}
Property \ref{P1} of Proposition \ref{prop:doubling} follows from our choice of a finite alphabet $A=\{1,\dots,M\}$. Property \ref{P2} follows from Lemma \ref{lem:tiledoubling} and the definition of the combinatorial trees $T_k$ in our combinatorial data. Property \ref{P3} is immediate from our construction of $\Delta$, with $\delta_1=K_1$ and $\delta_2=K_2$.

It remains to verify Property \ref{P4} of Proposition \ref{prop:doubling}. Consider $k\in\N$ and distinct $u,u_1,u_2 \in A^*$ such that $\{u,u_1\}$ and $\{u,u_2\}$ are in $E_{n}$. Let also $w_1,w_2 \in A^\N_u$, $v_1 \in A^\N_{u_1}$, and $v_2 \in A^\N_{u_2}$ such that for all $n\geq k$ and $i\in\{1,2\}$, $\{w_i(n), v_i(n)\}\in E_n$.

For each $i\in\{1,2\}$, let $x_i\in X$ denote the unique point such that
$$ x_i \in \bigcap_{n=0}^\infty X_{w_i(n)}.$$
By definition, we have $f(x_i) = w_i$. Notice that $x_1$ and $x_2$ are both in $X_u$ as $w_i\in A^\N_u$.

We first claim that, for $i\in\{1,2\}$,
\begin{equation}\label{eq:boundarypoint}
x_i \in X_{u_i} \cap X_u \subseteq \partial X_u.
\end{equation}
It follows from the definition of $\mathscr{C}$ that
$$ \emptyset \neq X_{w_i(n)} \cap X_{v_i(n)} \subseteq X_{w_i(n)} \cap X_{u_i}$$
for all $n>k$. Hence,
$$ \dist(x_i, X_{u_i}) \leq \diam(X_{w_i(n)}) \rightarrow 0 \quad \text{ as } n\rightarrow\infty,$$
and so $x_i \in X_u \cap X_{u_i} \subseteq \partial X_u$. 

We next claim that $x_1\neq x_2$. Suppose to the contrary that $x_1=x_2=x$, and choose $n>k$ such that $w_1(n)\neq w_2(n)$. Then $X_{w_1(n)}$ and $X_{w_2(n)}$ are distinct subsets of $X_u$ with $x\in X_{w_1(n)} \cap X_{w_2(n)}$. In addition, we showed in \eqref{eq:boundarypoint} that $x\in X_{u_1}$. It follows that there is an element $v\in A^n_{u_1}$ with $x\in X_{v}$. The word $v$, beginning as it does with $u_1\neq u$, is distinct from both $w_1(n)$ and $w_2(n)$, and so the three words $v$, $w_1(n)$, and $w_2(n)$ are distinct and of the same length $n$. Moreover, $x\in X_{w_1(n)} \cap X_{w_2(n)} \cap X_v$. However, this contradicts Lemma \ref{lem:tiles}(4).

Thus, $x_1$ and $x_2$ are distinct elements of $\partial X_u$. By Lemma \ref{lem:tiles}(6) and \eqref{eq:comparable}, 
$$ d(x_1,x_2) \geq K_3 \diam(X_u) \geq (K_3K_1/K_2) \Delta(u).$$
By Proposition \ref{prop:bilipschitz}, $f$ is bi-Lipschitz with constant depending only on $K_1, K_2, K_3$. Therefore
$$ d_{\mathscr{C},\Delta}([w_1],[w_2]) = d_{\mathscr{C},\Delta}(f(x_1),f(x_2)) \geq c \Delta(u),$$
for some $c$ depending only on $K_1, K_2, K_3$. This completes the proof.
\end{proof}

\section{Examples and simple cases of quasiconformal trees}\label{sec:examples}
In this section, we discuss some examples and simple special cases of quasiconformal trees based on our construction.

\subsection{Quasiarcs}\label{sec:ex-quasiarcs}
Here we discuss combinatorial data and diameter functions that give rise to quasiarcs. We start with a corollary in which the conditions of Proposition \ref{prop:doubling} can be verified, using Lemma \ref{lem:suffdoubling}.

\begin{lem}\label{lem:qarcs}
Let $\mathscr{C} = (A,(T_k)_{k\in\N})$ be combinatorial data such that $\card{A} = N \geq 2$ and each $T_k = (A^k,E_k)$ is a combinatorial arc. 
Let $\D \in \mathscr{D}(A)$ satisfy property \ref{P3} of Proposition \ref{prop:doubling} and assume that, for all $k\geq 0$ and $w\in A^k$,
\begin{equation}\label{eq:longchain} 
\sum_{wi \in A_w^{k+1}}\D(wi) \geq \D(w).
\end{equation}
Then, $(\mathcal{A},d_{\mathscr{C},\D})$ is a doubling bounded turning arc.
\end{lem}

\begin{proof}
First, since $\card{A}= M$,  \ref{P1} of Proposition \ref{prop:doubling} is immediately satisfied, and since each $T_k$ is a combinatorial arc, $\text{Val}(T_k)=2$ and condition \ref{P2} of Proposition \ref{prop:doubling} is also satisfied. Since $\card{A}\geq 2$, assumption (1) of Lemma \ref{lem:suffdoubling} is satisfied and by (\ref{eq:longchain}), assumption (2) of Lemma \ref{lem:suffdoubling} is satisfied. Hence, by Lemma \ref{lem:suffdoubling} and Proposition \ref{prop:comb-tree}, $(\mathcal{A},d_{\mathscr{C},\D})$ is doubling and bounded turning.

It remains to show that $(\mathcal{A},d_{\mathscr{C},\D})$ is an arc. By design, there exists exactly two words $w_1,w_2 \in A^{\N}$ such that for all $n\in\N$, the valency of $w_i(n)$ in $T_n$ is 1. Recalling the definition of $K$ from the proof of Lemma \ref{lem:tree}, we note that $K = \mathcal{A}$. Therefore,  $(\mathcal{A},d_{\mathscr{C},\D})$ is an arc.
\end{proof}

\begin{ex}\label{ex:quasiarc}
Let $M \in \{2,3,\dots\}$ and $A=\{1,\dots,M\}$. Let $\mathscr{C}_M=(A,(G_k)_{k\in\N})$ where for each $k\in\N$ the graph $G_k$ is a simple path with the following two rules:
\begin{enumerate}
\item For each $w\in A^*$ and $i\in\{1,\dots,M-1\}$ we have that $wi$ is adjacent to $wi'$, where $i'=i+1$.
\item If $wiv,wjv'\in A^*$ with $i<j$, $|v|=|v'|$ and $wiv$ is adjacent to $wjv'$, then $wivM$ is adjacent to $wjv'1$.
\end{enumerate}
In other words, each word in $A^k$ is simply adjacent to the following word in lexicographic order in $G_k$.

Let $\delta\in (M^{-1},1]$ and $\D \in \mathscr{D}(A,M^{-1},\d)$. We write $\mathcal{A} = A^{\N}/\sim$ and for each $w\in A^*$, $\mathcal{A}_w = A^{\N}_w/\sim$.

The following lemma summarizes some properties of this construction. 

\begin{lem}\label{lem:arcproperties}
\begin{enumerate}
\item Suppose $v,v'\in A^k$ with $v$ coming earlier than $v'$ in lexicographic order. Then $\mathcal{A}_{v} \cap \mathcal{A}_{v'}\neq \emptyset$ if and only if $v$ and $v'$ are adjacent in $G_k$.
\item In case (1), $[vM^\infty]=[v'1^\infty]$ is the unique element of $\mathcal{A}_{v} \cap \mathcal{A}_{v'}\neq \emptyset$.
\item For each $v\in A^*$, the set $\mathcal{A}_v$ is a topological arc with $M^{-1}\Delta(v)\leq\diam{\mathcal{A}_v}\leq \Delta(v)$.
\end{enumerate}
\end{lem}

\begin{proof}
We begin with (1). Suppose $v,v'\in A^k$, with $v$ preceding $v'$ in lexicographic order, and $\mathcal{A}_{v} \cap \mathcal{A}_{v'}\neq \emptyset$. This means that there are infinite words $w,w'$ with $[vw]=[v'w']$. Suppose $v$ and $v'$ were not adjacent; let $u$ be a word on the simple path $T_k$ between them (and hence lexicographically between $v$ and $v'$). Let $n\in\mathbb{N}$ be such that $\Delta(t)<\frac{1}{2}\Delta(u)$ for all $t\in A^n$.

Because $u$ is lexicographically between $v$ and $v'$, each $t\in A^n_u$ is lexicographically between $(vw)(n)$ and $(v'w')(n)$, and hence is on the unique simple path between $(vw)(n)$ and $(v'w')(n)$ in $T_n$. %Each $t\in A^n_u$ is on the unique simple path between $(vw)(n)$ and $(v'w')(n)$ in $T_{n}$.
By Lemma \ref{lem:tree1}, $[vw]$ and $[v'w']$ are both in $\mathcal{A}_{t}$ for each $t\in A^n_u$. In particular, all $\mathcal{A}_t$ for $t\in A^n_u$ share a common point. Therefore, by Lemma \ref{lem:dist} and our choice of $n$ above,
\begin{equation}\label{eq:arcequation}
 \diam(\mathcal{A}_u) \leq 2\max\{\diam(\mathcal{A}_{t}) : t\in A^n_u\} < \Delta(u).
\end{equation}

On the other hand, our combinatorial data $\mathscr{C}_M$ satisfies the assumptions of Lemma \ref{lem:suffdoubling}. Indeed, Lemma \ref{lem:suffdoubling}(1) holds because the graphs $G_k$ in $\mathcal{C}_M$ simply consist of arcs in lexicographical order, and Lemma \ref{lem:suffdoubling}(2) holds because any pair $u,u'$ of distinct vertices in some $\partial_{\mathcal{C}}A^{k+1}_w$ are separated by at least $M$ other vertices, each with diameter function giving weight $\geq M^{-1}\Delta(w)$. 

Therefore, by Lemma \ref{lem:suffdoubling}, $\diam(\mathcal{A}_u)=\Delta(u)$, which contradicts \eqref{eq:arcequation}.

This proves the ``forward direction'' of (1). For the other direction, it is immediate from the construction of $\mathscr{C}_M$ that if $v$ and $v'$ are adjacent in $T_k$, with $v$ lexicographically preceding, then for each $n\in\mathbb{N}$
$$d_{\mathscr{C},\Delta}(vM^\infty,v'1^\infty) \leq \Delta(vM^n) + \Delta(v'1^n) \rightarrow 0 \text{ as } n\rightarrow \infty, $$
and so $[vM^\infty]=[v'1^\infty] \in \mathcal{A}_v \cap \mathcal{A}_{v'}$.

For (2), suppose there was a point $p$ other than  $[vM^\infty]=[v'1^\infty]$ in $\mathcal{A}_v \cap \mathcal{A}_{v'}$. Then there would be an infinite word $w\in A^\mathbb{N}$, $w\neq M^\infty$, such that $p=[vw]$. Choose $n$ such that the $n$th letter of $w$ is not $M$. Then $vw(n)$ and $v'1^n$ are not adjacent in $T_{k+n}$, but $[vw]\in\mathcal{A}_{vw(n)} \cap \mathcal{A}_{v'1^n}$. This contradicts (1).

For fact (3), it is an immediate consequence of Remark \ref{rem:arc} that each $\mathcal{A}_v$ is a topological arc. The diameter of $\mathcal{A}_v$ is at most $\D(v)$ by Lemma \ref{lem:dist}. If $v$ has at least two neighbors in $T_{|v|}$, then $\diam(T_{|v|})=\Delta(v)$ by Lemma \ref{lem:suffdoubling}. Otherwise, $vi$ has at least two neighbors in $T_{|v|+1}$ for some $i\in A$, and so Lemma \ref{lem:suffdoubling} says that $\diam(\mathcal{A}_{vi}) = \Delta(vi)$. Therefore,
$$ \diam{\mathcal{A}_v} \geq \diam{\mathcal{A}_{vi}} = \Delta(vi) \geq M^{-1} \Delta(v).\qedhere$$
\end{proof}

\begin{prop}\label{prop:qarcs}
The space $(\mathcal{A},d_{\mathscr{C},\D})$ is a quasiarc.
\end{prop}

\begin{proof}
By Lemma \ref{lem:arcproperties} we know that $(\mathcal{A},d_{\mathscr{C},\D})$ is a topological arc, and by Theorem \ref{prop:comb-tree} we know that $(\mathcal{A},d_{\mathscr{C},\D})$ is bounded turning. Moreover, property \ref{P3} of Proposition \ref{prop:doubling} is satisfied and for any $w\in A^*$ and $i\in A$
\[ \sum_{i\in A}\D(wi) \geq \sum_{i=1}^M \frac1{M} = 1\]
and (\ref{eq:longchain}) holds. Therefore, by Lemma \ref{lem:qarcs}, $(\mathcal{A},d_{\mathscr{C},\D})$ is doubling.
\end{proof}

We note that a more refined statement holds; see Lemma \ref{lem:HM2}. Furthermore, the converse of Proposition \ref{prop:qarcs} is also true: every quasiarc is bi-Lipschitz equivalent to $(\mathcal{A},d_{\mathscr{C},\D})$ for some $\d \in [M^{-1},1)$ and some $\D \in \mathscr{D}(A,M^{-1},\d)$; see Proposition \ref{lem:HM}.
\end{ex}

\subsection{The Vicsek tree and variations}
Here we discuss a concrete example of a self-similar quasiconformal tree, the Vicsek tree, and how it can be viewed through our construction.

\begin{ex}\label{ex:vicsek} The \emph{Vicsek tree} $\mathbb{V}$ is defined as the attractor of the iterated function system $\{\phi_1,\dots,\phi_5\}$ on $\mathbb{C}$ with
\[ \phi_1(z) = \tfrac13(z-2+2i), \quad \phi_2(z) = \tfrac13(z+2+2i), \quad \phi_3(z) = \tfrac13(z+2-2i), \quad \phi_4(z) = \tfrac13 z, \quad \phi_5(z) = \tfrac13(z-2-2i).\]

Let $A = \{1,\dots,5\}$. For $k\in\N$ we define trees $T_k = (A^k, E_k)$ as follows. Firstly, $E_1 = \{\{i,4\} : i=1,2,3,5\}$. Inductively, assume that for some $k\in \N$ we have defined $T_k=(A^k, E_k)$ such that 
\begin{itemize}
\item If $w\in A^{k-1}$ and $i\in \{1,2,3,5\}$, then $wi$ and $w4$ are adjacent.
\item If If $w,u\in A^{k-1}$, $i,j\in A$ with $i\leq j$, and $wi$ is adjacent to $uj$, then either $(i,j) = (1,3)$, or $(i,j) = (1,4)$, or $(i,j) = (2,4)$, or $(i,j) = (2,5)$, or $(i,j) = (3,4)$, or $(i,j) = (4,5)$.
\end{itemize}
For the definition of $T_{k+1}$ fix $w,u \in A^{k-1}$.
\begin{enumerate}
\item If $i\in A$, then $wii_1$ is adjacent to $wi i_2$ with $i_1, i_2 \in A$ if and only if $i_1\in\{1,2,3,5\}$ and $i_2=4$.
\item If $w1$ is adjacent to $u3$, then $w11$ is adjacent to $u33$. 
\item If $w1$ is adjacent to $u4$, then $w13$ is adjacent to $u41$.
\item If $w2$ is adjacent to $u4$, then $w25$ is adjacent to $u42$.
\item If $w2$ is adjacent to $u5$, then $w22$ is adjacent to $u55$.
\item If $w3$ is adjacent to $u4$, then $w31$ is adjacent to $u43$. 
\item If $w4$ is adjacent to $u5$, then $w45$ is adjacent to $u52$.     
\end{enumerate}

Figure \ref{fig:T4} shows an illustration of $\mathbb{V}$ as well as the first two combinatorial trees $T_1$ and $T_2$\footnote{This picture of $\mathbb{V}$ was generated using the IFS Construction Kit (version April 11, 2019) created by Larry Riddle. This is available at \url{http://larryriddle.agnesscott.org/ifskit/download.htm}.}.
\begin{figure}[h]\label{fig:T4}
    \centering
    \begin{minipage}{0.3\textwidth}
        \centering
        \includegraphics[width=\textwidth]{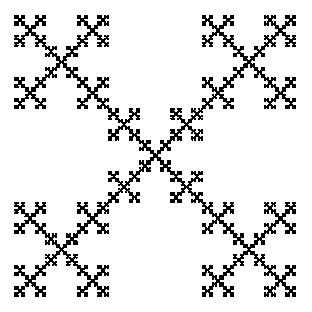} % first figure itself
        %\caption{first figure}
    \end{minipage}\hfill
    \begin{minipage}{0.5\textwidth}
        \centering
        \includegraphics[width=\textwidth]{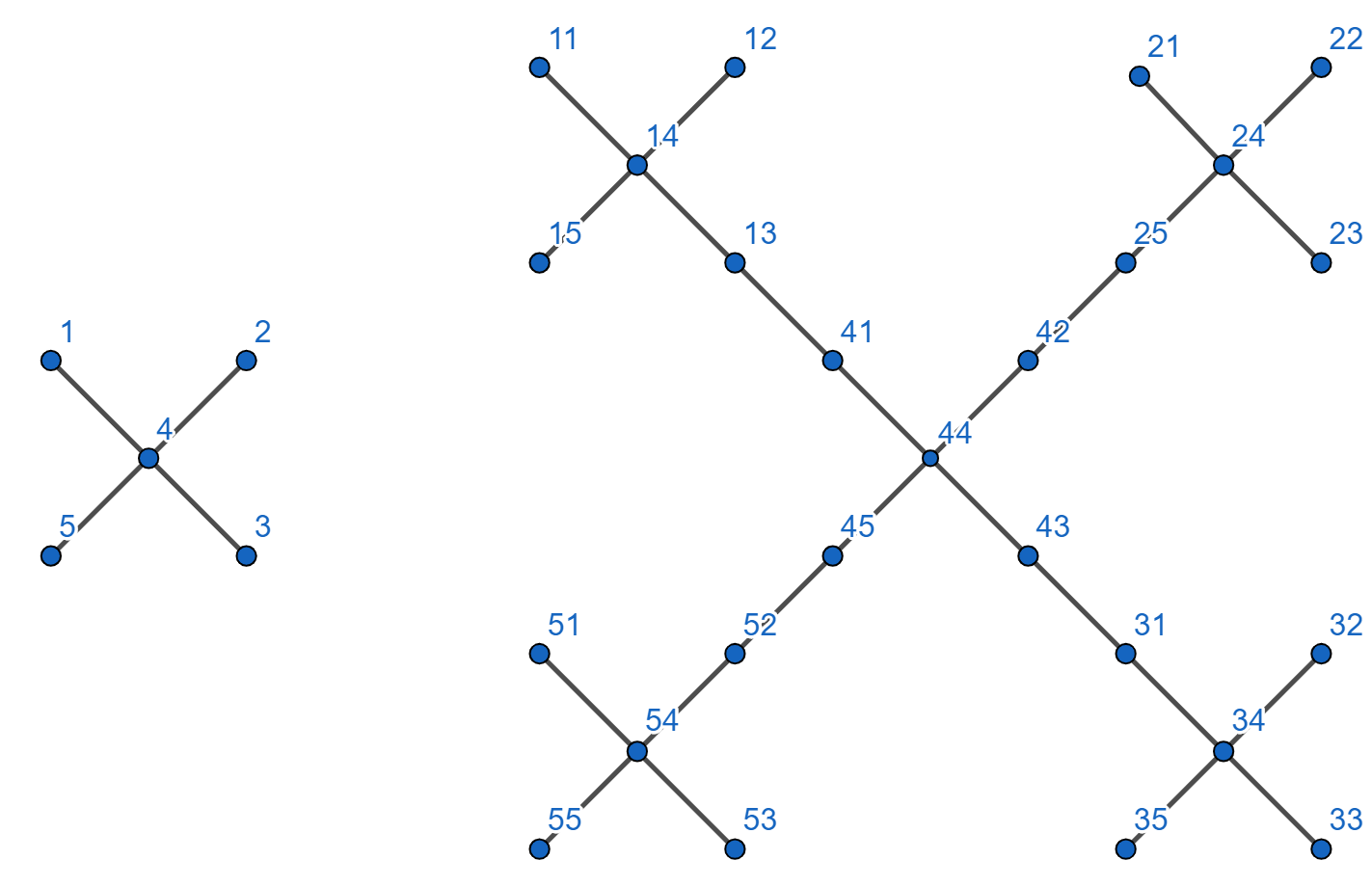} % second figure itself
       % \caption{second figure}
    \end{minipage}
\caption{On the left we have $\mathbb{V}$ while on the right we have the trees $T_1$ and $T_2$ of $\mathscr{C}$.}
\end{figure}

Define a diameter function $\D:A^* \to [0,1]$ by simply setting $\D(w) = 3^{-|w|}$. Clearly $\D \in \mathscr{D}(A,1/3,1/3)$. 

\begin{claim}\label{claim:T4}
The space $(\mathcal{A},d_{\mathscr{C},\D})$ is bi-Lipschitz equivalent to $\mathbb{V}$.
\end{claim}
\begin{proof}
The proof essentially follows that of Theorem \ref{thm:main}. For each $w=i_1\cdots i_n\in A^*$, let $X_w = \phi_{i_1}\circ\cdots\circ\phi_{i_n}(\mathbb{V})$. The collection of sets $\{X_w : w\in A^*\}$ satisfies the conclusions of Lemma \ref{lem:tiles}. Moreover, given $k\in\N$ and distinct $w,u \in A^k$, we have that $X_w\cap X_u \neq \emptyset$ if and only if $w$ is adjacent to $u$ in $T_k$. Define now $F:\mathcal{A} \to \mathbb{V}$ by
such that if $w=i_1i_2\cdots \in A^{\N}$, then 
\[ F([w]) := \bigcap_{n=1}^{\infty}X_{i_1\cdots i_n}, \qquad \text{for $w=i_1i_2\cdots \in A^{\N}$}.\]
The rest of the proof is as in \textsection\ref{sec:thmproof}, and we leave the details to the reader.
\end{proof}
\end{ex}

It follows immediately from Claim \ref{claim:T4} that $(\mathcal{A},d_{\mathscr{C},\D})$ is doubling, since $\mathbb{V}$ is. One could also see this by noting that conditions \ref{P1}, \ref{P2}, and \ref{P3} from Proposition \ref{prop:doubling} are clearly satisfied by this combinatorial data. To show that \ref{P4} also holds, we verify Lemma \ref{lem:suffdoubling}. Item (1) of Lemma \ref{lem:suffdoubling} is easy to check. For item (2), take any $w\in A^k$, and any $u,u' \in \partial_{\mathscr{C}}A^{k+1}_w$. The combinatorial arc that joins $u$ with $u'$ in $T_{k+1}$ contains three vertices, $\{u, w4, u'\}$, and so the total $\Delta$-length of this combinatorial arc
\[ \D(u) + \D(w4) + \D(u') = \tfrac13\D(w) + \tfrac13\D(w) + \tfrac13\D(w) = \D(w).\] 
Therefore, Lemma \ref{lem:suffdoubling} holds in this example, and therefore so does assumption \ref{P4} of Proposition \ref{prop:doubling}. Thus, all the conditions of Proposition \ref{prop:doubling} are satisfied and $(\mathcal{A},d_{\mathscr{C},\D})$ can be seen to be doubling by this proposition.

One may obtain new self-similar quasiconformal trees by keeping the same combinatorial data as the Vicsek tree but altering the diameter function $\Delta$. We describe two examples.

\begin{ex}
Keep the same combinatorial data $\mathscr{C}=\{T_k\}$ for $\mathbb{V}$ defined above, but now use the diameter function $\Delta_2(w) = 2^{-|w|}$ rather than $\Delta(w)=3^{-|w|}$ as before. Then the associated quotient space $(\mathcal{A}', d_{\mathscr{C},\Delta_2})$ is a ``snowflake'' of the previous example, in the following sense: It is bi-Lipschitz equivalent to the space $(\mathbb{V},|\cdot|^p)$, where $p = \frac{\log(2)}{\log(3)}$. The proof parallels that of Claim \ref{claim:T4}, the only difference being that the tiles $X_w$ of $\mathbb{V}$ under the snowflaked Euclidean metric $|\cdot|^p$ have diameters $(3^{-|w|})^p = 2^{-|w|} = \D_2(w)$.

\end{ex}

\begin{ex}
We again keep the combinatorial data $\mathscr{C}=\{T_k\}$ of the Vicsek tree, but modify the diameter function once more. Define a diameter function $\Delta_3$ by setting $\Delta_3(\varepsilon)=1$ and inductively setting
\[ \Delta_3(wi)= \begin{cases} 
      \frac{1}{2}\Delta_3(w) & \text{if } i\in\{2,4,5\} \\
      \frac{1}{4}\Delta_3(w) & \text{if } i\in \{1,3\}.
   \end{cases}
\]
In this case, the space $(\mathcal{A}'', d_{\mathscr{C},\Delta_3})$ is a quasiconformal tree which contains both geodesic segments (e.g., the path from $[1^\infty]$ to $[3^\infty]$) as well as non-geodesic ``snowflake'' segments (e.g., the path from $[2^\infty]$ to $[5^\infty]$.)

\end{ex}

\begin{rem}
A similar example to the Vicsek tree appears in \cite{Bonk-Tran, BM2} in the form of the \emph{continuum self-similar tree} (\emph{CSST}). The CSST is a quasiconformal tree, and hence by  Theorem \ref{thm:maincombthm} is bi-Lipschitz to one of our combinatorial models. However, it is not obvious to us that there is a simple concrete or dynamical way to form ``tiles'' in the CSST that satisfy all the assumptions in Lemma \ref{lem:tiles}, as we did for the Vicsek tree.
\end{rem}

\subsection{A non-doubling tree}

Below we give an example which illustrates the importance of condition \ref{P4} for the conclusions of Proposition \ref{prop:doubling}. Thus, we will construct combinatorial data $\mathscr{C}$ in which all graphs $G_k$ are trees, satisfying all the conditions of Proposition \ref{prop:doubling} except \ref{P4}, and for which the resulting metric tree is not doubling.

\begin{ex}\label{ex:nondoubling}
Let $\mathscr{C}$ be the combinatorial data of Example \ref{ex:vicsek}. For each $n\in \N$ let $w_n = 2\cdots 2 = 2^n \in A^n$, and let $u_{n,1},\dots,u_{n,N_n}$ denote those elements of $A^n$ such that $w_n1u_{n,i}$ has valence 1 in $T_{2n+1}$. 

Define $\D:A^* \to [0,1]$ with the following rules.
\begin{enumerate}
\item If $w$ is a word of the form $w_n1vu$, where $v \in A^n \setminus \{u_{n,1},\dots,u_{n,N_n}\}$ and $u\in A^*$, then let $\D(w_n1vu) = 4^{-|u|}\D(w_n1v) = 3^{-2n-1} 4^{-|u|}$.
\item For all other words $w\in A^*$, let $\Delta(w) = 3^{-|w|}$.
\end{enumerate}

We see that for each $n\in\N$, the following hold:
\begin{itemize}
\item If $v\in \{u_{n,1},\dots,u_{n,N_n}\}$, then $\mathcal{A}_{w_n1v}$ is bi-Lipschitz homeomorphic to $\mathbb{V}$ scaled by a factor of $3^{-2n-1}$.
\item If $v \in A^n \setminus \{u_{n,1},\dots,u_{n,N_n}\}$, then $\diam{\mathcal{A}_{w_n1v}} = 0$. (Indeed, two elements of $\mathcal{A}_{w_n1v}$ can be joined by a chain of $3^k$ steps at level $k$, each with the $\Delta$ value being $3^{-2n-1} 4^{-k}$ for arbitrary $k\in\mathbb{N}$, which forces the distance to be zero.) 
\end{itemize}
Therefore, for each $n\in\N$, the point $[w_n13^{\infty}]\in \mathcal{A}$ has at least $N_n$ branches, each of diameter at least $3^{-2n-1}$. Since $N_n \to \infty$ as $n\to \infty$, it follows that $(\mathcal{A},d_{\mathscr{C},\D})$ is not doubling.

Note also that $A$, $\mathscr{C}$, and $\D$ satisfy properties \ref{P1}, \ref{P2}, \ref{P3}, but not \ref{P4}. Indeed, the fact that $\diam{\mathcal{A}_{w_n1v}} = 0$ for certain words, as in the second bullet above, already violates \ref{P4}.
\end{ex}

\section{Combinatorial descriptions of more general spaces with ``good tilings''}\label{sec:gencombdata}
In this section, we axiomatize a notion of a ``good tiling'' of a compact space, and show that every compact space with such a tiling (not necessarily a tree) can be built from our combinatorial data.

Let $X$ be a compact space for which there is a finite alphabet $A$, constants $r\in (0,1)$, $C>1$, and a collection of nonempty closed, connected subsets $\{X_w : w\in A^*\}$ with the following properties.
\begin{enumerate}
\item $X_{\varepsilon} = X$.
\item For all $w\in A^*$ and all $i\in A$, $X_{wi} \subset X_w$. Moreover, $\bigcup_{i\in A}X_{wi} = X_w$.
\item For all $w\in A^*$, $C^{-1} r^{|w|} \leq \diam{X_w} \leq C r^{|w|}$.
\item If for $k\in\N$ and $w,u \in A^k$ we have $X_{w}\cap X_{u} =\emptyset$, then $d(X_{w},X_{u}) \geq C^{-1}r^{k}$.
\end{enumerate}
Tilings of metric spaces with very similar properties have certainly been considered by other authors, e.g., \cite{BM, Kigami2}. The goal here is simply to write down some simple conditions that can be interpreted in our framework.

For each $k\in\N$ define a graph $G_k = (A^k,E_k)$ with the rule that for words $w,u\in A^k$, $w$ is adjacent to $u$ if and only if $X_w \cap X_u \neq \emptyset$. It is easy to see that the collection $\mathscr{C} = (A,(G_k)_{k\in\N})$ is combinatorial data in the sense of Definition \ref{def:combdata}. Define also $\D:A^* \to [0,1]$ with $\D(w) = r^{|w|}$. Clearly, $\D \in \mathscr{D}(A,r,r)$.

\begin{prop}\label{prop:goodtiles}
The space $(X,d)$ is bi-Lipschitz homeomorphic to $(\mathcal{A},d_{\mathscr{C},\D})$.
\end{prop}

Before the proof, we re-emphasize two points about Proposition \ref{prop:goodtiles}. For one, even if $X$ is a metric tree, Proposition \ref{prop:goodtiles} does not force the combinatorial data $\mathscr{C}$ to consist of combinatorial trees. The second point is that in general, it is not obvious to us which spaces admit good tilings in the sense of this section. Thus, Proposition \ref{prop:goodtiles} is not in itself a generalization of Theorem \ref{thm:maincombthm}, and proceeds along different lines. The tiles we constructed for quasiconformal trees in Lemma \ref{lem:tiles} do not satisfy the conditions of this section, as they may in principle fail conditions (3) or (4) of this section.

However, Proposition \ref{prop:goodtiles} does yield descriptions of some natural examples, as we show following the proof.

\begin{proof}[Proof of Proposition \ref{prop:goodtiles}]
Since $r<1$, property (3) in conjuction with the compactness of sets $\{X_w:w\in A^*\}$ gives that for any $w\in A^{\N}$, the set $\bigcap_{n\in\N}X_{w(n)}$ contains exactly one point which we denote by $x_w$.

Let $w,u \in A^{\N}$ such that $x_w \neq x_u$. Then, there exists $n\in\N$ such that $X_{w(n)}\cap X_{u(n)} = \emptyset$. Let $A^{\N}_{w_1},\dots, A^{\N}_{w_n}$ be a chain joining $w$ with $u$. Then, $x_w\in X_{w_1}$, $x_u \in X_{w_n}$ and $X_{w_i}\cap X_{w_{i+1}} \neq \emptyset$ for all $i\in\{1,\dots,n-1\}$. By the triangle inequality,
\[ d(x_w,x_u) \leq \sum_{i=1}^n\diam{X_{w_i}} \leq C\sum_{i=1}^n\D(w_i).\]
Taking the infimum over all such chains, we obtain that  $d(x_w,x_u) \leq Cd_{\mathscr{C},\D}([w],[u])$. Therefore, if $[w]=[u]$, then $x_w = x_u$. 

We can now define $F: \mathcal{A} \to X$ with $F([w]) = x_w$. By the preceding paragraph, $F$ is well defined and $C$-Lipschitz.

To see why $F$ is bi-Lipschitz, fix $w,u\in A^{\N}$. 

If $d(x_w,x_u)=0$, that is $x_w = x_u$, then for all $n\in\N$ $X_{w(n)}\cap X_{u(n)} \neq \emptyset$. Therefore, $w(n)$ is adjacent or equal to $u(n)$ for all $n$ and it follows that $d_{\mathscr{C},\D}([w],[u]) =0$.

If $x_w\neq x_u$, then there exists $n\in\N$ such that $X_{w(n)}\cap X_{u(n)} \neq \emptyset$ and $X_{w(n+1)}\cap X_{u(n+1)} = \emptyset$ . It follows that $w(n)$ is adjacent to $u(n)$ in $G_n$ and $\{A^{\N}_{w(n)}, A^{\N}_{u(n)}\}$ is a chain joining $w$ with $u$. Therefore,
\begin{align*} d(F([w]),F([u])) \geq \dist(X_{w(n+1)},X_{u(n+1)}) \geq C^{-1}r^{n+1} = (2C/r)^{-1} 2r^n &= (2C/r)^{-1}(\D(w(n)) + \D(u(n)))\\ 
&\geq (2C/r)^{-1} d_{\mathscr{C},\D}([w],[u]). \qedhere
\end{align*}
\end{proof}

\begin{ex}\label{ex:gasket}
Proposition \ref{prop:goodtiles} applies to many metric spaces which are attractors for certain iterated function systems, like the square, the Sierpi\'nski gasket, and the Sierpi\'nski carpet. See the figures below for possible graphs  $G_1$ and $G_2$ for the gasket, square, and carpet.

\begin{figure}[h]
    \centering
    \begin{minipage}{0.4\textwidth}
        \centering
        \includegraphics[width=\textwidth]{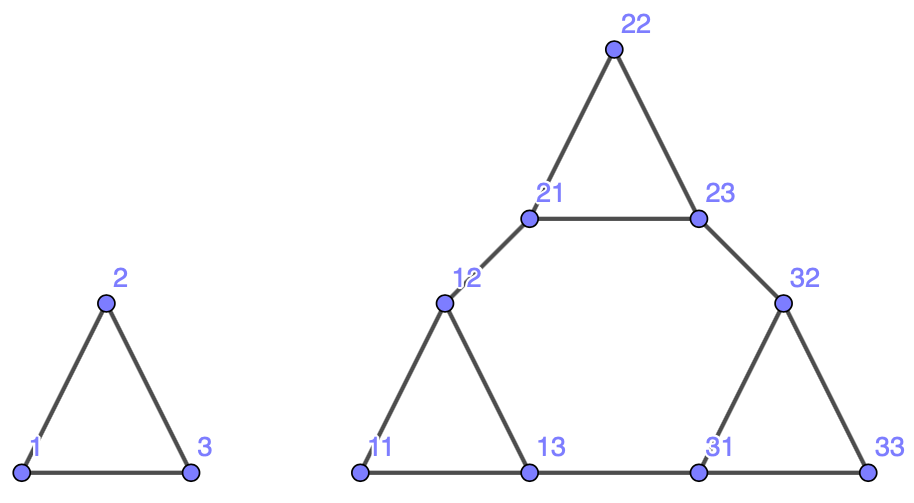} % first figure itself
        %\caption{first figure}
    \end{minipage}\hfill
    \begin{minipage}{0.4\textwidth}
        \centering
        \includegraphics[width=\textwidth]{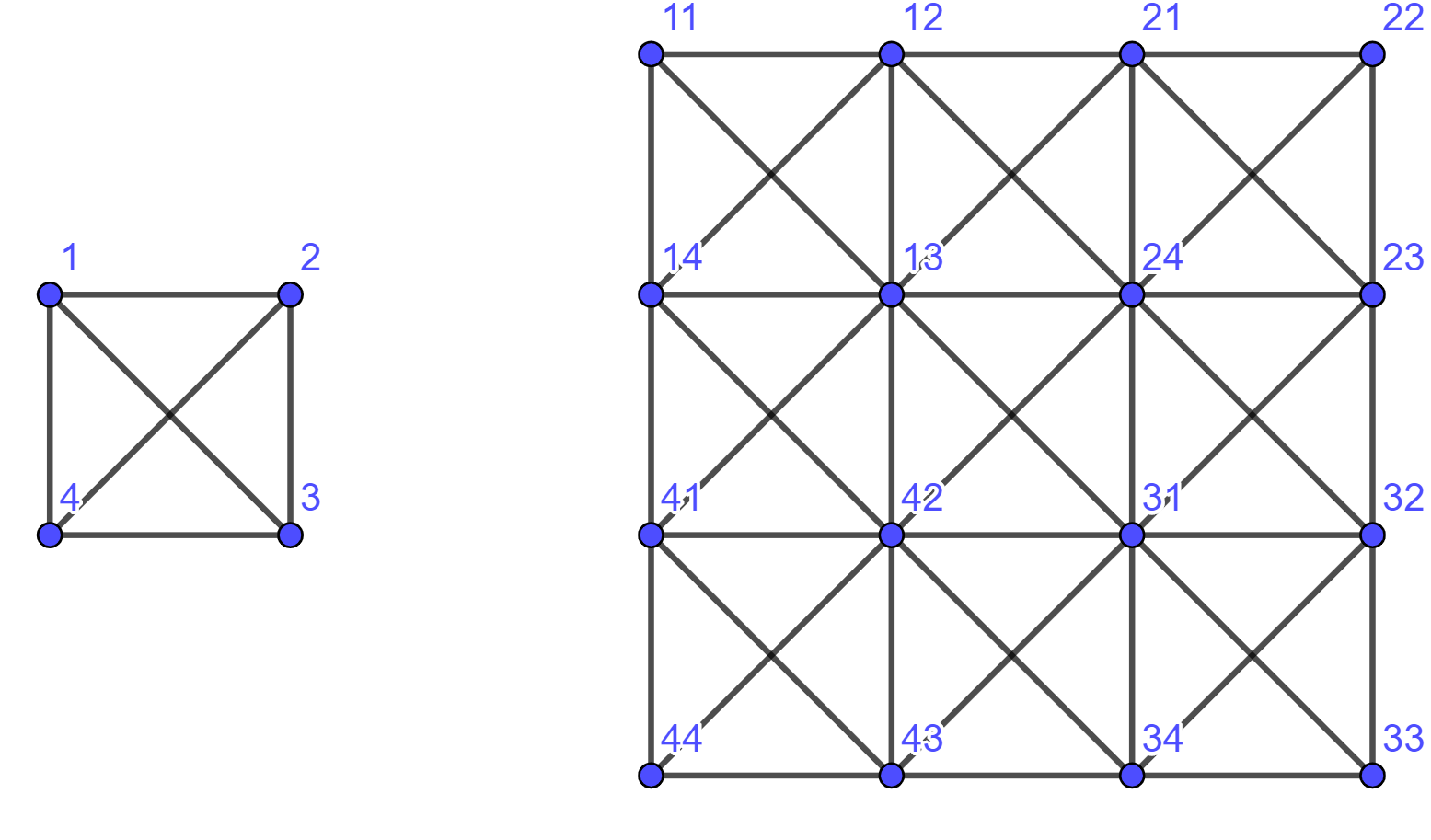} % second figure itself
       % \caption{second figure}
    \end{minipage}
\caption{On the left we have the graphs $G_1,G_2$ for the Sierpi\'nski gasket while on the right we have the graphs $G_1,G_2$ for the square.}
\end{figure}

\begin{figure}[h]
\includegraphics[width=0.8\textwidth]{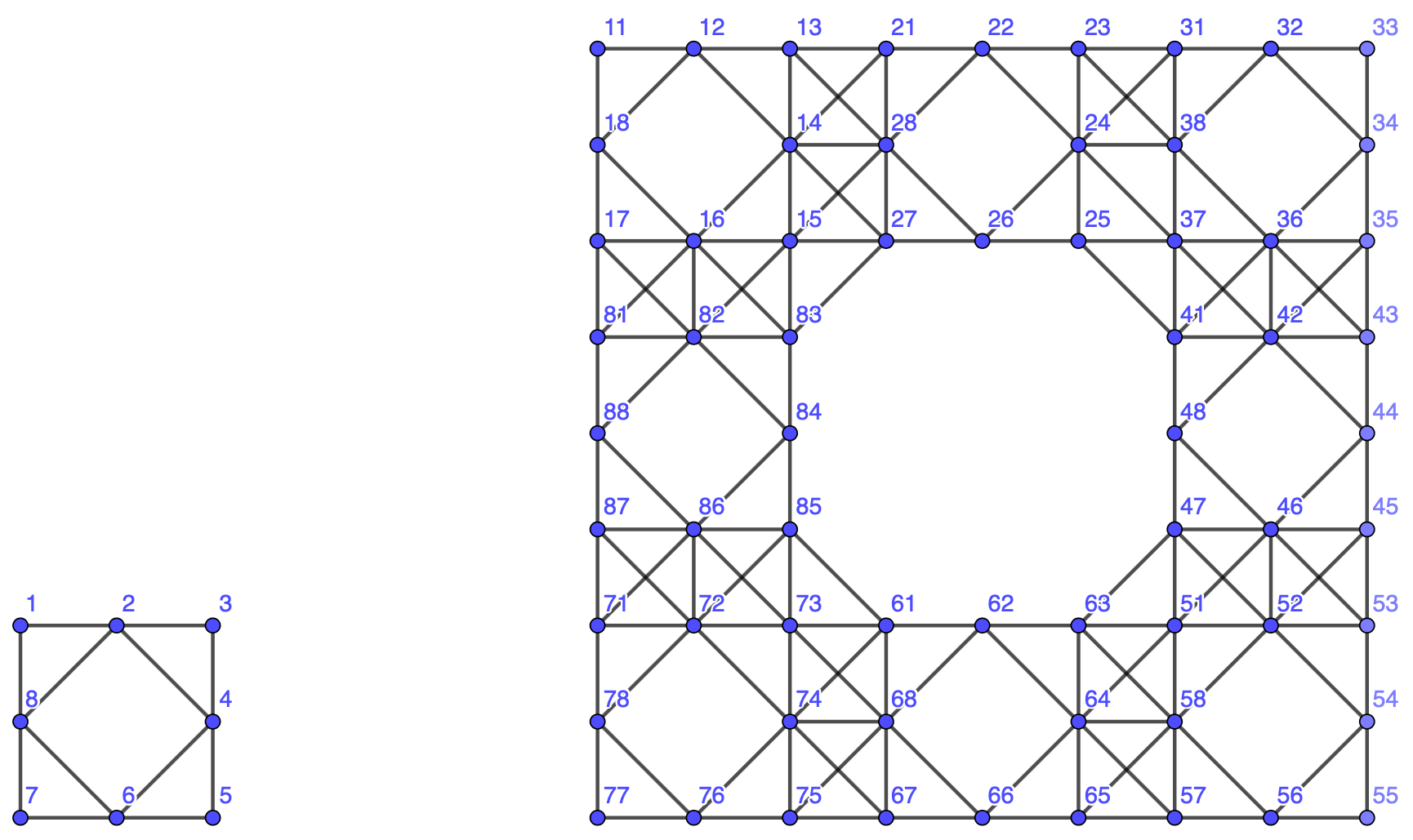}
\caption{Possible graphs $G_1$ and $G_2$ for the standard Sierpi\'nski carpet.}
\end{figure}
\end{ex}

\section{Bi-Lipschitz embedabbility of quasiconformal trees}\label{sec:embed}
This section is devoted to the proof of the following quantitative version of Theorem \ref{thm:mainembed}.

\begin{thm}\label{thm:mainembed2}
Let $X$ be a $C$-doubling, $c$-bounded turning tree. Assume that $\LL(X)$ admits an $L$-bi-Lipschitz embedding into some $\R^M$. Then $X$ admits an $L'$-bi-lipschitz embedding into some $\R^N$. Here $N$ and $L'$ depend only on $C$, $c$, $M$ and $L'$.
\end{thm}

The proof of Theorem \ref{thm:mainembed2} consists of two steps. In Section \ref{sec:quasiarcs} we prove the special case of embedabbility of quasi-arcs, i.e., quasiconformal trees in which the set of leaves consists of exactly two points. This is done in Proposition \ref{prop:qcircles} below, which is a stronger version of Proposition \ref{prop:introqcircles} from the introduction.

Then, in Section \ref{sec:Seo-method}, we employ a bi-Lipschitz welding theorem of Lang and Plaut \cite{LP} and a characterization of metric spaces admitting bi-Lipschitz embedding into Euclidean spaces by Seo \cite{Seo} to complete the proof of  Theorem \ref{thm:mainembed2}.

\subsection{Bi-Lipschitz embeddability of quasi-arcs}\label{sec:quasiarcs}

The main result of this subsection is the following special case of Theorem \ref{thm:mainembed} where the leaf set $\mathcal{L}(X)$ consists of only two points. In particular, this gives a detailed, sharp version of Proposition \ref{prop:introqcircles}.

We first introduce a piece of terminology: A metric space $X$ is \textit{$(C,s)$-homogeneous}, for some $C,s\geq 0$, if every subset of diameter $d$ can be covered by at most $C\epsilon^{-s}$ sets of diameter at most $\epsilon d$. In particular, every doubling metric space is $(C,s)$-homogeneous for some $C$ and $s$ depending on the doubling constant \cite[Section 10.13]{Heinonen}.

\begin{prop}\label{prop:qcircles}
Given $s\geq 1$, $C>0$ and $c\geq 1$, there exists $L=L(c,C,s)>1$ with the following property. If $\G = ([0,1],d)$ is $c$-bounded turning and $(C,s)$-homogeneous, then it is $L$-bi-Lipschitz embeddable in $\R^{\lfloor s \rfloor + 1}$.
\end{prop}

Proposition \ref{prop:qcircles} generalizes Theorem C in \cite{HM}, where it was assumed that $s<2$. We remark that the dimension $\lfloor s \rfloor + 1$ in Proposition \ref{prop:qcircles} is sharp when $s>1$, in the sense that there exists a $1$-bounded turning, $(C,s)$-homogeneous metric $d$ on $[0,1]$ (namely the snowflaked Euclidean metric $|\cdot|^{1/s}$) such that $([0,1],d)$ can not be bi-Lipschitz embedded in $\R^{\lfloor s \rfloor}$.

For the proof of Proposition \ref{prop:qcircles}, we may assume that $\diam{\G} = 1$. The proof uses a construction of Herron and Meyer \cite{HM} and a bi-Lipschitz embedding method of Romney-Vellis \cite{RV} (see also \cite{BH} and \cite{Wu}). 

Let $M \in \{2,3,\dots\}$, $A=\{1,\dots,M\}$, and $\mathscr{C}_M=(A,(G_k)_{k\in\N})$ be as in Example \ref{ex:quasiarc}.

\begin{lem}[{\cite[Lemma 3.1]{HM}}]\label{lem:HM2}
If $d\in (M^{-1},1)$ and $\D \in \mathscr{D}(A,M^{-1},\d)$, then the space $(\mathcal{A},d_{\mathscr{C}_M,\D})$ is $s$-homogeneous with $s=\log(M)/\log(1/\d)$.
\end{lem}

The following result can be obtained following the arguments of Theorem B of \cite{HM} essentially verbatim; we provide a brief reference to the necessary arguments.

\begin{prop}\label{lem:HM}
Let $s\geq 1$, $c\geq 1$, and $\G$ a $c$-bounded turning and $s$-homogeneous metric arc with $\diam{G}=1$. Then for any $M\in\{2,3,\dots\}$ and any $\d \in (M^{-1/s},1)$, there exists $\D \in \mathscr{D}(A,M^{-1},\d)$ and an $L$-bi-Lipschitz homeomorphism $f : \G \to (\mathcal{A},d_{\mathscr{C}_M,\D})$. The constant $L$ depends only on $c$, $s$, and $M$.
\end{prop}
\begin{proof}
Following exactly the procedure on \cite[p. 622]{HM}, we divide $\G$ into $M$ sub-arcs of equal diameter, then iterate this procedure on each sub-arc. Letting $\mathscr{C}_M = (A=\{1, \dots, M\}, G_k)$ as above, this yields an assignment to each element $w\in A^*$ of an arc $\gamma_w\subseteq\G$, with nesting and adjacency properties reflecting that of $\mathscr{C}_M$ and $\sup_{w\in A^k} \diam(\gamma_w) \rightarrow 0$ as $k\rightarrow\infty$.

The argument on \cite[p. 622-623]{HM} provides a diameter function $\D\in\mathscr{D}(A,M^{-1},\d)$ such that
$$ \D(w) \approx \diam(\gamma_w),$$
with implied constant depending only on $c$, $s$, and  $M$.

Defining $F\colon \cA \rightarrow \G$ by $F([w]) = \cap_{k=1}^\infty \gamma_{w(k)}$, we see exactly as in Lemma \ref{lem:welldef} and Proposition \ref{prop:bilipschitz} of the present paper that $F$ is well-defined, surjective, and bi-Lipschitz. Taking $f=F^{-1}$ completes the proof.
\end{proof}

We now fix parameters $M$ and $\delta$ that will enable us to use a construction from \cite{RV}. Given $s\geq 1$, let
\begin{itemize}
\item $n$ be the minimal integer satisfying $n>(\lfloor s \rfloor + 1 -s)^{-1}$,
\item $p=\lfloor s \rfloor - 1 + \frac{n-1}{n} = \lfloor s \rfloor - \frac{1}{n}>0$,
\item $M_0 = 9^{n(\lfloor s \rfloor + 1)}$,
\item $M = M_0^{1+p}$, and
\item $\delta = M_0^{-1}$.
\end{itemize}
The above parameters all depend on $s$, but we suppress this in the notation. Observe that $\delta > M^{-1/s} \geq M^{-1}$ in all cases, and in fact $\delta$ is an integer multiple of $M^{-1}$.  Only $\delta$ and $M$ will play a direct role below.

Given Proposition \ref{lem:HM}, the proof of Proposition \ref{prop:qcircles} now reduces to the following lemma.

\begin{lem}\label{lem:snowembed}
Let $s\geq 1$ and choose $M$ and $\delta$ as above. Let $\D\in \mathscr{D}(A,M^{-1},\delta)$. Then there is a bi-Lipschitz embedding of $(\mathcal{A},d_{\mathscr{C}_M,\D})$ into $\R^{\lfloor s \rfloor + 1}$ with bi-Lipschitz constant depending only on $M$, $\d$ and $s$, and thus only on $s$.
\end{lem}

The construction of the embedding follows ideas and notation from \cite{RV}. We fix parameters $M$ and $\delta$ as in the statement of Lemma \ref{lem:snowembed} and write $\mathscr{C}=\mathscr{C}_M$. We also fix $\D\in \mathscr{D}(A,M^{-1},\delta)$ for the remainder of this subsection.

Let 
\begin{align*}
I &= [0,1]\times\{0\}^{\lfloor s \rfloor}\\
L&=\left( (\{0\}\times[0,1/2])\cup([0,1/2]\times\{1/2\})\right)\times\{0\}^{\lfloor s \rfloor-1}
\end{align*}
with the convention that $E\times\{0\}^0 = E$. An $I$-segment (resp. $L$-segment) is the image of $I$ (resp. $L$) under a similarity mapping of $\R^{\lfloor s \rfloor + 1}$, and is parallel to the coordinate axes. 

Given an $I$- or $L$- segment $\tau$ with length $\ell$ and endpoints $x^*,y^*$, we define the \emph{cubic thickening} $Q(\tau)$ of $\tau$ to be the union of all closed cubes parallel to coordinate axes, of side length $(1-2\delta)\ell$ and centered on points $z\in\tau$ such that 
\[ \min\{|z-x^*|,|z-y^*|\} \geq \ell(1-2\delta)/2.\] 
Define also $\mathcal{C}(\tau)$ to be the closed cube which is parallel to coordinate axes, has side length $\ell$, and is centered on the midpoint of $\tau$. The intersection $Q(\tau) \cap \partial\mathcal{C}(\tau)$ has exactly two components which we call the \emph{entrances} of $Q(\tau)$. 

For each $\tau\in\{I,L\}$ we  define two polygonal arcs $\mathcal{J}(\tau)$ and $\mathcal{J}_0(\tau)$ in the following lemma.

\begin{lem}\label{lem:RV}
Given $\tau \in \{I,L\}$ there exist two polygonal arcs $\mathcal{J}(\tau)$ and $\mathcal{J}_0(\tau)$, each contained in $Q(\tau)$, whose endpoints are the same as those of $\tau$ and that satisfy the following properties.
\begin{enumerate}
\item[(J1)] The arcs $\mathcal{J}(\tau), \mathcal{J}_0(\tau)$ consist of $M$-many $I$-segments and $L$-segments $\sigma_{i}$, $i\in \{1,\dots,M\}$, labeled according to their order in $\mathcal{J}(\tau)$ with $\sigma_1$ containing the origin. Each $\sigma_i$ in $\mathcal{J}(\tau)$ has length $\delta$ and each $\sigma_i$ in $\mathcal{J}_0(\tau)$ has length $M^{-1}$.
\item[(J2)] The segments $\sigma_{1}$ and $\sigma_{M}$ are $I$-segments.
\item[(J3)] For all $i\in \{1,\dots,M-1\}$, $Q(\sigma_i)\cap Q(\sigma_{i+1})$ is an entrance of $Q(\sigma_i)$ and an entrance of $Q(\sigma_{i+1})$. If $i,j \in \{1,\dots,M\}$, with $|i-j|>1$, then $Q(\sigma_i)\cap Q(\sigma_j) = \emptyset$.
\item[(J4)] If $E_1,E_2$ are the entrances of $Q(\tau)$, then an entrance of $Q(\sigma_1)$ is contained in $E_1$ and an entrance of $Q(\sigma_M)$ is contained in $E_2$. Moreover, for any $i\in\{2,\dots,M-1\}$, $Q(\sigma_i)\cap \partial Q(\tau) = \emptyset$.
\end{enumerate}
\end{lem}

\begin{proof}
The constructions of $\mathcal{J}_0(I)$ and $\mathcal{J}_0(L)$ are quite simple. Write $I = \bigcup_{m=1}^{M}\sigma_m$ with 
\[ \sigma_m =\left[\frac{m-1}{M},\frac{m}{M}\right]\times\{0\}^{\lfloor s \rfloor} \subset \R^{\lfloor s \rfloor + 1}\] 
and set $\mathcal{J}_{0}(I) = \bigcup_{m=1}^M \sigma_m = I$. Similarly write $L = \bigcup_{m=1}^{M}\sigma_m$ where $\sigma_{m}$ is an $L$-segment if $m=\frac{M+1}{2}$ and an $I$-segment otherwise and each $\sigma_m$ has length $1/M$. Set $\mathcal{J}_{0}(L) = \bigcup_{m=1}^M \sigma_m$.

The constructions of $\mathcal{J}(I)$ and $\mathcal{J}(L)$ are more complicated and can be found in \cite[\textsection6.1, \textsection6.2]{RV} (where they are denoted as $J_I(N,n)$ and $J_L(N,n)$, respectively). 
%We will not describe the constructions here but we will give a rough sketch.
Without describing the construction, we briefly explain how our parameters match with those of \cite{RV}. The parameter $N$ appearing on \cite[p. 1181]{RV} matches our $\lfloor s \rfloor -1$. Our parameters $p$ and $n$ match the ones given there. Our parameter $M_0$ corresponds to $M$ on \cite[p. 1182]{RV}, and our parameter $M$ corresponds to $M^{1+p}$ on \cite[p. 1182]{RV}. Making allowances for the changes in notation, our desired properties of $\mathcal{J}(I)$ and $\mathcal{J}(L)$ are listed in Section 3.3 of \cite{RV} as properties (1)-(3).
%The parameter $N$ appearing on \cite[p. 4]{RV} matches our $\lfloor s \rfloor -1$. Our parameters $p$ and $n$ match the ones given there. Our parameter $M_0$ corresponds to $M$ on \cite[p. 4]{RV}, and our parameter $M$ corresponds to $M^{1+p}$ on \cite[p. 5]{RV}.
\end{proof}

We record a few more simple consequences of properties (J1)--(J4).
\begin{lem}\label{lem:pipelines}
Consider $\tau \in \{I,L\}$, $J \in \{\mathcal{J}(\tau) ,\mathcal{J}_0(\tau)\}$. Recall that $J$ is a union of sets $\{\sigma_i\}_{i=1}^M$, each of which is an $I$-segment or $L$-segment. Then:
\begin{enumerate}
\item For each $i\in \{1, \dots, M\}$, $Q(\sigma_i) \subset Q(\tau)$. 
\item For each $i\in\{2,\dots,M-1\}$,
\[ \dist(Q(\sigma_i),\partial Q(\tau)) \geq M^{-2} 
.\]
\item If $i,j \in \{1,\dots,M\}$ with $|i-j|>1$, then,
\[ \dist(Q(\sigma_i), Q(\sigma_j)) \geq M^{-2} 
.\]
\item Let $E$ be the entrance of $Q(\tau)$ that contains an endpoint of $\sigma_1$ (resp. endpoint of $\sigma_M$) and let $P$ be the $\lfloor s \rfloor$-dimensional plane that contains $E$. Then for all $i\in\{2,\dots,M\}$ (resp. $i\in\{1,\dots,M-1\}$)
\[ \dist(Q(\sigma_i),P) \geq M^{-2}.\] 
\end{enumerate}
\end{lem}
\begin{proof}
All four statements are obvious in the case $J=\mathcal{J}_0(\tau)$, so we now assume that $J=\mathcal{J}(\tau)$. Statement (1) is an immediate consequence of the fact that $J\subseteq Q(\tau)$ and property (J4) of Lemma \ref{lem:RV}.

For the remaining three properties, it is useful to first observe that, since $\delta$ is an integer multiple of $M^{-1}$, the sets $Q(\tau)$ and $Q(\sigma_i)$ are each unions of axis-parallel cubes whose vertices lie on the $M^{-2}$-scale grid $M^{-2}\mathbb{Z}^{\lfloor s \rfloor + 1}.$ 

Statements (2) and (4) follow immediately from this observation and (J4). Statement (3) follows immediately from this observation and (J3).  
\end{proof}

We now use Lemma \ref{lem:RV} to construct arcs in $\R^{\lfloor s \rfloor +1}$ that mimic the metric properties of the combinatorial construction $\mathscr{C}, \D$ fixed below the statement of Lemma \ref{lem:snowembed}.

\begin{lem}\label{lem:defnoftau}
For each $w\in A^*$ there exists an $I$- or $L$-segment $\tau_w$ with the following properties.
\begin{enumerate}
\item If $w,u \in A^k$ are adjacent, then $\tau_w$ and $\tau_u$ intersect at an endpoint while $Q(\tau_w)\cap Q(\tau_u)$ is contained in an entrance of $Q(\tau_w)$ and an entrance of $Q(\tau_u)$. If $w,u \in A^k$ are distinct but not adjacent, then $Q(\tau_w)\cap Q(\tau_u)$ and $\tau_w\cap\tau_u$ are empty.
\item For any $w\in A^*$, there exists $\tau\in \{I,L\}$ such that $\tau_w$ and $Q(\tau_w)$ are scaled copies of $\tau$ and $Q(\tau)$, respectively, by a factor of $\D(w)$.
\end{enumerate}
\end{lem}

\begin{proof}
The construction is done in an inductive manner.

Let $\tau_{\varepsilon}:=I \subset \R^{\lfloor s \rfloor + 1}$. Property (1) of the lemma is vacuous in this base case, while property (2) is immediate. 

Assume now that for some integer $k\geq 0$ we have defined $I$- and $L$-segments $\tau_w$ (for all $j\leq k$ and $w\in A^j$) satisfying the properties of the lemma. Fix $w\in A^{k}$, and let $u$ be the preceding vertex of $A^k$ in lexicographic order, assuming for the moment that such a vertex exists. Let $E$ be the entrance of $Q(\tau_w)$ that intersects an entrance of $Q(\tau_u)$. Suppose that $\tau_w$ is a rescaled copy of $\tau \in \{I,L\}$. Let $\phi_w : \R^{\lfloor s\rfloor + 1} \to \R^{\lfloor s\rfloor + 1}$ be a similarity map such that $Q(\tau)$ is mapped onto $Q(\tau_w)$, the entrance of $Q(\tau)$ that contains the origin is mapped onto the entrance of $Q(\tau_w)$ that contains $Q(\tau_w)\cap Q(\tau_u)$, and the other entrance of $Q(\tau)$ is mapped to the other entrance of $Q(\tau_w)$.

If there is no $u\in A^k$ preceding $w$ in lexicographic order, then $w=1^k$ for some $k\geq 0$. In that case, if $k=0$ we set $\phi_w$ to be the identity, and if $k\geq 1$ we set $u=1^{k-1}2$ and do the analogous construction of $\phi_w$ to arrange that the entrance of $Q(\tau)$ that does not contain the origin is mapped onto the entrance of $Q(\tau_w)$ that contains $Q(\tau_w)\cap Q(\tau_u)$.

We now define $\tau_{wi}$ for each $i\in A$:
\begin{itemize}
\item If $\D(w1)=M^{-1}\D(w)$, then for each $i\in A$ set $\tau_{wi} = \phi_w(\sigma_i)$ where $\sigma_i \subset\mathcal{J}_0(\tau)$.
\item If $\D(w1) = \delta\D(w)$, then for each $i\in A$ set $\tau_{wi} = \phi_w(\sigma_i)$ where $\sigma_i \subset\mathcal{J}(\tau)$.
\end{itemize}

This completes the definition of the arcs $\tau_{w}$ for all $w\in A^{k+1}$. We now prove that the family $\{\tau_w : w\in A^{k+1}\}$ satisfies properties (1) and (2) of the lemma.

For property (2) of the lemma, by design, and the inductive hypothesis (2), for all $i\in A$
\[ \diam{\tau_{wi}} = \diam{\phi_w(\sigma_i)} = \frac{\diam{Q(\tau_w)}}{\diam{Q(\tau)}}\diam{\sigma_i} = \D(w)\diam{\sigma_i} = \D(wi)\diam{\tau'}\]
for some $\tau' \in \{I,L\}$. Therefore, $\diam{Q(\tau_{wi})} = \D(wi)\diam{Q(\tau')}$ for some $\tau' \in \{I,L\}$ and property (2) holds for $k+1$.

We now turn to the proof of (1). Let $w \in A^{k}$ and $i\in A$. Let also $u\in A^k$ and $j\in A$. We consider two cases.

\emph{Case 1.} Assume that $w=u$ and $i\neq j$. If $wi$ is adjacent to $wj$, then by design of paths $\mathcal{J}(\tau)$ and $\mathcal{J}_0(\tau)$, we have that $\tau_{wi}$ and $\tau_{wj}$ share an endpoint and by (J3) $Q(\tau_{wi})\cap Q(\tau_{wj})$ is a common entrance of $Q(\tau_{wi})$ and $Q(\tau_{wj})$. If $wi$ is not adjacent to $wj$, then again by (J3) $Q(\tau_{wi})\cap Q(\tau_{wj}) = \emptyset$ which also implies that $\tau_{wi}\cap \tau_{wj} = \emptyset$.

\emph{Case 2.} Assume that $u\neq w$. The proof splits in two subcases. 

\emph{Case 2.1}. Assume that $i\not\in\{1,M\}$. Then $wi$ is not adjacent to $uj$ and by (J4) $Q(\tau_{wi})$ is contained in the interior of $Q(\tau_{w})$ which is disjoint from $Q_{u}$ by the inductive hypothesis. Therefore, $Q(\tau_{wi})\cap Q(\tau_{uj})$ and $\tau_{wi}\cap\tau_{uj}$ are both empty. 

\emph{Case 2.2} Assume that $i\in\{1,M\}$. Without loss of generality, we assume that $i=1$; the case $i=M$ is similar. By design $Q(\tau_{wi})$ intersects one entrance of $Q(\tau_w)$ but not the other. Therefore, if $u$ is not adjacent to $w$ or if it is adjacent to $w$ but is preceded by $w$, then the inductive hypothesis implies that $Q(\tau_{w1})\cap Q(\tau_{uj})$ and $\tau_{w1}\cap\tau_{uj}$ are both empty. Assume now that $u$ is adjacent to $w$ and precedes $w$. Then, the only $j\in A$ for which $Q(\tau_{uj})$ intersects the entrance of $Q(\tau_u)$ which contains $Q(\tau_{w})\cap Q(\tau_{u})$ is $j=M$. In this case, $\tau_{uM}\cap \tau_{w1}$ is the common endpoint of $\tau_w$ and $\tau_u$. Therefore, $Q(\tau_{wi})\cap Q(\tau_{uj})$ is nonempty and is contained in an entrance of $Q(\tau_{wi})$ and an entrance of $Q(\tau_{uj})$.
\end{proof}

Lemma \ref{lem:defnoftau}(2) implies that for all $w\in A^*$,
\begin{equation}\label{eq:est1}
2^{-1/2}\D(w) \leq \diam{\tau_w} \leq \D(w).
\end{equation}

\begin{lem}\label{lem:pipes2}
Let $w,u \in A^k$ be adjacent words, with $w$ preceding $u$ in lexicographic order. If $i\in A\setminus \{M\}$ or if $j\in A\setminus \{1\}$, then
\[ \dist(Q(\tau_{wi}),Q(\tau_{uj})) \gtrsim_s \max\{\D(w),\D(u)\}.\]
\end{lem}

\begin{proof}
Set $E = Q(\tau_{w})\cap Q(\tau_u)$. By Lemma \ref{lem:defnoftau}, $E$ is contained in an entrance of $Q(\tau_w)$ and in an entrance of $Q(\tau_u)$. Let $P$ be the $\lfloor s\rfloor$-dimensional plane in $\R^{\lfloor s\rfloor+1}$ that contains $E$. Then, $P$ separates the interior of $Q(\tau_{wi})$ from the interior $Q(\tau_{uj})$. By Lemma \ref{lem:pipelines},
\[ \dist(Q(\tau_{wi}),Q(\tau_{uj})) \geq \max\{\dist(Q(\tau_{wi}),P), \dist(Q(\tau_{uj}),P)\} \gtrsim_s \max\{\D(w),\D(u)\}. \qedhere\]
\end{proof} 

For each $w\in A^*$ and $k\geq |w|$, set 
\[ \mathcal{Q}^{(k)}_w := \bigcup_{u \in A^{k}_w}Q(\tau_u),\quad \mathcal{Q}^{(k)} := \bigcup_{u\in A^k}Q(\tau_u),\quad \mathcal{Q}_w :=  \bigcap_{n\geq |w|}\mathcal{Q}^{(n)}_w.\]

By \eqref{eq:est1}, if $w \in A^{\N}$, then $\lim_{n\to\infty}\diam{Q(\tau_{w(n)})} \leq \lim_{n\to\infty}(\lfloor s \rfloor + 1)^{1/2}\delta^{n} = 0$. For each $w\in A^{\N}$ denote by $x_w$ the unique point 
\[ \{x_w\} := \bigcap_{n\in\N} Q(\tau_{w(n)}) = \bigcap_{n\in\N}\mathcal{Q}_{w(n)}.\]

Define a map $F : (\mathcal{A},d_{\mathscr{C},\D}) \to \mathcal{Q}_{\varepsilon} \subset\R^{\lfloor s \rfloor + 1}$ by $F([w]) =x_w$. 

\begin{lem}
$F$ is well-defined, and $F(\mathcal{A}_w) = \mathcal{Q}_w$ for all $w\in A^*$.
\end{lem}
\begin{proof}
Let $[w]=[v]\in\mathcal{A}$, with $w\neq v$. By Lemma \ref{lem:arcproperties}, there is a $n\in\mathbb{N}$ and $u,u'$ adjacent in $A^n$ such that $w=uM^\infty$ and $v=u'1^\infty$ (or vice versa).

For each $n\in\N$, $Q(\tau_{uM^n})$ intersects with $Q(\tau_{u'1^n})$ on a common entrance. Denote by $p$ the unique point in $\bigcap_{n\in\N}(Q(\tau_{uM^n})\cap Q(\tau_{u'1^n}))$. Then $Q(\tau_{w(k)})$ and $Q(\tau_{v(k)})$ both contain $p$ for all $k$, and hence $F([v])=F([w])=p$. So $F$ is well-defined. 

For the second part, fix $n\in\mathbb{N}$ and $w\in A^n$. For $k\geq n$,
note that $\{\mathcal{Q}_w^{(k)}\}$ converges in Hausdorff distance to $\mathcal{Q}_w$. By construction, each point of $F([\mathcal{A}_w])$ is contained in the Hausdorff limit of the sets $\mathcal{Q}_w^{(k)}$, and hence in $\mathcal{Q}_w$. Thus, $F(\mathcal{A}_w) \subseteq \mathcal{Q}_w$. 

For the other inclusion, fix $p\in \mathcal{Q}_w$. Let $v_0=w$. For each $k\geq 1$, we inductively set $v_k\in A^{|w|+k}_{v_{k-1}} \subseteq A^{|w|+k}_w$ to be a word with $p\in \mathcal{Q}_{v_k}$. Let $v$ be the infinite word such that $v(|w|+k)=v_k$ for all $k\geq 0$. Then immediately $p=F([v])$. Therefore, $\mathcal{Q}_w\subseteq F(\mathcal{A}_w)$.
\end{proof}

It remains to show now that $F$ is $L$-bi-Lipschitz with $L$ depending only on $s$.

\begin{proof}[{Proof of Lemma \ref{lem:snowembed}}]
Fix distinct $[w],[w']\in \mathcal{A}$. Without loss of generality, assume that $w$ precedes $w'$ in lexicographic order. 
Let $\sigma$ be the unique arc in $\mathcal{A}$ whose endpoints are $[w]$ and $[w']$. Let also $w_0\in A^*$ be the longest word such that $[w],[w']\in \mathcal{A}_{w_0}$. Let also $i,j \in A$ such that $[w]\in \mathcal{A}_{w_0i}$ and $[w']\in \mathcal{A}_{w_0j}$. By maximality of $w_0$ we have that $i\neq j$. We consider the following possible two cases.

\emph{Case 1.} Suppose that $|i-j|>1$. On one hand, there exists $i' \in A$ such that $\mathcal{A}_{w_0i'} \subset \sigma$ which implies that
\[M^{-1}\D(w_0) \leq \D(w_0i') \leq \diam{\sigma}  = d_{\mathscr{C},\D}([w],[w']) \leq \D(w_0).\]
On the other hand, $F([w]) \in Q(\tau_{w_0i})$, $F([w']) \in Q(\tau_{w_0j})$ and by Lemma \ref{lem:pipelines},
\[M^{-2}\D(w_0) \leq \dist(Q(\tau_{w_0i}),Q(\tau_{w_0j})) \leq |F([w])-F([w'])| \leq \diam{Q(\tau_{w_0})} \leq (\lfloor s\rfloor +1)^{1/2}\D(w_0).\]
Therefore, $d_{\mathscr{C},\D}([w],[w']) \approx_{s} \D(w_0) \approx_s |F([w])-F([w'])|$. This completes the proof in Case 1.

\emph{Case 2.} Suppose that $|i-j|=1$. Without loss of generality, assume that $j=i+1$. Let $k$ and $l$ be the unique integers such that 
\[ \mathcal{A}_{w_0iM^k}\cup \mathcal{A}_{w_0j1^l} \subset \sigma \subset \mathcal{A}_{w_0iM^{k-1}}\cup \mathcal{A}_{w_0j1^{l-1}}.\]
Let also $i',j'\in A$ such that $[w] \in \mathcal{A}_{w_0iM^{k-1}i'}$ and $[w'] \in \mathcal{A}_{w_0j1^{l-1}j'}$. Note that $i' \neq M$ while $j' \neq 1$.
On one hand, using the $1$-bounded turning property of $(\cA, d_{\mathscr{C},\D})$ and Lemma \ref{lem:arcproperties}, we have
\begin{align*} 
\max\{\D(w_0iM^k), \D(w_0j1^l)\}  
&\leq M d_{\mathscr{C},\D}([w],[w'])\\
&\leq M \diam(\mathcal{A}_{w_0iM^{k-1}}\cup \mathcal{A}_{w_0j1^{l-1}})\\
&\leq 2M\max\{ \D(w_0iM^{k-1}) , \D(w_0j1^{l-1})\} \\
&\leq 2M^2\max\{\D(w_0iM^k), \D(w_0j1^l)\}.
\end{align*}

On the other hand, by Lemma \ref{lem:pipes2},
\begin{align*} 
|F([w])-F([w'])| &\lesssim \max\{\diam{Q(\tau_{w_0iM^{k-1}})}, \diam{ Q(\tau_{w_0j1^{l-1}})}\}\\
&\lesssim \max\{\D(w_0iM^{k-1}), \D(w_0j1^{l-1})\}\\
&\lesssim \dist( Q(\tau_{w_0iM^{k-1}i'}), Q(\tau_{w_0iM^{k-1}i'}))\\
&\leq |F([w])-F([w'])|,
\end{align*}
with implied constants depending on the parameter $s$.

Therefore,
\[ |F([w])-F([w'])| \approx_s \max\{\D(w_0iM^{k-1}), \D(w_0j1^{l-1})\} \approx_{s} d_{\mathscr{C},\D}([w],[w']).\]
This completes the proof in Case 2 and the proof of the lemma.
\end{proof}

\subsection{Proof of Theorem \ref{thm:mainembed2}}\label{sec:Seo-method}

Here we prove Theorem \ref{thm:mainembed} using two bi-Lipschitz embedding results of Lang and Plaut \cite{LP} and of Seo \cite{Seo}. The first result says that one can ``glue" two bi-Lipschitz embeddings into a single embedding.

\begin{thm}[{\cite[Theorem 3.2]{LP}}]\label{thm:welding}
Let $X$ be a metric space and let $X_1, X_2 \subset X$ be closed subsets such that $X = X_1 \cup X_2$. If $X_1$ $L_1$-bi-Lipschitz embeds in $\R^{n_1}$ and $X_2$ $L_2$-bi-Lipschitz embeds in $\R^{n_2}$, then $X$ $L$-bi-Lipschitz embeds in $\R^{n_1+n_2+1}$ with $L$ depending on $L_1$, $L_2$, $n_1$ and $n_2$.
\end{thm}

Using Theorem \ref{thm:welding} we show that balls of $X$ that are appropriately far from $\LL(X)$ admit a bi-Lipschitz embedding into some Euclidean space quantitatively.

\begin{lem}\label{lem:Whitneycube}
Let $X$ be a doubling, bounded turning tree. For every $0<\beta<1$, there exist $L$ and $N$ depending only on the doubling constant of $X$, the bounded turning constant of $X$, and $\beta$ such that if $B(x,r)$ is a ball with $x\in X\setminus \LL(X)$ and $ r < \beta\dist(x,\LL(X))$, then $B(x,r)$ admits an $L$-bi-Lipschitz embedding into $\R^{N}$.
\end{lem}

\begin{proof}
Fix $0<\beta<1$. Let $B=\overline{B}(x,r)$ be a ball with $x\in X\setminus\LL(X)$ and $ r < \beta\dist(x,\LL(X))$. Let $D$ denote the doubling constant of $X$ and $H$ the bounded turning constant. We will argue that $B$ is contained in a union of at most $K=K(\beta, D,H)$ quasi-arcs. By  Proposition \ref{prop:qcircles} and Theorem \ref{thm:welding}, the latter implies that $B$ admits an $L$-bi-Lipschitz embedding into $\R^{N}$ with $N$ and $L$ depending only on $K$ and $D$, hence only on $\beta$, $D$ and $H$.

Let $\Gamma$ be the collection of all arcs in $X$ that join $x$ to a leaf of $X$. For each $\gamma\in \Gamma$, parametrize it by a continuous $\gamma\colon [0,1]\rightarrow X$ such that $\gamma(0)=x$ and $\gamma(1)\in\LL(X)$. Let $x_\gamma=\gamma(t_\gamma)$, where
$$ t_\gamma = \sup \{t\in [0,1] : \gamma(t)\in B\}.$$
In other words, $x_\gamma$ is the ``last'' point on $\gamma$ contained in $B$. Similarly, let $y_\gamma$ denote the last point on $\gamma$ contained in $\overline{B}(x,r/\beta)$. Note that $B$ and $\overline{B}(x,r/\beta)$ are disjoint from $\LL(X)$ by assumption, so the points $x_\gamma$ and $y_\gamma$ must exist for each $\gamma\in\Gamma$. 

Two properties of these points are clear:
\begin{enumerate}
\item If $x_\gamma \neq x_{\gamma'}$, then $y_\gamma \neq y_{\gamma'}$. In particular, 
\begin{equation}\label{eq:card}
\text{card}{\{x_{\g} : \g \in \G\}} \leq \text{card}{\{y_{\g} : \g \in \G\}}.
\end{equation}
\item We have $d(x_\gamma, x)=r$ and $d(y_\gamma,x)=r/\beta$ for each $\gamma\in \Gamma$.
\end{enumerate}

Finally, let $\Gamma_0$ be the collection of arcs joining $x$ to $x_\gamma$, as $\gamma$ ranges in  $\Gamma$. We will show that $\Gamma_0$ contains a controlled finite number of distinct elements, by showing that the collection $\{x_\gamma : \gamma\in \Gamma\}$ contains a controlled number of distinct elements. Since $B$ is contained in the union of all arcs of $\Gamma_0$, this will complete the proof.

Suppose $\gamma, \gamma'\in \Gamma$ have $x_\gamma \neq x_{\gamma'}$. We then claim that 
$$ d(y_\gamma, y_{\gamma'}) \geq \eta r,$$
for some constant $\eta$ depending only on $D$ and $H$. 

Indeed, the arc $[y_\gamma, y_{\gamma'}]$ must contain $x_\gamma$, and hence its diameter is at least
$$ d(y_\gamma, x_\gamma) \geq \left(\frac{1}{\beta} - 1\right)r,$$
and so
$$ d(y_\gamma, y_{\gamma'}) \geq \frac{1}{H}\diam([y_\gamma, y_{\gamma'}]) \geq \frac{1}{H} \left(\frac{1}{\beta} - 1\right)r = \eta r.$$

The total number of different arcs in $\Gamma_0$ is controlled by the total number of distinct $x_\gamma$, which is controlled by $\text{card}\{y_\gamma: \gamma\in\Gamma\}$ by \eqref{eq:card}. The points $y_\gamma$ form an $\eta r$-separated set in $\overline{B}(x,r/\beta)$, and so the cardinality of this set is bounded by a constant $K$ depending only on $\eta$, $\beta$, and the doubling constant $D$.
\end{proof}

The second bi-Lipschitz embedding result that we need is Seo's general bi-Lipschitz embeddability criterion \cite{Seo}. In fact, we use a simplified version of Seo's result presented by Romney in \cite[Theorem 2.2]{Romney}. Before stating the result we recall a generalized notion of Whitney decomposition for metric measure spaces due to Christ \cite{Christ} and Seo \cite{Seo}.

\begin{definition}[{\cite{Christ,Seo, Romney}}]\label{def:ChristWhitney}
Let $(X,d,\mu)$ be a metric measure space and let $\Omega$ be an open proper subset of $X$. A collection $\mathscr{Q}$ of open subsets of $\Omega$ is a \emph{Christ-Whitney decomposition} of $\Omega$ if there exist constants $\d\in(0,1)$, $C_1>c_0>0$, and $a\geq 4$ such that the following properties are satisfied:
\begin{enumerate}
\item $\bigcup_{Q\in\mathscr{Q}} Q$ is dense in $\Omega$.
\item For every $Q,Q' \in \mathscr{Q}$ with $Q\neq Q'$ we have $Q\cap Q' = \emptyset$.
\item For every $Q\in \mathscr{Q}$, there exists $x\in \Omega$ and $k\in\mathbb{Z}$ such that
\[ B(x,c_0 \d^k) \subset Q \subset B(x,C_1\d^k)\]
and
\[ (a-2)C_1\delta^k \leq \dist(Q,X\setminus\Omega) \leq \left(\frac{a C_1}{\delta}\right)\delta^k.\]
\end{enumerate}
\end{definition}

\begin{lem}[{\cite[Theorem 11]{Christ}, \cite[Lemma 2.1]{Seo}, \cite[Lemma 2.5]{Romney}}]\label{lem:CWdecomp}
Let $X$ be a doubling metric space and $Y$ be a nonempty closed proper subset of $X$. Then $X\setminus Y$ has a Christ-Whitney decomposition, with constants $\delta, c_0, C_1, a$ absolute.
\end{lem}

\begin{thm}[{\cite[Theorem 1.1]{Seo}, \cite[Theorem 2.2]{Romney}}]\label{thm:seo}
Let $X$ be a complete metric measure space. Then $X$ admits an $L$-bi-Lipschitz embedding into some Euclidean space $\R^M$ if and only if the following conditions hold for some constants $L_1, L_2, M_1, M_2$:
\begin{enumerate}
\item\label{eq:seo1} $X$ is doubling.
\item\label{eq:seo2} There is a non-empty closed subset of $Y\subseteq X$ which admits an $L_1$-bi-Lipschitz embedding into some $\R^{M_1}$.
\item\label{eq:seo3} There is a Christ-Whitney decomposition of $X\setminus Y$ such that each cube admits an $L_2$-bi-Lipschitz embedding into some $\R^{M_2}$. 
\end{enumerate}
The distortion $L$ and target dimension $M$ of the embedding of $X$ depend only on the doubling constant of $\mu$, $M_1$, $M_2$, and $L_1$, $L_2$.
\end{thm}

\begin{proof}[{Proof of Theorem \ref{thm:mainembed2}}]
It suffices to show that $X$ satisfies the conditions of Theorem \ref{thm:seo} with $Y = \overline{\LL(X)}$. The doubling property \eqref{eq:seo1} in Theorem \ref{thm:seo} is satisfied by assumption. We assume that $\LL(X)$, hence $Y$, admits a bi-Lipschitz embedding into some $\R^{M_1}$, so \eqref{eq:seo2} is assumed to hold in Theorem \ref{thm:main}. It remains to prove \eqref{eq:seo3}. 

By Lemma \ref{lem:CWdecomp} there exists a Christ-Whitney decomposition $\mathscr{Q}$ for some constants $\d\in(0,1)$, $C_1>c_0>0$, and $a\geq 4$. Let $Q\in\mathscr{Q}$ be an arbitrary cube of this decomposition.

The doubling property of $X$ implies that there exists $N\in\N$, depending only on the doubling constant of $X$ and the constants of the Christ-Whitney decomposition, and at most $N$ balls $B_1,\dots,B_n$ with centers on $Q$ and of radius $\frac{1}{3}\dist(Q,Y)$, such that $Q\subset B_1\cup\cdots\cup B_n$. In particular, the balls $B_i$ each satisfy the assumptions of Lemma \ref{lem:Whitneycube} with $\beta=\frac{1}{2}$.

Thus, by Lemma \ref{lem:Whitneycube}, each $B_i$ admits an $L'$-bi-Lipschitz embedding into $\R^{M'}$, where $L'$ and $M'$ depend only on the doubling and bounded turning constants of $X$. By Theorem \ref{thm:welding}, $Q\subseteq B_1\cup\cdots\cup B_n$ admits an $L_2$-bi-Lipschitz embedding into $\R^{M_2}$, where $L_2$ and $M_2$ depend only on the doubling and bounded turning constants of $X$. This verifies condition \eqref{eq:seo3} of Theorem \ref{thm:seo} and completes the proof of Theorem \ref{thm:mainembed2}.
\end{proof}

\bibliography{quasitrees-ref}
\bibliographystyle{amsbeta}

\end{document}